\documentclass[10pt,reqno,oneside]{amsproc}

\allowdisplaybreaks

\usepackage{amsmath, amsthm, amssymb, enumerate, enumitem, bbm,multicol}
\usepackage{color}

\chardef\forshowkeys=0
\chardef\refcheck=0
\chardef\showllabel=0
\chardef\sketches=0

\ifnum\forshowkeys=1
\definecolor{mygray}{rgb}{.6, .6, .6}

\usepackage[color,notcite,notref]{showkeys}
\usepackage{marginnote}
\fi

\usepackage{times}
\usepackage{graphicx}
\usepackage[usenames,dvipsnames,svgnames,table]{xcolor}
\usepackage[colorlinks=true, pdfstartview=FitV, linkcolor=blue, citecolor=blue, urlcolor=blue]{hyperref}

\usepackage{tikz}

\usepackage{color}
\usepackage{xcolor}

\usepackage{datetime}
\usepackage{fancyhdr,lastpage}

\ifnum\refcheck=1
 \usepackage{refcheck}    
\fi

\chardef\coloryes=0 
\chardef\isitdraft=0 

\ifnum\isitdraft=1 
\def\eqref#1{({\ref{#1}})}                

\definecolor{refkey}{rgb}{.3,0.3,0.3}

\def\nnewpage{} 
\else

\def\startnewsection#1#2{\section{#1}\label{#2}\setcounter{equation}{0}}   
\def\nnewpage{} 

\fi

\begin{document}

\def\ques{{\colr \underline{??????}\colb}}
\def\nto#1{{\colC \footnote{\em \colC #1}}}
\def\fractext#1#2{{#1}/{#2}}
\def\fracsm#1#2{{\textstyle{\frac{#1}{#2}}}}   
\def\nnonumber{}
\def\les{\lesssim}

\def\colr{{}}
\def\colg{{}}
\def\colb{{}}
\def\colu{{}}
\def\cole{{}}
\def\colA{{}}
\def\colB{{}}
\def\colC{{}}
\def\colD{{}}
\def\colE{{}}
\def\colF{{}}

\ifnum\coloryes=1

\definecolor{coloraaaa}{rgb}{0.1,0.2,0.8}
\definecolor{colorbbbb}{rgb}{0.1,0.7,0.1}
\definecolor{colorcccc}{rgb}{0.8,0.3,0.9}
\definecolor{colordddd}{rgb}{0.0,.5,0.0}
\definecolor{coloreeee}{rgb}{0.8,0.3,0.9}
\definecolor{colorffff}{rgb}{0.8,0.9,0.9}
\definecolor{colorgggg}{rgb}{0.5,0.0,0.4}
\definecolor{coloroooo}{rgb}{0.45,0.0,0}

\def\colb{\color{black}}

\def\colr{\color{red}}
\def\cole{\color{coloroooo}}

\def\colu{\color{blue}}
\def\colg{\color{colordddd}}
\def\colgray{\color{colorffff}}

\def\colA{\color{coloraaaa}}
\def\colB{\color{colorbbbb}}
\def\colC{\color{colorcccc}}
\def\colD{\color{colordddd}}
\def\colE{\color{coloreeee}}
\def\colF{\color{colorffff}}
\def\colG{\color{colorgggg}}

\def\cole{}

\fi
\ifnum\isitdraft=1
\chardef\coloryes=1 
\baselineskip=17.6pt
\pagestyle{myheadings}
\def\const{\mathop{\rm const}\nolimits}  
\def\diam{\mathop{\rm diam}\nolimits}    
\def\dist{\mathop{\rm dist}\nolimits}    
\def\rref#1{{\ref{#1}{\rm \tiny \fbox{\tiny #1}}}}
\def\theequation{\fbox{\bf \thesection.\arabic{equation}}}
\def\startnewsection#1#2{\colg \section{#1}\colb\label{#2}
\setcounter{equation}{0}
\pagestyle{fancy}
\lhead{\colb Section~\ref{#2}, #1 }
\cfoot{}
\rfoot{\thepage\ of \pageref{LastPage}}

\chead{}
\rhead{\thepage}
\def\nnewpage{\newpage}
\newcounter{startcurrpage}
\newcounter{currpage}
\def\llll#1{{\rm\tiny\fbox{#1}}}
\def\blackdot{{\color{red}{\hskip-.0truecm\rule[-1mm]{4mm}{4mm}\hskip.2truecm}}\hskip-.3truecm}
\def\bluedot{{\colC {\hskip-.0truecm\rule[-1mm]{4mm}{4mm}\hskip.2truecm}}\hskip-.3truecm}
\def\purpledot{{\colA{\rule[0mm]{4mm}{4mm}}\colb}}
\def\pdot{\purpledot}
\else
\baselineskip=12.8pt
\def\blackdot{{\color{red}{\hskip-.0truecm\rule[-1mm]{4mm}{4mm}\hskip.2truecm}}\hskip-.3truecm}
\def\purpledot{{\rule[-3mm]{8mm}{8mm}}}
\def\pdot{}
\fi

\def\textand{\qquad \text{and}\qquad}
\def\pp{p}
\def\qq{{\tilde p}}
\def\KK{K}
\def\MM{M}
\def\ema#1{{#1}}
\def\emb#1{#1}

\ifnum\isitdraft=1
\def\llabel#1{\nonumber}
\else
\def\llabel#1{\nonumber}
\def\llabel#1{\label{#1}}
\fi

\def\tepsilon{\tilde\epsilon}
\def\epsilonz{\epsilon_0}
\def\restr{\bigm|}
\def\into{\int_{\Omega}}
\def\intu{\int_{\Gamma_1}}
\def\intl{\int_{\Gamma_0}}
\def\tpar{\tilde\partial}
\def\bpar{\,|\nabla_2|}
\def\barpar{\bar\partial}
\def\FF{F}
\def\gdot{{\color{green}{\hskip-.0truecm\rule[-1mm]{4mm}{4mm}\hskip.2truecm}}\hskip-.3truecm}
\def\bdot{{\color{blue}{\hskip-.0truecm\rule[-1mm]{4mm}{4mm}\hskip.2truecm}}\hskip-.3truecm}
\def\cydot{{\color{cyan} {\hskip-.0truecm\rule[-1mm]{4mm}{4mm}\hskip.2truecm}}\hskip-.3truecm}
\def\rdot{{\color{red} {\hskip-.0truecm\rule[-1mm]{4mm}{4mm}\hskip.2truecm}}\hskip-.3truecm}

\def\tdot{\fbox{\fbox{\bf\color{blue}\tiny I'm here; \today \ \currenttime}}}
\def\nts#1{{\color{red}\hbox{\bf ~#1~}}} 

\def\ntsr#1{\vskip.0truecm{\color{red}\hbox{\bf ~#1~}}\vskip0truecm} 

\def\ntsf#1{\footnote{\hbox{\bf ~#1~}}} 
\def\ntsf#1{\footnote{\color{red}\hbox{\bf ~#1~}}} 
\def\bigline#1{~\\\hskip2truecm~~~~{#1}{#1}{#1}{#1}{#1}{#1}{#1}{#1}{#1}{#1}{#1}{#1}{#1}{#1}{#1}{#1}{#1}{#1}{#1}{#1}{#1}\\}
\def\biglineb{\bigline{$\downarrow\,$ $\downarrow\,$}}
\def\biglinem{\bigline{---}}
\def\biglinee{\bigline{$\uparrow\,$ $\uparrow\,$}}
\def\ceil#1{\lceil #1 \rceil}
\def\gdot{{\color{green}{\hskip-.0truecm\rule[-1mm]{4mm}{4mm}\hskip.2truecm}}\hskip-.3truecm}
\def\bluedot{{\color{blue} {\hskip-.0truecm\rule[-1mm]{4mm}{4mm}\hskip.2truecm}}\hskip-.3truecm}
\def\rdot{{\color{red} {\hskip-.0truecm\rule[-1mm]{4mm}{4mm}\hskip.2truecm}}\hskip-.3truecm}
\def\dbar{\bar{\partial}}
\newtheorem{Theorem}{Theorem}[section]
\newtheorem{Corollary}[Theorem]{Corollary}
\newtheorem{Proposition}[Theorem]{Proposition}
\newtheorem{Lemma}[Theorem]{Lemma}
\newtheorem{Remark}[Theorem]{Remark}
\newtheorem{definition}{Definition}[section]
\def\theequation{\thesection.\arabic{equation}}
\def\cmi#1{{\color{red}IK: #1}}
\def\cmj#1{{\color{red}IK: #1}}
\def\cml{\rm \colr Linfeng:~} 
\def\TT{\mathbf{T}}
\def\XX{\mathbf{X}}

\def\sqrtg{\sqrt{g}}
\def\DD{{\mathcal D}}
\def\OO{\tilde\Omega}
\def\EE{{\mathcal E}}
\def\lot{{\rm l.o.t.}}                       
\def\endproof{\hfill$\Box$\\}
\def\square{\hfill$\Box$\\}
\def\inon#1{\ \ \ \ \text{~~~~~~#1}}                
\def\comma{ {\rm ,\qquad{}} }            
\def\commaone{ {\rm ,\qquad{}} }         
\def\dist{\mathop{\rm dist}\nolimits}    
\def\ad{\mathop{\rm ad}\nolimits}    
\def\sgn{\mathop{\rm sgn\,}\nolimits}   
\def\Tr{\mathop{\rm Tr}\nolimits}    
\def\dive{\mathop{\rm div}\nolimits}    
\def\grad{\mathop{\rm grad}\nolimits}    
\def\curl{\mathop{\rm curl}\nolimits}    
\def\det{\mathop{\rm det}\nolimits}    
\def\supp{\mathop{\rm supp}\nolimits}    
\def\re{\mathop{\rm {\mathbb R}e}\nolimits}    
\def\wb{\bar{\omega}}
\def\Wb{\bar{W}}
\def\indeq{\quad{}}                     
\def\indeqtimes{\indeq\indeq\indeq\indeq\times} 
\def\period{.}                           
\def\semicolon{\,;}                      
\newcommand{\cD}{\mathcal{D}}
\newcommand{\cH}{\mathcal{H}}
\newcommand{\imp}{\Rightarrow}
\newcommand{\tr}{\operatorname{tr}}
\newcommand{\vol}{\operatorname{vol}}
\newcommand{\id}{\operatorname{id}}
\newcommand{\p}{\parallel}
\newcommand{\norm}[1]{\Vert#1\Vert}
\newcommand{\abs}[1]{\vert#1\vert}
\newcommand{\nnorm}[1]{\left\Vert#1\right\Vert}
\newcommand{\aabs}[1]{\left\vert#1\right\vert}

\ifnum\showllabel=1
\def\llabel#1{\label{#1}}
\else
\def\llabel#1{\notag}
\fi

\title{Mach limits in analytic spaces on exterior domains}\par \author[J.~Jang]{Juhi Jang}\address{Department of Mathematics\\ University of Southern California\\ Los Angeles, CA 90089} \email{juhijang@usc.edu} \par \author[I.~Kukavica]{Igor Kukavica} \address{Department of Mathematics\\ University of Southern California\\ Los Angeles, CA 90089} \email{kukavica@usc.edu} \par \author[L.~Li]{Linfeng Li} \address{Department of Mathematics\\ University of Southern California\\ Los Angeles, CA 90089} \email{lli265@usc.edu} \par \begin{abstract} We address the Mach limit problem for the Euler equations in an exterior domain with analytic boundary. We first prove the existence of tangential analytic vector fields for the exterior domain with constant analyticity radii, and introduce an analytic norm in which we distinguish derivatives taken from different directions. Then we prove the uniform boundedness of the solutions in the analytic space on a time interval independent of the Mach number, and Mach limit holds in the analytic norm. The results extends more generally to Gevrey initial data with convergence in a Gevrey norm. \end{abstract} \par \maketitle \tableofcontents \par \startnewsection{Introduction}{sec01}  The incompressible limit concerns the passage from compressible fluids to incompressible fluids as the Mach number tends to zero.  The first rigorous result on this singular limit problem can be traced back to Klainerman and Majda \cite{KM81} (see also Ebin \cite{Ebin77}), with a great deal of activity and progress in recent decades \cite{A05, A08, A06,D1,D2,DG,DM,F,FN,H,LM,M, Asa87, FKM, Igu97, Iso1, Iso2, Iso3, KM82, MS01, Sch86, Sch05, Uka}. A general approach to this problem is to first prove the existence of solutions on some time interval independent of the Mach number  and then to show the convergence to solutions of the limiting equations when the Mach number tends to zero.  A key ingredient in establishing a uniform time interval is a uniform upper bound, while in showing the convergence, the most critical issue is the vanishing of the acoustic waves.  It is well-known that the analysis depends on various settings such as isentropic vs.~non-isentropic, inviscid vs.~viscous, well-prepared data vs.~general data, or the whole space vs.~domain with a boundary.  \par For the non-isentropic problem with general data, the non-isentropic Euler flows feature intriguing wave-transport structure, and the coefficients of governing equations (for instance, see $E$ in \eqref{SDHNRTYERTDGHFHDFGSDFGRTYFGHDFGDFSFGWERTFSDFGDFADWERTFGHDFGHFGH01} and \eqref{SDHNRTYERTDGHFHDFGSDFGRTYFGHDFGDFSFGWERTFSDFGDFADWERTFGHDFGHFGH02}) depend on the dependent variables, which makes the low Mach number limit a difficult problem. M{\'e}tivier and Schochet in \cite{MS01} gave the first satisfactory answer by making use of the microlocal defect measure, and Alazard in \cite{A05} extended the existence result to the case of domains with boundary and the convergence result for exterior domains. Both results were obtained in Sobolev spaces. In a recent work \cite{JKL}, we studied the non-isentropic problem with general analytic or Gevrey initial data in $\mathbb{R}^3$ and proved that the convergence holds in these strong norms.  \par The goal of this paper is to extend our previous results on analytic/Gevrey convergence in the zero Mach limit from \cite{JKL} to the more difficult case of the domain with boundary.  The natural setting for the Mach limit is the case of an exterior domain since there the convergence was established in Sobolev spaces.  One of the main difficulties in our approach is the existence of a complete family of analytic vector fields which respect to the boundary condition $v^\epsilon \cdot \nu |_{\partial \Omega}= 0$.  In Theorem~\ref{P01} below, we construct a family of tangential analytic vector fields for the exterior domain with constant analyticity radii.  The existence of a family of complete tangential analytic vector fields was established by Komatsu in \cite{Ko01} for the case of a bounded domain. The completeness refers to the fact that the vector fields span the tangential space in the closure of the domain.  In our proof of analytic boundedness, we require for the tangential fields to span only in a neighborhood of the boundary.  In order to construct the necessary tangential fields, we use interior analytic hypoellipticity for the Dirichlet problem (cf.~\cite{LM01}) and analytic regularization by the heat kernel.  As an additional benefit, we obtain a low number of tangential vector fields (three) needed for the construction.  In order to obtain uniform boundedness in analytic spaces, we first establish boundedness of the entropy, divergence, followed by then normal and tangential derivative reduction schemes for the velocity. The normal and tangential derivatives can be reduced by using elliptic regularity, which leads to the estimates of divergence component, curl component, pure time derivatives, commutators, and lower-order terms associated to the boundary.  In order to obtain that the Mach limit holds, we need to change the approach from \cite[Section~7]{JKL} since the interpolation inequality used there results in boundary terms which can not be handled.  Instead, we use a simpler interpolation inequality \eqref{SDHNRTYERTDGHFHDFGSDFGRTYFGHDFGDFSFGWERTFSDFGDFADWERTFGHDFGHFGH228} and a discrete dominated convergence theorem. As a byproduct of the Mach convergence, we obtain the analyticity of solutions of the stratified incompressible Euler equation in an exterior (or bounded) domain. For the unstratified version of the Euler equations, the analyticity was proven in \cite{KV2} using different methods. \par The paper is organized as follows.  In Section~\ref{sec01a}, we recall from \cite{MS01,A05} symmetrization of the compressible Euler equations.  In Section~\ref{sec02}, we construct the tangential vector fields for the exterior domain and state the main results.  The a~priori estimate needed for the uniform analytic boundedness, stated in Lemma~\ref{L12}, is given at the end of Section~\ref{sec06} and relies on the bounds on entropy and the velocity given in previous sections.  The proof of the second main theorem is given in Section~\ref{sec08}.  We emphasize that the construction of the system of analytic tangential fields also applies to the case of a bounded domain; only when proving the Mach limit convergence, we rely on the fact that the domain is external.  All considerations also apply in the Gevrey norm, or more generally to the spaces used in the book by Lions and Magenes~\cite{LM01}. The approach in this paper benefits from ideas in~\cite{CKV}. For other approaches to analyticity, we refer to \cite{B, BB, Bi, BF, BGK, BoGK, DE, DL, FT, G, GK, KV1, KP, LO, OT}. \par \startnewsection{The setting and notation }{sec01a}  \par We address the incompressible limit for classical solutions of the compressible Euler equations for non-isentropic fluids in an exterior domain $\Omega \subseteq \mathbb{R}^3$ with analytic boundary $\partial \Omega$. More specifically, we consider the compressible Euler equations for an inviscid, non-isentropic fluid \begin{align} &\partial_t \rho + v\cdot \nabla \rho + \rho \nabla\cdot v =0, \label{Euler1}  \\& \rho\left( \partial_t v +  v\cdot \nabla v\right) + \nabla P =0, \label{Euler2} \\& \partial_t S + v\cdot \nabla S =0, \label{Euler3} \end{align} where $\rho\colon \Omega\times[0,T)\to \mathbb R_+$ is the density,  $v\colon\Omega\times[0,T)\to \mathbb R^3$ is the velocity,  $P\colon\Omega\times[0,T)\to \mathbb R_+$ is the pressure, and  $S\colon\Omega\times[0,T)\to \mathbb R$ is the entropy of the fluid.  To close the system \eqref{Euler1}--\eqref{Euler3}, we assume the equation of state \begin{equation} P=P(\rho, S) . \llabel{8ThswELzXU3X7Ebd1KdZ7v1rN3GiirRXGKWK099ovBM0FDJCvkopYNQ2aN94Z7k0UnUKamE3OjU8DFYFFokbSI2J9V9gVlM8ALWThDPnPu3EL7HPD2VDaZTggzcCCmbvc70qqPcC9mt60ogcrTiA3HEjwTK8ymKeuJMc4q6dVz200XnYUtLR9GYjPXvFOVr6W1zUK1WbPToaWJJuKnxBLnd0ftDEbMmj4loHYyhZyMjM91zQS4p7z8eKa9h0JrbacekcirexG0z4n3xz0QOWSvFj3jLhWXUIU21iIAwJtI3RbWa90I7rzAIqI3UElUJG7tLtUXzw4KQNETvXzqWaujEMenYlNIzLGxgB3AuJ86VS6RcPJ8OXWw8imtcKZEzHop84G1gSAs0PCowMI2fLKTdD60ynHg7lkNFjJLqOoQvfkfZBNG3o1DgCn9hyUh5VSP5z61qvQwceUdVJJsBvXDG4ELHQHIaPTbMTrsLsmtXGyOB7p2Os43USbq5ik4Lin769OTkUxmpI8uGYnfBKbYI9AQzCFw3h0geJftZZKU74rYleajmkmZJdiTGHOOaSt1NnlB7Y7h0yoWJryrVrTzHO82S7oubQAWx9dz2XYWBe5Kf3ALsUFvqgtM2O2IdimrjZ7RN284KGYtrVaWW4nTZXVbRVoQ77hVLX6K2kqFWFmaZnsF9Chp8KxrscSGPiStVXBJ3xZcD5IP4Fu9LcdTR2VwbcLDlGK1ro3EEyqEAzw6sKeEg2sFfjzMtrZ9kbdxNw66cxftlzDGZhxQAWQKkSXjqmmrEpSDHNRTYERTDGHFHDFGSDFGRTYFGHDFGDFSFGWERTFSDFGDFADWERTFGHDFGHFGH138} \end{equation} For instance, in the case of ideal gas the equation of state reads \begin{equation} P(\rho, S)=\rho^\gamma e^{{S}}   , \llabel{NuG6Pyloq8hHlSfMaLXm5RzEXW4Y1Bqib3UOhYw95h6f6o8kw6frZwg6fIyXPnae1TQJMt2TTfWWfjJrXilpYGrUlQ4uM7Dsp0rVg3gIEmQOzTFh9LAKO8csQu6mh25r8WqRIDZWgSYkWDulL8GptZW10GdSYFUXLzyQZhVZMn9amP9aEWzkau06dZghMym3RjfdePGln8s7xHYCIV9HwKa6vEjH5J8Ipr7NkCxWR84TWnqs0fsiPqGgsId1fs53AT71qRIczPX77Si23GirL9MQZ4FpigdruNYth1K4MZilvrRk6B4W5B8Id3Xq9nhxEN4P6ipZla2UQQx8mdag7rVD3zdDrhBvkLDJotKyV5IrmyJR5etxS1cvEsYxGzj2TrfSRmyZo4Lm5DmqNiZdacgGQ0KRwQKGXg9o8v8wmBfUutCOcKczzkx4UfhuAa8pYzWVq9Sp6CmAcZLMxceBXDwugsjWuiiGlvJDb08hBOVC1pni64TTqOpzezqZBJy5oKS8BhHsdnKkHgnZlUCm7j0IvYjQE7JN9fdEDddys3y1x52pbiGLca71jG3euliCeuzv2R40Q50JZUBuKdU3mMay0uoS7ulWDh7qG2FKw2TJXzBES2JkQ4UDy4aJ2IXs4RNH41spyTGNhhk0w5ZC8B3nUBp9p8eLKh8UO4fMqY6wlcAGMxCHtvlOxMqAJoQQU1e8a2aX9Y62rlIS6dejKY3KCUm257oClVeEe8p1zUJSvbmLdFy7ObQFNlJ6FRdFkEmqMN0FdNZJ08DYuq2pLXJNz4rOZkZXSDHNRTYERTDGHFHDFGSDFGRTYFGHDFGDFSFGWERTFSDFGDFADWERTFGHDFGHFGH15} \end{equation} where $\gamma>1$ is the adiabatic exponent.  After some rescalings and a change of variables (cf.~\cite{MS01, A05}), we consider the symmetrized version of the compressible Euler equations for non-isentropic fluids \begin{align} & E(S,\epsilon u) (\partial_t u + v \cdot \nabla u) + \frac{1}{\epsilon} L(\partial_x) u  =  0, \label{SDHNRTYERTDGHFHDFGSDFGRTYFGHDFGDFSFGWERTFSDFGDFADWERTFGHDFGHFGH01} \\& \partial_t S + v\cdot \nabla S  =  0 , \label{SDHNRTYERTDGHFHDFGSDFGRTYFGHDFGDFSFGWERTFSDFGDFADWERTFGHDFGHFGH02} \end{align} where $u=(p,v)^{T}$ and \begin{align} \begin{split} E(S, \epsilon u)  =  \begin{pmatrix} a(S, \epsilon u)  & 0 \\ 0 & r(S, \epsilon u)I_3 \\ \end{pmatrix}, \qquad L(\partial_x) = \begin{pmatrix} 0 & \dive \\ \nabla & 0\\ \end{pmatrix} . \llabel{2IjTD1fVtz4BmFIPi0GKDR2WPhOzHzTLPlbAEOT9XW0gbTLb3XRQqGG8o4TPE6WRcuMqMXhs6xOfv8stjDiu8rtJtTKSKjlGkGwt8nFDxjA9fCmiuFqMWjeox5Akw3wSd81vK8c4C0OdjCHIseHUOhyqGx3KwOlDql1Y4NY4IvI7XDE4cFeXdFVbCFHaJsb4OC0huMj65J4favgGo7qY5XtLyizYDvHTRzd9xSRVg0Pl6Z89XzfLhGlHIYBx9OELo5loZx4wag4cnFaCEKfA0uzfwHMUVM9QyeARFe3Py6kQGGFxrPf6TZBQRla1a6AekerXgkblznSmmhYjcz3ioWYjzh33sxRJMkDosEAAhUOOzaQfKZ0cn5kqYPnW71vCT69aEC9LDEQ5SBK4JfVFLAoQpNdzZHAlJaLMnvRqH7pBBqOr7fvoaeBSA8TEbtxy3jwK3v244dlfwRLDcgX14vTpWd8zyYWjweQmFyD5y5lDNlZbAJaccldkxYn3VQYIVv6fwmHz19w3yD4YezRM9BduEL7D92wTHHcDogZxZWRWJxipvfz48ZVB7FZtgK0Y1woCohLAi70NOTa06u2sYGlmspVl2xy0XB37x43k5kaoZdeyEsDglRFXi96b6w9BdIdKogSUMNLLbCRzeQLUZmi9O2qvVzDhzv1r6spSljwNhG6s6iSdXhobhbp2usEdl95LPAtrBBibPCwShpFCCUayzxYS578rof3UwDPsCIpESHB1qFPSW5tt0I7ozjXun6cz4cQLBJ4MNmI6F08S2Il8C0JQYiUlSDHNRTYERTDGHFHDFGSDFGRTYFGHDFGDFSFGWERTFSDFGDFADWERTFGHDFGHFGH243} \end{split} \end{align} The parameter $\epsilon>0$ represents the Mach number. In agreement with the equation of state of the ideal gas, we adopt the assumption \begin{align} a(S, \epsilon u) = f_1(S) g_1(\epsilon u)    \llabel{I1YkKoiubVtfGuOegSllvb4HGn3bSZLlXefaeN6v1B6m3Ek3JSXUIjX8PdNKIUFNJvPHaVr4TeARPdXEV7BxM0A7w7jep8M4QahOihEVoPxbi1VuGetOtHbPtsO5r363Rez9nA5EJ55pcLlQQHg6X1JEWK8Cf9kZm14A5lirN7kKZrY0K10IteJd3kMGwopVnfYEG2orGfj0TTAXtecJKeTM0x1N9f0lRpQkPM373r0iA6EFs1F6f4mjOB5zu5GGTNclBmkb5jOOK4ynyMy04oz6m6AkzNnPJXhBnPHRuN5LyqSguz5NnW2lUYx3fX4huLieHL30wg93Xwcgj1I9dO9bEPCR0vc6A005QVFy1lyK7oVRVpbJzZnxYdcldXgQaDXY3gzx368ORJFK9UhXTe3xYbVHGoYqdHgVyf5kKQzmmK49xxiApjVkwgzJOdE4vghAv9bVIHewcVqcbSUcF1pHzolNjTl1BurcSamIPzkUS8wwSa7wVWR4DLVGf1RFr599HtyGqhDT0TDlooamgj9ampngaWenGXU2TzXLhIYOW5v2dArCGsLks53pWAuAyDQlF6spKydHT9Z1Xn2sU1g0DLlaoYuLPPB6YKoD1M0fiqHUl4AIajoiVQ6afVT6wvYMd0pCYBZp7RXHdxTb0sjJ0Beqpkc8bNOgZ0Tr0wqh1C2HnYQXM8nJ0PfuGJBe2vuqDukLVAJwv2tYcJOM1uKh7pcgoiiKt0b3eURecDVM7ivRMh1T6pAWlupjkEjULR3xNVAu5kEbnrVHE1OrJ2bxdUPyDvyVSDHNRTYERTDGHFHDFGSDFGRTYFGHDFGDFSFGWERTFSDFGDFADWERTFGHDFGHFGH245} \end{align} and \begin{align} r(S, \epsilon u) = f_2(S) g_2(\epsilon u)    , \llabel{ix6sCBpGDSxjBCn9PFiuxkFvw0QPofRjy2OFItVeDBtDzlc9xVyA0de9Y5h8c7dYCFkFlvWPDSuNVI6MZ72u9MBtK9BGLNsYplX2yb5UHgHADbW8XRzkvUJZShWQHGoKXyVArsHTQ1VbddK2MIxmTf6wET9cXFbuuVxCbSBBp0v2JMQ5Z8z3pMEGpTU6KCcYN2BlWdp2tmliPDHJQWjIRRgqi5lAPgiklc8ruHnvYFMAIrIh7Ths9tEhAAYgSswZZfws19P5weJvMimbsFHThCnSZHORmyt98w3U3zantzAyTwq0CjgDIEtkbh98V4uo52jjAZz1kLoC8oHGvZ5RuGwv3kK4WB50ToMtq7QWG9mtbSIlc87ruZfKwZPh31ZAOsq8ljVQJLTXCgyQn0vKESiSqBpawtHxcIJe4SiE1izzximkePY3s7SX5DASGXHqCr38VYP3HxvOIRZtMfqNoLFoU7vNdtxzwUkX32t94nFdqqTRQOvYqEbigjrSZkTN7XwtPFgNsO7M1mbDAbtVB3LGCpgE9hVFKYLcSGmF8637aZDiz4CuJbLnpE7yl85jg1MTPOLOGEPOeMru1v25XLJFzhwgElnuYmqrX1YKVKvgmMK7gI46h5kZBOoJtfC5gVvA1kNJr2o7om1XNpUwtCWXfFTSWDjsIwuxOJxLU1SxA5ObG3IOUdLqJcCArgzKM08DvX2mui13Tt71IwqoFUI0EEf5SV2vxcySYIQGrqrBHIDTJv1OB1CzDIDdW4E4jJmv6KtxoBOs9ADWBq218BJJzRyUQi2GSDHNRTYERTDGHFHDFGSDFGRTYFGHDFGDFSFGWERTFSDFGDFADWERTFGHDFGHFGH246} \end{align} where $f_1$, $f_2$, $g_1$, and $g_2$ are positive entire real-analytic functions. We impose the impermeability boundary condition \begin{align} v\cdot \nu|_{\partial \Omega}  = 0. \label{SDHNRTYERTDGHFHDFGSDFGRTYFGHDFGDFSFGWERTFSDFGDFADWERTFGHDFGHFGH03} \end{align} Due to presence of the boundary, we also require some compatibility conditions.  Since the matrix $E(S,\epsilon u)$ is invertible, we obtain \begin{align} \begin{split} \partial_t v =  -v\cdot \nabla v + \frac{1}{\epsilon r} \nabla p \label{SDHNRTYERTDGHFHDFGSDFGRTYFGHDFGDFSFGWERTFSDFGDFADWERTFGHDFGHFGH247} . \end{split} \end{align} Differentiating \eqref{SDHNRTYERTDGHFHDFGSDFGRTYFGHDFGDFSFGWERTFSDFGDFADWERTFGHDFGHFGH247} with respect to time recursively, we get \begin{align} \begin{split} \partial_t^{k} v(0) =  A_k(u(0), S(0)) , \end{split}    \llabel{pweET8LaO4ho95g4vWQmoiqjSwMA9CvnGqxl1LrYuMjGboUpuvYQ2CdBlAB97ewjc5RJESFGsORedoM0bBk25VEKB8VA9ytAEOyofG8QIj27aI3jyRmzyETKxpgUq4BvbcD1b1gKByoE3azgelVNu8iZ1w1tqtwKx8CLN28ynjdojUWvNH9qyHaXZGhjUgmuLI87iY7Q9MQWaiFFSGzt84mSQq25ONltTgbl8YDQSAzXqpJEK7bGL1UJn0f59vPrwdtd6sDLjLoo18tQXf55upmTadJDsELpH2vqYuTAmYzDg951PKFP6pEizIJQd8NgnHTND6z6ExRXV0ouUjWTkAKABeAC9Rfjac43AjkXnHdgSy3v5cBets3VXqfpPBqiGf90awg4dW9UkvRiJy46GbH3UcJ86hWVaCMjedsUcqDSZ1DlP2mfBhzu5dvu1i6eW2YNLhM3fWOdzKS6Qov14wxYYd8saS38hIlcPtS4l9B7hFC3JXJGpstll7a7WNrVMwunmnmDc5duVpZxTCl8FI01jhn5Bl4JzaEV7CKMThLji1gyZuXcIv4033NqZLITGUx3ClPCBKO3vRUimJql5blI9GrWyirWHoflH73ZTeZXkopeq8XL1RQ3aUj6Essnj20MA3AsrSVft3F9wzB1qDQVOnHCmmP3dWSbjstoj3oGjadvzqcMB6Y6kD9sZ0bdMjtUThULGTWU9Nmr3E4CNbzUOvThhqL1pxAxTezrHdVMgLYTTrSfxLUXCMrWAbE69K6XHi5re1fx4GDKkiB7f2DXzXez2k2YSDHNRTYERTDGHFHDFGSDFGRTYFGHDFGDFSFGWERTFSDFGDFADWERTFGHDFGHFGH27} \end{align} for some functions $A_k$, where $k\in \mathbb{N}_0$.  We say that the initial data satisfy the compatibility condition of all orders if \begin{align} \nu \cdot A_k(u(0), S(0))|_{\partial\Omega} = 0 \comma k\in \mathbb{N}_0 .    \llabel{cYc4QjUyMYR1oDeYNWf74hByFdsWk4cUbCRDXaq4eDWd7qbOt7GOuoklgjJ00J9IlOJxntzFVBCFtpABpVLEE2y5Qcgb35DU4igj4dzzWsoNFwvqjbNFma0amFKivAappzMzrVqYfOulMHafaBk6JreOQBaTEsJBBtHXjn2EUCNleWpcvWJIggWXKsnB3wvmoWK49Nl492ogR6fvc8ffjJmsWJr0jzI9pCBsIUVofDkKHUb7vxpuQUXA6hMUryvxEpcTqlTkzz0qHbXpO8jFuh6nwzVPPzpA8961V78cO2Waw0yGnCHVqBVjTUHlkp6dGHOdvoEE8cw7QDL1o1qg5TXqoV720hhQTyFtpTJDg9E8Dnsp1QiX98ZVQN3sduZqcn9IXozWhFd16IB0K9JeBHvi364kQlFMMJOn0OUBrnvpYyjUBOfsPzxl4zcMnJHdqOjSi6NMn8bR6kPeklTFdVlwDSrhT8Qr0sChNh88j8ZAvvWVD03wtETKKNUdr7WEK1jKSIHFKh2sr1RRVRa8JmBtkWI1ukuZTF2B4p8E7Y3p0DX20JM3XzQtZ3bMCvM4DEAwBFp8qYKpLSo1a5sdRPfTg5R67v1T4eCJ1qg14CTK7u7agjQ0AtZ1Nh6hkSys5CWonIOqgCL3u7feRBHzodSJp7JH8u6RwsYE0mcP4rLaWAtlyRwkHF3eiUyhIiA19ZBu8mywf42nuyX0eljCt3Lkd1eUQEZoOZrA2OqfoQ5CahrByKzFgDOseim0jYBmXcsLAyccCJBTZPEjyzPb5hZKWOxT6dytSDHNRTYERTDGHFHDFGSDFGRTYFGHDFGDFSFGWERTFSDFGDFADWERTFGHDFGHFGH28} \end{align} One may readily check that the above condition is satisfied  in the smooth or non-analytic Gevrey case  for initial data $(p^\epsilon_0, v^\epsilon_0, S^\epsilon_0)$ vanishing in a neighborhood of the boundary~$\partial\Omega$. \par \startnewsection{Analytic vector field in an exterior domain and the main results}{sec02} Assume that $\Omega \subseteq \mathbb{R}^3$ is an exterior domain, located on one side of its nonempty, compact, and analytic boundary~$\partial \Omega$. Denote by $d = d(x)$ the signed distance function to the boundary $\partial \Omega$, taking positive values inside $\Omega$ and negative values outside $\Omega$. Since  $\partial \Omega$ is smooth, we have   \begin{align}    \nabla d = -\nu(x) \inon{on $\partial\Omega$},    \label{SDHNRTYERTDGHFHDFGSDFGRTYFGHDFGDFSFGWERTFSDFGDFADWERTFGHDFGHFGH111}   \end{align} where $\nu$ is the unit outward normal vector.  Moreover, the signed distance function is a real-analytic function in a neighborhood of the boundary $\partial \Omega$.  Namely, we may extend $\nu$ with the formula \eqref{SDHNRTYERTDGHFHDFGSDFGRTYFGHDFGDFSFGWERTFSDFGDFADWERTFGHDFGHFGH111} to  a neighborhood  \begin{align} 	\Omega_{\delta_0} = \{ x\in \Omega \colon d(x) < \delta_0 \}    \llabel{u82IahtpDm75YDktQvdNjWjIQH1BAceSZKVVP136vL8XhMm1OHKn2gUykFUwN8JMLBqmnvGuwGRoWUoNZY2PnmS5gQMcRYHxLyHuDo8bawaqMNYtonWu2YIOzeB6RwHuGcnfio47UPM5tOjszQBNq7mcofCNjou83emcY81svsI2YDS3SyloBNx5FBVBc96HZEOXUO3W1fIF5jtEMW6KW7D63tH0FCVTZupPlA9aIoN2sf1Bw31ggLFoDO0Mx18ooheEdKgZBCqdqpasaHFhxBrEaRgAuI5dqmWWBMuHfv90ySPtGhFFdYJJLf3Apk5CkSzr0KbVdisQkuSAJEnDTYkjPAEMua0VCtCFfz9R6Vht8UacBe7opAnGa7AbLWjHcsnARGMbn7a9npaMflftM7jvb200TWxUC4lte929joZrAIuIao1ZqdroCL55LT4Q8kNyvsIzPx4i59lKTq2JBBsZbQCECtwarVBMTH1QR6v5srWhRrD4rwf8ik7KHEgeerFVTErONmlQ5LR8vXNZLB39UDzRHZbH9fTBhRwkA2n3pg4IgrHxdfEFuz6REtDqPdwN7HTVtcE18hW6yn4GnnCE3MEQ51iPsGZ2GLbtCSthuzvPFeE28MM23ugTCdj7z7AvTLa1AGLiJ5JwWCiDPyMqa8tAKQZ9cfP42kuUzV3h6GsGFoWm9hcfj51dGtWyZzC5DaVt2Wi5IIsgDB0cXLM1FtExERIZIZ0RtQUtWcUCmFmSjxvWpZcgldopk0D7aEouRkuIdOZdWFORuqbPY6HkWOVi7FuVSDHNRTYERTDGHFHDFGSDFGRTYFGHDFGDFSFGWERTFSDFGDFADWERTFGHDFGHFGH29} \end{align} of the boundary $\partial \Omega$ such that \begin{align} \sum_{l=0}^\infty \sum_{|\alpha| = l} \frac{\eta^l}{(l-3)!}  \Vert \partial^\alpha \nu \Vert_{L^\infty (\Omega_{\delta_0})}  \les 1 , \label{SDHNRTYERTDGHFHDFGSDFGRTYFGHDFGDFSFGWERTFSDFGDFADWERTFGHDFGHFGH87} \end{align}  for some constants $\delta_0,\eta>0$. \par \subsection{Analytic vector fields} A vector field $X$ is \emph{tangential} to $\partial \Omega$ if $X d = 0$ on $\partial \Omega$.  Such $X$ may be restricted to $\partial \Omega$ by $Xf = (X\tilde{f})|_{\partial\Omega}$ for $f\in C^\infty (\partial\Omega)$, where $\tilde{f} \in C^\infty (\bar{\Omega})$ is an arbitrary extension of $f$.  In fact, if $\tilde{f}$ vanishes on $\partial\Omega$, then there exists $\tilde{g} \in C^\infty (\bar{\Omega})$ such that $\tilde{f} =  \tilde{g} d$ near $\partial \Omega$. \par Existence of global analytic vector fields in a bounded domain in $\mathbb{R}^3$  has been proven in Komatsu \cite{Ko01}. Here we  construct a family of tangential fields for an exterior domain $\Omega\subseteq {\mathbb R}^{3}$. An additional benefit from the construction is that we only need three tangential vector fields  (which is the minimum possible by the hairy ball theorem). \par \cole \begin{Theorem} \label{P01} There exist analytic vector fields $X_0, T_1, T_2, T_3$ defined globally on $\bar{\Omega}$  with the following properties. \begin{enumerate}[label*=\arabic*.] \item The fields $T_1, T_2, T_3$ are tangential to $\partial\Omega$, with  \begin{align} T_j = \sum_{i=1}^3 b_{ij}(x) \partial_i \comma 1\leq j \leq 3 , \label{SDHNRTYERTDGHFHDFGSDFGRTYFGHDFGDFSFGWERTFSDFGDFADWERTFGHDFGHFGH10} \end{align} where the coefficients $b_{ij}(x)$ are real-analytic functions with constant analyticity radii, i.e., \begin{equation} |\partial^{\alpha}b_{ij}(x)| \les C^{|\alpha|} |\alpha|! \comma x\in \bar\Omega \comma \alpha \in \mathbb{N}_0^3 , \llabel{MLWnxpSaNomkrC5uIZK9CjpJyUIeO6kgb7tr2SCYx5F11S6XqOImrs7vv0uvAgrb9hGPFnkRMj92HgczJ660kHbBBlQSIOY7FcX0cuyDlLjbU3F6vZkGbaKaMufjuxpn4Mi457MoLNW3eImcj6OOSe59afAhglt9SBOiFcYQipj5uN19NKZ5Czc231wxGx1utgJB4ueMxx5lrs8gVbZs1NEfI02RbpkfEOZE4eseo9teNRUAinujfeJYaEhns0Y6XRUF1PCf5eEAL9DL6a2vmBAU5AuDDtyQN5YLLWwPWGjMt4hu4FIoLCZLxeBVY5lZDCD5YyBwOIJeHVQsKobYdqfCX1tomCbEj5m1pNx9pnLn5A3g7Uv777YUgBRlNrTyjshaqBZXeAFtjyFlWjfc57t2fabx5Ns4dclCMJcTlqkfquFDiSdDPeX6mYLQzJzUmH043MlgFedNmXQPjAoba07MYwBaC4CnjI4dwKCZPO9wx3en8AoqX7JjN8KlqjQ5cbMSdhRFstQ8Qr2ve2HT0uO5WjTAiiIWn1CWrU1BHBMvJ3ywmAdqNDLY8lbxXMx0DDvco3RL9Qz5eqywVYqENnO8MH0PYzeVNi3yb2msNYYWzG2DCPoG1VbBxe9oZGcTU3AZuEKbkp6rNeTX0DSMczd91nbSVDKEkVazIqNKUQapNBP5B32EyprwPFLvuPiwRPl1GTdQBZEAw3d90v8P5CPAnX4Yo2q7syr5BW8HcT7tMiohaBW9U4qrbumEQ6XzMKR2BREFXk3ZOMVMYSw9SF5ekq0myNKGSDHNRTYERTDGHFHDFGSDFGRTYFGHDFGDFSFGWERTFSDFGDFADWERTFGHDFGHFGH21} \end{equation} for $1\leq i,j \leq 3$. \par \item There exists $\eta_0>0$, such that the partial derivatives may be expressed as \begin{align} \frac{\partial}{\partial x_k}  =  \xi_k(x) X_0 + \sum_{j=1}^3 \eta_{jk}(x) T_j \comma 1\leq k \leq 3 \label{SDHNRTYERTDGHFHDFGSDFGRTYFGHDFGDFSFGWERTFSDFGDFADWERTFGHDFGHFGH11} \end{align} on $\bar{\Omega}_{\eta_0}$, for some analytic coefficients $\xi_k(x)$ and $\eta_{jk}(x)$, with $1\leq j,k \leq 3$. \end{enumerate} \end{Theorem} \colb \par The same proof works in all space dimensions, with a change in the number of vector fields. If the dimension is $n$, then the number of tangential fields becomes $n(n-1)/2$. \par In the proof of the above proposition, we also need the following statement, which follows from \cite[Theorem~8.1.3]{LM01}. \par \cole \begin{Lemma} \label{L10} (Local analytic hypoellipticity) Let $\Omega$ be an exterior domain in $\mathbb{R}^3$ with analytic boundary, and let $R_0>0$ be such that \begin{equation} B_{R_0/2} \supseteq \partial\Omega . \llabel{nH0qivlRA18CbEzidOiuyZZ6kRooJkLQ0EwmzsKlld6KrKJmRxls12KG2bv8vLxfJwrIcU6Hxpq6pFy7OimmodXYtKt0VVH22OCAjfdeTBAPvPloKQzLEOQlqdpzxJ6JIzUjnTqYsQ4BDQPW6784xNUfsk0aM78qzMuL9MrAcuVVKY55nM7WqnB2RCpGZvHhWUNg93F2eRT8UumC62VH3ZdJXLMScca1mxoOO6oOLOVzfpOBOX5EvKuLz5sEW8a9yotqkcKbDJNUslpYMJpJjOWUy2U4YVKH6kVC1Vx1uvykOyDszo5bzd36qWH1kJ7JtkgV1JxqrFnqmcUyZJTp9oFIcFAk0ITA93SrLaxO9oUZ3jG6fBRL1iZ7ZE6zj8G3MHu86Ayjt3flYcmTkjiTSYvCFtJLqcJPtN7E3POqGOKe03K3WV0epWXDQC97YSbADZUNp81GFfCPbj3iqEt0ENXypLvfoIz6zoFoF9lkIunXjYyYL52UbRBjxkQUSU9mmXtzIHOCz1KH49ez6PzqWF223C0Iz3CsvuTR9sVtQCcM1eopDPy2lEEzLU0USJtJb9zgyGyfiQ4foCx26k4jLE0ula6aSIrZQHER5HVCEBL55WCtB2LCmveTDzVcp7URgI7QuFbFw9VTxJwGrzsVWM9sMJeJNd2VGGFsiWuqC3YxXoJGKwIo71fgsGm0PYFBzX8eX7pf9GJb1oXUs1q06KPLsMucNytQbL0Z0Qqm1lSPj9MTetkL6KfsC6ZobYhc2quXy9GPmZYj1GoeifeJ3pRAfn6Ypy6jSDHNRTYERTDGHFHDFGSDFGRTYFGHDFGDFSFGWERTFSDFGDFADWERTFGHDFGHFGH05} \end{equation} Assume that $f$ satisfies \begin{align} | \partial^\alpha f(x) |  \leq M R^{| \alpha|} |\alpha|! \comma \alpha \in \mathbb{N}_0^3 \comma x\in B_{2R_0}\cap \bar{\Omega} , \label{SDHNRTYERTDGHFHDFGSDFGRTYFGHDFGDFSFGWERTFSDFGDFADWERTFGHDFGHFGH140} \end{align} for some constants $M,R>0$, and suppose that $\psi$ solves \begin{align} & (-\Delta + 1) \psi  =  f \inon{in $\Omega$} \label{SDHNRTYERTDGHFHDFGSDFGRTYFGHDFGDFSFGWERTFSDFGDFADWERTFGHDFGHFGH141} \\ & \psi  =  0 \inon{on $\partial \Omega$} . \label{SDHNRTYERTDGHFHDFGSDFGRTYFGHDFGDFSFGWERTFSDFGDFADWERTFGHDFGHFGH142} \end{align} Then we have \begin{align} | \partial^\alpha \psi(x) |  \leq C(M+\Vert \psi\Vert_{L^2(B_{2R_0}\cap \Omega)}) (CR)^{| \alpha|} |\alpha|!  \comma \alpha  \in \mathbb{N}_0^3 , \label{SDHNRTYERTDGHFHDFGSDFGRTYFGHDFGDFSFGWERTFSDFGDFADWERTFGHDFGHFGH145} \end{align} for $x\in B_{R_0}\cap \bar{\Omega}$, where $C$ depends on $\partial \Omega$ and $R_0$. \end{Lemma} \colb \par This statement also applies to the case $\Omega={\mathbb R}^{3}$, in which case we refer to it as the interior hypoellipticity. It follows directly from \cite[Theorem~8.1.3]{LM01}. \par We say that $\psi\colon \bar\Omega\to {\mathbb R}$ is a \emph{globally defining function} if  $\psi\in C^{1}$ and \begin{align} & \psi|_{\partial\Omega} = 0 \textand \nabla \psi|_{\partial\Omega} \neq 0. \llabel{Ns4Y5nSEpqN4mRmamAGfYHhSaBrLsDTHCSElUyRMh66XU7hNzpZVC5VnV7VjL7kvWKf7P5hj6t1vugkLGdNX8bgOXHWm6W4YEmxFG4WaNEbGKsv0p4OG0NrduTeZaxNXqV4BpmOdXIq9abPeDPbUZ4NXtohbYegCfxBNttEwcDYSD637jJ2ms6Ta1J2xZPtKnPwAXAtJARc8n5d93TZi7q6WonEDLwWSzeSueYFX8cMhmY6is15pXaOYBbVfSChaLkBRKs6UOqG4jDVabfbdtnyfiDBFI7uhB39FJ6mYrCUUTf2X38J43KyZg87igFR5Rz1t3jH9xlOg1h7P7Ww8wjMJqH3l5J5wU8eH0OogRCvL7fJJg1ugRfMXIGSuEEfbh3hdNY3x197jRqePcdusbfkuJhEpwMvNBZVzLuqxJ9b1BTfYkRJLjOo1aEPIXvZAjvXnefhKGsJGawqjtU7r6MPoydEH26203mGiJhFnTNCDBYlnPoKO6PuXU3uu9mSg41vmakk0EWUpSUtGBtDe6dKdxZNTFuTi1fMcMhq7POvf0hgHl8fqvI3RK39fn9MaCZgow6e1iXjKC5lHOlpGpkKXdDxtz0HxEfSMjXYL8Fvh7dmJkE8QAKDo1FqMLHOZ2iL9iIm3LKvaYiNK9sb48NxwYNR0nx2t5bWCkx2a31ka8fUIaRGzr7oigRX5sm9PQ7Sr5StZEYmp8VIWShdzgDI9vRF5J81x33nNefjBTVvGPvGsxQhAlGFbe1bQi6JapOJJaceGq1vvb8rF2F3M68eDlzGtXtVmSDHNRTYERTDGHFHDFGSDFGRTYFGHDFGDFSFGWERTFSDFGDFADWERTFGHDFGHFGH12} \end{align} \begin{proof}[Proof of Theorem~\ref{P01}] First we construct globally defining functions which have constant analyticity radii. Let $R_0\geq4$ be a radius such that $ \Omega^{\mathrm c} \Subset B_{R_0/2}$. We pick an arbitrary globally defining function $d_0\in C^\infty (\mathbb{R}^3)$ which is compactly supported in $B_{R_0}$ and it satisfies 	\begin{align} 	\nabla d_0 |_{\partial \Omega} \neq 0.	 	\label{SDHNRTYERTDGHFHDFGSDFGRTYFGHDFGDFSFGWERTFSDFGDFADWERTFGHDFGHFGH260} 	\end{align} For $\epsilon \in(0,1]$, we define $d_\epsilon \colon \mathbb{R}^3 \to \mathbb{R}$ by 	\begin{align} 	d_{\epsilon}(x) 	= 	\int_{\mathbb{R}^3} H(x-y, \epsilon) d_0(y) dy 	\comma x\in{\mathbb R}^{3}                 , 	\label{SDHNRTYERTDGHFHDFGSDFGRTYFGHDFGDFSFGWERTFSDFGDFADWERTFGHDFGHFGH261} 	\end{align}  where $H\colon \mathbb{R}^3 \times \mathbb{R} \to \mathbb{R} $ is the heat kernel.  Consider the Dirichlet problem \eqref{SDHNRTYERTDGHFHDFGSDFGRTYFGHDFGDFSFGWERTFSDFGDFADWERTFGHDFGHFGH141}--\eqref{SDHNRTYERTDGHFHDFGSDFGRTYFGHDFGDFSFGWERTFSDFGDFADWERTFGHDFGHFGH142} with  $f=f_\epsilon= (-\Delta + 1)d_\epsilon$.  Since $f_\epsilon$ satisfies \eqref{SDHNRTYERTDGHFHDFGSDFGRTYFGHDFGDFSFGWERTFSDFGDFADWERTFGHDFGHFGH140}, it follows from Lemma~\ref{L10} that the solution $\psi_\epsilon$ of \eqref{SDHNRTYERTDGHFHDFGSDFGRTYFGHDFGDFSFGWERTFSDFGDFADWERTFGHDFGHFGH141}--\eqref{SDHNRTYERTDGHFHDFGSDFGRTYFGHDFGDFSFGWERTFSDFGDFADWERTFGHDFGHFGH142} satisfies \eqref{SDHNRTYERTDGHFHDFGSDFGRTYFGHDFGDFSFGWERTFSDFGDFADWERTFGHDFGHFGH145} for $x\in B_{R_0}\cap \bar{\Omega}$; note that $\Vert\psi_{\epsilon}\Vert_{L^2(\Omega)}$ is uniformly bounded in $\epsilon\in(0,1]$. In order to obtain \eqref{SDHNRTYERTDGHFHDFGSDFGRTYFGHDFGDFSFGWERTFSDFGDFADWERTFGHDFGHFGH145} for remaining $x$, we consider \eqref{SDHNRTYERTDGHFHDFGSDFGRTYFGHDFGDFSFGWERTFSDFGDFADWERTFGHDFGHFGH141}--\eqref{SDHNRTYERTDGHFHDFGSDFGRTYFGHDFGDFSFGWERTFSDFGDFADWERTFGHDFGHFGH142} with $f_\epsilon= (-\Delta + 1)d_\epsilon$ and  the version of Lemma~\ref{L10} with $\Omega = \mathbb{R}^3$. Fix any $y\in \mathbb{R}^3 \setminus B_{R_0}$. Since $f_\epsilon$ satisfies \eqref{SDHNRTYERTDGHFHDFGSDFGRTYFGHDFGDFSFGWERTFSDFGDFADWERTFGHDFGHFGH140} in $B_{2}(y)$, Lemma~\ref{L10} implies that the solution $\psi_\epsilon$ of \eqref{SDHNRTYERTDGHFHDFGSDFGRTYFGHDFGDFSFGWERTFSDFGDFADWERTFGHDFGHFGH141} satisfies \eqref{SDHNRTYERTDGHFHDFGSDFGRTYFGHDFGDFSFGWERTFSDFGDFADWERTFGHDFGHFGH145} for all $x \in B_1(y)$.  Thus, we obtain 	\begin{align} 	|\partial^\alpha \psi_\epsilon (x)| 	\leq 	CM (CR)^{|\alpha|}  	|\alpha|! 	\comma \alpha \in \mathbb{N}_0^3 	\commaone x\in \bar{\Omega} 	. 	\label{SDHNRTYERTDGHFHDFGSDFGRTYFGHDFGDFSFGWERTFSDFGDFADWERTFGHDFGHFGH147} 	\end{align} By \eqref{SDHNRTYERTDGHFHDFGSDFGRTYFGHDFGDFSFGWERTFSDFGDFADWERTFGHDFGHFGH260}--\eqref{SDHNRTYERTDGHFHDFGSDFGRTYFGHDFGDFSFGWERTFSDFGDFADWERTFGHDFGHFGH261} and continuity, we infer that  	\begin{align} 	\nabla \psi_\epsilon |_{\partial \Omega}  	\neq  	0 	,    \llabel{5y14vmwIXa2OGYhxUsXJ0qgl5ZGAtHPZdoDWrSbBSuNKi6KWgr39s9tc7WM4Aws1PzI5cCO7Z8y9lMTLAdwhzMxz9hjlWHjbJ5CqMjhty9lMn4rc76AmkKJimvH9rOtbctCKrsiB04cFVDl1gcvfWh65nxy9ZS4WPyoQByr3vfBkjTZKtEZ7rUfdMicdyCVqnD036HJWMtYfL9fyXxO7mIcFE1OuLQsAQNfWv6kV8Im7Q6GsXNCV0YPoCjnWn6L25qUMTe71vahnHDAoXAbTczhPcfjrjW5M5G0nzNM5TnlJWOPLhM6U2ZFxwpg4NejP8UQ09JX9n7SkEWixERwgyFvttzp4Asv5FTnnMzLVhFUn56tFYCxZ1BzQ3ETfDlCad7VfoMwPmngrDHPfZV0aYkOjrZUw799etoYuBMIC4ovEY8DOLNURVQ5lti1iSNZAdwWr6Q8oPFfae5lAR9gDRSiHOeJOWwxLv20GoMt2Hz7YcalyPZxeRuFM07gaV9UIz7S43k5TrZiDMt7pENCYiuHL7gac7GqyN6Z1ux56YZh2dyJVx9MeUOMWBQfl0EmIc5Zryfy3irahCy9PiMJ7ofoOpdennsLixZxJtCjC9M71vO0fxiR51mFIBQRo1oWIq3gDPstD2ntfoX7YUoS5kGuVIGMcfHZe37ZoGA1dDmkXO2KYRLpJjIIomM6Nuu8O0jO5NabUbRnZn15khG94S21V4Ip457ooaiPu2jhIzosWFDuO5HdGrdjvvtTLBjovLLiCo6L5LwaPmvD6Zpal69Ljn11reT2CSDHNRTYERTDGHFHDFGSDFGRTYFGHDFGDFSFGWERTFSDFGDFADWERTFGHDFGHFGH31} 	\end{align} for sufficiently small $\epsilon >0$, which implies that $\psi_\epsilon$ is a globally defining function. \par Next, we pick an arbitrary globally defining function $\psi$ satisfying \eqref{SDHNRTYERTDGHFHDFGSDFGRTYFGHDFGDFSFGWERTFSDFGDFADWERTFGHDFGHFGH147}. Choose $\eta_0>0$ so that $\nabla \psi \neq 0$ on~$\bar{\Omega}_{\eta_0}$. For $1 \leq j,k\leq 3$, we set 	\begin{align} 	X_0=  	\sum_{i=1}^3 \frac{\partial \psi}{\partial x_i}  	\frac{\partial}{\partial x_i} 	\textand 	T_{jk} 	= 	\frac{\partial \psi}{\partial x_j} \frac{\partial}{\partial x_k}  	- 	\frac{\partial \psi}{\partial x_k} \frac{\partial}{\partial x_j} 	. 	\label{SDHNRTYERTDGHFHDFGSDFGRTYFGHDFGDFSFGWERTFSDFGDFADWERTFGHDFGHFGH42} 	\end{align} Since $\psi \in C^\infty (\bar{\Omega})$ vanishes on $\partial\Omega$, there exists some $\tilde{g} \in C^\infty (\bar{\Omega})$ such that $\psi = \tilde{g} d$ near $\partial \Omega$. For $1\leq j,k\leq 3$, we compute 	\begin{align} 	\begin{split} 	T_{jk} d  	& 	=  	\frac{\partial (\tilde{g} d)}{\partial x_j} \frac{\partial d}{\partial x_k}  	- 	\frac{\partial (\tilde{g} d)}{\partial x_k} \frac{\partial d}{\partial x_j}  	\\& 	= 	d	 	\left( 	\frac{\partial \tilde{g} }{\partial x_j} \frac{\partial d}{\partial x_k}  	- 	\frac{\partial \tilde{g}}{\partial x_k} \frac{\partial d}{\partial x_j}  	\right) 	+ 	\tilde{g} 	\left( 	\frac{\partial d}{\partial x_j} \frac{\partial d}{\partial x_k}  	- 	\frac{\partial d}{\partial x_k} \frac{\partial d}{\partial x_j}  	\right) 	= 	0 	,\inon{on $\partial \Omega$} 	, 	\end{split}    \llabel{PmvjrL3xHmDYKuv5TnpC1fMoURRToLoilk0FEghakm5M9cOIPdQlGDLnXerCykJC10FHhvvnYaTGuqUrfTQPvwEqiHOvOhD6AnXuvGlzVAvpzdOk36ymyUoFbAcAABItOes52Vqd0Yc7U2gBt0WfFVQZhrJHrlBLdCx8IodWpAlDS8CHBrNLzxWp6ypjuwWmgXtoy1vPbrauHyMNbkUrZD6Ee2fzIDtkZEtiLmgre1woDjuLBBSdasYVcFUhyViCxB15yLtqlqoUhgL3bZNYVkorzwa3650qWhF22epiXcAjA4ZV4bcXxuB3NQNp0GxW2Vs1zjtqe2pLEBiS30E0NKHgYN50vXaK6pNpwdBX2Yv7V0UddTcPidRNNCLG47Fc3PLBxK3Bex1XzyXcj0Z6aJk0HKuQnwdDhPQ1QrwA05v9c3pnzttztx2IirWCZBoS5xlOKCiD3WFh4dvCLQANAQJGgyvODNTDFKjMc0RJPm4HUSQkLnTQ4Y6CCMvNjARZblir7RFsINzHiJlcgfxSCHtsZOG1VuOzk5G1CLtmRYIeD35BBuxZJdYLOCwS9lokSNasDLj5h8yniu7hu3cdizYh1PdwEl3m8XtyXQRCAbweaLiN8qA9N6DREwy6gZexsA4fGEKHKQPPPKMbksY1jM4h3JjgSUOnep1wRqNGAgrL4c18Wv4kchDgRx7GjjIBzcKQVf7gATrZxOy6FF7y93iuuAQt9TKRxS5GOTFGx4Xx1U3R4s7U1mpabpDHgkicxaCjkhnobr0p4codyxTCkVj8tW4iP2OhSDHNRTYERTDGHFHDFGSDFGRTYFGHDFGDFSFGWERTFSDFGDFADWERTFGHDFGHFGH32} 	\end{align} which implies that $T_{jk}$ are tangential vector fields, and thus \eqref{SDHNRTYERTDGHFHDFGSDFGRTYFGHDFGDFSFGWERTFSDFGDFADWERTFGHDFGHFGH10} is proven. From \eqref{SDHNRTYERTDGHFHDFGSDFGRTYFGHDFGDFSFGWERTFSDFGDFADWERTFGHDFGHFGH42}, we have 	\begin{align} 	\frac{\partial \psi}{\partial x_k} X_0 	+ 	\sum_{j=1}^3 \frac{\partial \psi}{\partial x_j} T_{jk} 	= 	\vert \nabla \psi \vert^2 \frac{\partial}{\partial x_k} 	,   	\llabel{TRF6kU2k2ooZJFsqY4BFSNI3uW2fjOMFf7xJveilbUVTArCTvqWLivbRpg2wpAJOnlRUEPKhj9hdGM0MigcqQwkyunBJrTLDcPgnOSCHOsSgQsR35MB7BgkPk6nJh01PCxdDsw514O648VD8iJ54FW6rs6SyqGzMKfXopoe4eo52UNB4Q8f8NUz8u2nGOAXHWgKtGAtGGJsbmz2qjvSvGBu5e4JgLAqrmgMmS08ZFsxQm28M3z4Ho1xxjj8UkbMbm8M0cLPL5TS2kIQjZKb9QUx2Ui5Aflw1SLDGIuWUdCPjywVVM2ct8cmgOBS7dQViXR8Fbta1mtEFjTO0kowcK2d6MZiW8PrKPI1sXWJNBcREVY4H5QQGHbplPbwdTxpOI5OQZAKyiix7QeyYI91Ea16rKXKL2ifQXQPdPNL6EJiHcKrBs2qGtQbaqedOjLixjGiNWr1PbYSZeSxxFinaK9EkiCHV2a13f7G3G3oDKK0ibKVy453E2nFQS8Hnqg0E32ADddEVnmJ7HBc1t2K2ihCzZuy9kpsHn8KouARkvsHKPy8YodOOqBihF1Z3CvUFhmjgBmuZq7ggWLg5dQB1kpFxkk35GFodk00YD13qIqqbLwyQCcyZRwHAfp79oimtCc5CV8cEuwUw7k8Q7nCqWkMgYrtVRIySMtZUGCHXV9mr9GHZol0VEeIjQvwgw17pDhXJSFUcYbqUgnGV8IFWbS1GXaz0ZTt81w7EnIhFF72v2PkWOXlkrw6IPu5679vcW1f6z99lM2LI1Y6Naaxfl18gT0gDptVlSDHNRTYERTDGHFHDFGSDFGRTYFGHDFGDFSFGWERTFSDFGDFADWERTFGHDFGHFGH41} 	\end{align} for $1\leq k \leq 3$.  Thus \eqref{SDHNRTYERTDGHFHDFGSDFGRTYFGHDFGDFSFGWERTFSDFGDFADWERTFGHDFGHFGH11} follows since $\nabla \psi \neq 0$ on $\bar{\Omega}_{\eta_0}$. To show that one may choose only three rather than nine tangential fields, note that $T_{jk}=-T_{kj}$ for $1\leq j,k\leq 3$, and thus also $T_{11}=T_{22}=T_{33}=0$. \end{proof} \par \subsection{Main results} Let $\Omega\subseteq {\mathbb R}^{3}$ be an exterior domain with an analytic boundary $\partial \Omega$. We assume that the initial data $(p_0^\epsilon, v_0^\epsilon, S_0^\epsilon)$ satisfies 	\begin{align} 	\Vert (p_0^\epsilon, v_0^\epsilon, S_0^\epsilon) \Vert_{H^5(\Omega)}  	\leq 	M_0 	\label{SDHNRTYERTDGHFHDFGSDFGRTYFGHDFGDFSFGWERTFSDFGDFADWERTFGHDFGHFGH109} 	\end{align} 	and 	\begin{align} 	\sum_{j=0}^\infty  	\frac{ \tau_0^{(j-3)_+}}{(j-3)!} 	\Vert \partial_x^j (p_0^\epsilon, v_0^\epsilon, S_0^\epsilon) \Vert_{L^2(\Omega)}    	\leq 	M_0, 	\label{SDHNRTYERTDGHFHDFGSDFGRTYFGHDFGDFSFGWERTFSDFGDFADWERTFGHDFGHFGH23} \end{align} for some fixed constants $\tau_0, M_0 >0$. It can be  shown, analogously to Lemma~\ref{L32} below, that \eqref{SDHNRTYERTDGHFHDFGSDFGRTYFGHDFGDFSFGWERTFSDFGDFADWERTFGHDFGHFGH23} implies \begin{align} \sum_{j=0}^\infty \sum_{k=0}^\infty  \frac{ \bar{\tau}_0^{(j+k-3)_+}}{(j+k-3)!} \Vert \partial_x^j \TT^k (p_0^\epsilon, v_0^\epsilon, S_0^\epsilon) \Vert_{L^2(\Omega)}    \leq Q(M_0),	 \label{SDHNRTYERTDGHFHDFGSDFGRTYFGHDFGDFSFGWERTFSDFGDFADWERTFGHDFGHFGH503} \end{align} for some function $Q$ and some constant $\bar{\tau}_0>0$. In \eqref{SDHNRTYERTDGHFHDFGSDFGRTYFGHDFGDFSFGWERTFSDFGDFADWERTFGHDFGHFGH23}--\eqref{SDHNRTYERTDGHFHDFGSDFGRTYFGHDFGDFSFGWERTFSDFGDFADWERTFGHDFGHFGH503} and below, we use $\TT = (T_1, T_2, T_3)$ to denote the tangential vector fields from Theorem~\ref{P01},  and use $\partial_x = (\partial_1, \partial_2, \partial_3)$ to denote the gradient. We adopt the following agreement for the iterative derivatives $\TT^k$ and $\partial_{x}^j$. The symbol $\TT^k$ is understood in the tensorial sense, i.e., $\TT^k$ stands for the list of all possible operators $T_{\beta_1} \cdots T_{\beta_k}$, where $\beta_i \in \{1,2,3\}$ for $1\leq i \leq k$.  We also adopt the same agreement for $\partial_x^j$. \par For $\tau >0$, we define the mixed weighted analytic space 	\begin{align} 	A(\tau)  	=  	\{ u \in C^\infty (\Omega) \colon  \Vert u \Vert_{A(\tau)} <\infty  \} 	,    \llabel{CN4jfGSbCro5Dv78CxaukYiUIWWyYDRw8z7KjPx7ChC7zJvb1b0rFd7nMxk091wHvy4u5vLLsJ8NmAkWtxuf4P5NwP23b06sFNQ6xgDhuRGbK7j2O4gy4p4BLtop3h2kfyI9wO4AaEWb36YyHYiI1S3COJ7aN1r0sQOrCAC4vL7yrCGkIRlNuGbOuuk1awLDK2zlKa40hyJnDV4iFxsqO001rqCeOAO2es7DRaCpUG54F2i97xSQrcbPZ6K8Kudn9e6SYo396Fr8LUxyXOjdFsMrl54EhT8vrxxF2phKPbszrlpMAubERMGQAaCBu2LqwGasprfIZOiKVVbuVae6abaufy9KcFk6cBlZ5rKUjhtWE1Cnt9RmdwhJRySGVSOVTv9FY4uzyAHSp6yT9s6R6oOi3aqZlL7bIvWZ18cFaiwptC1ndFyp4oKxDfQz28136a8zXwsGlYsh9Gp3TalnrRUKttBKeFr4543qU2hh3WbYw09g2WLIXzvQzMkj5f0xLseH9dscinGwuPJLP1gEN5WqYsSoWPeqjMimTybHjjcbn0NO5hzP9W40r2w77TAoz70N1au09bocDSxGc3tvKLXaC1dKgw9H3o2kEoulIn9TSPyL2HXO7tSZse01Z9HdslDq0tmSOAVqtA1FQzEMKSbakznw839wnH1DpCjGIk5X3B6S6UI7HIgAaf9EV33Bkkuo3FyEi8Ty2ABPYzSWjPj5tYZETYzg6Ix5tATPMdlGke67Xb7FktEszyFycmVhGJZ29aPgzkYj4cErHCdP7XFHUO9zoy4ASDHNRTYERTDGHFHDFGSDFGRTYFGHDFGDFSFGWERTFSDFGDFADWERTFGHDFGHFGH33} 	\end{align} where 	\begin{align} 	\begin{split} 	\Vert u \Vert_{A(\tau)} 	& 	= 	\sum_{(j,k,i)\in \mathbb{N}_0^3}  	\frac{\kappa^{(j-1)_+} \bar{\kappa}^{k}\tau(t)^{(j+k+i-3)_+}}{(j+k+i-3)!} 	\Vert \partial_x^j \TT^k (\epsilon \partial_t)^i u \Vert_{L^2(\Omega)} ; 	\label{SDHNRTYERTDGHFHDFGSDFGRTYFGHDFGDFSFGWERTFSDFGDFADWERTFGHDFGHFGH97} 	\end{split} 	\end{align} here  $\tau \in (0,1]$ represents the mixed space-time analyticity radius,  while $\kappa$ and $\bar{\kappa}$, where $0 < \kappa \leq  \bar{\kappa} \leq 1$ are parameters which represent the balances of radii in different directions. In \eqref{SDHNRTYERTDGHFHDFGSDFGRTYFGHDFGDFSFGWERTFSDFGDFADWERTFGHDFGHFGH97} and below we use the convention $n!= 1$ for $n \in -\mathbb{N}$.  We show in Section~\ref{sec07} that \eqref{SDHNRTYERTDGHFHDFGSDFGRTYFGHDFGDFSFGWERTFSDFGDFADWERTFGHDFGHFGH503} implies, regardless of the choices of $\kappa \in (0,1]$ and $\bar \kappa \in (0,1]$, that we have \begin{align} \Vert (p_0, v_0, S_0) \Vert_{A(\tilde{\tau}_0)} \leq Q(M_0) , \llabel{ZaiSROpIn0tp7kZzUVHQtm3ip3xEd41By72uxIiY8BCLbOYGoLDwpjuza6iPakZdhaD3xSXyjpdOwoqQqJl6RFglOtX67nm7s1lZJmGUrdIdXQ7jps7rcdACYZMsBKANxtkqfNhktsbBf2OBNZ5pfoqSXtd3cHFLNtLgRoHrnNlwRnylZNWVNfHvOB1nUAyjtxTWW4oCqPRtuVuanMkLvqbxpNi0xYnOkcdFBdrw1Nu7cKybLjCF7P4dxj0Sbz9faVCWkVFos9t2aQIPKORuEjEMtbSHsYeG5Z7uMWWAwRnR8FwFCzXVVxnFUfyKLNk4eOIlyn3ClI5HP8XP6S4KFfIl62VlbXgcauth861pUWUx2aQTWgrZwcAx52TkqoZXVg0QGrBrrpeiwuWyJtd9ooD8tUzAdLSnItarmhPAWBmnmnsbxLIqX4RQSTyoFDIikpeILhWZZ8icJGa91HxRb97knWhp9sAVzPo8560pRN2PSMGMMFK5XW52OnWIyoYngxWno868SKbbu1Iq1SyPkHJVCvseVGWrhUdewXw6CSY1be3hD9PKha1y0SRwyxiAGzdCMVMmiJaemmP8xrbJXbKLDYE1FpXUKADtF9ewhNefd2XRutTl1HYJVp5cAhM1JfK7UIcpkdTbEndM6FWHA72PgLHzXlUo39oW90BuDeJSlnVRvz8VDV48tId4DtgFOOa47LEH8QwnRGNBM0RRULluASzjxxwGIBHmVyyLdkGww5eEgHFvsFUnzl0vgOaQDCVEz64r8UvVHTtDykrEuFaS35p5yn6QSDHNRTYERTDGHFHDFGSDFGRTYFGHDFGDFSFGWERTFSDFGDFADWERTFGHDFGHFGH98} \end{align} for some function $Q$, where $\tilde{\tau}_0 = \bar{\tau}_0 / Q(M_0)$ represents the mixed space-time analyticity radius. Note that the time derivatives of the initial data are defined iteratively by differentiating the equations \eqref{SDHNRTYERTDGHFHDFGSDFGRTYFGHDFGDFSFGWERTFSDFGDFADWERTFGHDFGHFGH01}--\eqref{SDHNRTYERTDGHFHDFGSDFGRTYFGHDFGDFSFGWERTFSDFGDFADWERTFGHDFGHFGH02} and then evaluating at $t=0$.  The analyticity radius function is defined as \begin{align} \tau(t) = \tau(0) - Kt, \label{SDHNRTYERTDGHFHDFGSDFGRTYFGHDFGDFSFGWERTFSDFGDFADWERTFGHDFGHFGH99} \end{align} where $\tau(0) \leq \min\{ 1, \tilde{\tau}_0 \}$ is a sufficiently small parameter (different from $\tilde{\tau}_0$),  and $K\geq 1$ is a sufficiently large parameter, both to be determined below. \par Our first main theorem provides a uniform in $\epsilon$ boundedness of the analytic norm on a time interval independent of $\epsilon$. \par \cole \begin{Theorem} \label{T01} \rm Let $\Omega \subseteq \mathbb{R}^3$ be an exterior domain with analytic boundary $\partial \Omega$. We assume that the initial data $(p^\epsilon_0, v^\epsilon_0, S^\epsilon_0)$ satisfies \eqref{SDHNRTYERTDGHFHDFGSDFGRTYFGHDFGDFSFGWERTFSDFGDFADWERTFGHDFGHFGH109}--\eqref{SDHNRTYERTDGHFHDFGSDFGRTYFGHDFGDFSFGWERTFSDFGDFADWERTFGHDFGHFGH23} for some fixed constants $M_0, \tau_0>0$.  Also, suppose that $(p_0^\epsilon, v_0^\epsilon, S_0^\epsilon)$ satisfies compatibility condition of all orders. Then there exist sufficiently small constants $\kappa$, $\bar{\kappa}$, $\tau(0)$, $\epsilon_0$, $T_0>0$, depending on $M_0$, such that \begin{align} \Vert (p^\epsilon, v^\epsilon, S^\epsilon)(t) \Vert_{A(\tau)}  \leq M \comma 0<\epsilon \leq \epsilon_0 \commaone t\in [0,T_0] , \llabel{ZUcX3mfETExz1kvqEpOVVEFPIVpzQlMOIZ2yTTxIUOm0fWL1WoxCtlXWs9HU4EF0IZ1WDv3TP42LN7TrSuR8uMv1tLepvZoeoKLxf9zMJ6PUIn1S8I4KY13wJTACh5Xl8O5g0ZGwDdtu68wvrvnDCoqYjJ3nFKWMAK8VOeGo4DKxnEOyBwgmttcES8dmToAD0YBFlyGRBpBbo8tQYBwbSX2lcYnU0fhAtmyR3CKcUAQzzETNgbghHT64KdOfLqFWuk07tDkzfQ1dgBcw0LSYlr79U81QPqrdfH1tb8kKnDl52FhCj7TXiP7GFC7HJKfXgrP4KOOg18BM001mJPTpubQr61JQu6oGr4baj60kzdXoDgAOX2DBkLymrtN6T7us2Cp6eZm1aVJTY8vYPOzMnsAqs3RL6xHumXNAB5eXnZRHaiECOaaMBwAb15iFWGucZlU8JniDNKiPGWzq41iBj1kqbakZFSvXqvSiRbLTriSy8QYOamQUZhOrGHYHWguPBzlAhuao59RKUtrF5KbjsKseTPXhUqRgnNALVtaw4YJBtK9fN7bN9IEwKLTYGtnCcc2nfMcx7VoBt1IC5teMHX4g3JK4JsdeoDl1Xgbm9xWDgZ31PchRS1R8W1hap5Rh6JjyTNXSCUscxK4275D72gpRWxcfAbZY7Apto5SpTzO1dPAVyZJiWCluOjOtEwxUB7cTtEDqcAbYGdZQZfsQ1AtHyxnPL5K7D91u03s8K20rofZ9w7TjxyG7qbCAhssUZQuPK7xUeK7F4HKfrCEPJrgWHDZQpvRkSDHNRTYERTDGHFHDFGSDFGRTYFGHDFGDFSFGWERTFSDFGDFADWERTFGHDFGHFGH190} \end{align} where $\tau$ is as in \eqref{SDHNRTYERTDGHFHDFGSDFGRTYFGHDFGDFSFGWERTFSDFGDFADWERTFGHDFGHFGH99} and $K$ and $M$ are sufficiently large constants depending on $M_0$. \end{Theorem} \colb \par We note that the same statement and the proof apply the case of a bounded domain. \par The second main result states that the solutions of \eqref{SDHNRTYERTDGHFHDFGSDFGRTYFGHDFGDFSFGWERTFSDFGDFADWERTFGHDFGHFGH01}--\eqref{SDHNRTYERTDGHFHDFGSDFGRTYFGHDFGDFSFGWERTFSDFGDFADWERTFGHDFGHFGH02}  converge to the solution of the stratified incompressible Euler equations \begin{align} &r(S, 0) (\partial_t v + v\cdot \nabla v) + \nabla \pi = 0, \label{SDHNRTYERTDGHFHDFGSDFGRTYFGHDFGDFSFGWERTFSDFGDFADWERTFGHDFGHFGH202} \\& \dive v = 0, \label{SDHNRTYERTDGHFHDFGSDFGRTYFGHDFGDFSFGWERTFSDFGDFADWERTFGHDFGHFGH203} \\& \partial_t S + v \cdot \nabla S = 0 , \label{SDHNRTYERTDGHFHDFGSDFGRTYFGHDFGDFSFGWERTFSDFGDFADWERTFGHDFGHFGH204} \end{align} as $\epsilon \to 0$. We define the spatial analytic space \begin{align} X(\delta)  = \{ u \in C^\infty (\Omega) \colon  \Vert u \Vert_{X(\delta)} <\infty  \} , 	 \llabel{O8XveaSBOXSeeXV5jkgzLUTmMbomaJfxu8gArndzSIB0YQSXvcZW8voCOoOHyrEuGnS2fnGEjjaLzZIocQegwHfSFKjW2LbKSnIcG9WnqZya6qAYMSh2MmEAsw18nsJFYAnbrxZT45ZwBsBvK9gSUgyBk3dHqdvYULhWgGKaMfFk78mP20meVaQp2NWIb6hVBSeSVwnEqbq6ucnX8JLkIRJbJEbwEYwnvLBgM94Gplclu2s3Um15EYAjs1GLnhzG8vmhghsQcEDE1KnaHwtuxOgUDLBE59FLxIpvuKfJEUTQSEaZ6huBCaKXrlnir1XmLKH3hVPrqixmTkRzh0OGpOboN6KLCE0GaUdtanZ9Lvt1KZeN5GQcLQLL0P9GXuakHm6kqk7qmXUVH2bUHgav0Wp6Q8JyITzlpqW0Yk1fX8gjGcibRarmeSi8lw03WinNXw1gvvcDeDPSabsVwZu4haO1V2DqwkJoRShjMBgryglA93DBdS0mYAcEl5aEdpIIDT5mbSVuXo8NlY24WCA6dfCVF6Ala6iNs7GChOvFAhbxw9Q71ZRC8yRi1zZdMrpt73douogkAkGGE487Vii4OfwJesXURdzVLHU0zms8W2ZtziY5mw9aBZIwk5WNmvNM2HdjnewMR8qp2VvupcV4PcjOGeu35u5cQXNTykfTZXAJHUnSs4zxfHwf10ritJYoxRto5OMFPhakRgzDYPm02mG18vmfV11Nn87zSX59DE0cN99uEUz2rTh1FP8xjrmq2Z7utpdRJ2DdYkjy9JYkoc38KduZ9vydSDHNRTYERTDGHFHDFGSDFGRTYFGHDFGDFSFGWERTFSDFGDFADWERTFGHDFGHFGH504} \end{align} where \begin{align} \Vert u \Vert_{X(\delta)} = \sum_{j=0}^\infty \frac{\delta^{(j-3)_+}}{(j-3)!} \Vert \partial_x^j  u \Vert_{L^2(\Omega)} , \llabel{OwkO0djhXSxSvHwJoXE79f8qhiBr8KYTxOfcYYFsMyj0HvK3ayUwt4nA5H76bwUqyJQodOu8UGjbt6vlcxYZt6AUxwpYr18uOv62vjnwFrCrfZ4nlvJuh2SpVLOvpOlZnPTG07VReixBmXBxOBzpFW5iBIO7RVmoGnJu8AxolYAxlJUrYKVKkpaIkVCuPiDO8IHPUndzeLPTILBP5BqYyDLZDZadbjcJAT644Vp6byb1g4dE7YdzkeOYLhCReOmmxF9zsu0rp8Ajzd2vHeo7L5zVnL8IQWnYATKKV1f14s2JgeCb3v9UJdjNNVBINix1q5oyrSBM2Xtgrv8RQMaXka4AN9iNinzfHxGpA57uAE4jMfg6S6eNGKvJL3tyH3qwdPrx2jFXW2WihpSSxDraA7PXgjK6GGlOg5PkRd2n53eEx4NyGhd8ZRkONMQqLq4sERG0CssQkdZUaOvWrplaBOWrSwSG1SM8Iz9qkpdv0CRMsGcZLAz4Gk70eO7k6df4uYnR6T5DuKOTsay0DawWQvn2UOOPNqQT7H4HfiKYJclRqM2g9lcQZcvCNBP2BbtjvVYjojrrh78tWR886ANdxeASVPhK3uPrQRs6OSW1BwWM0yNG9iBRI7opGCXkhZpEo2JNtkyYOpCY9HL3o7Zu0J9FTz6tZGLn8HAeso9umpyucs4l3CA6DCQ0m0llFPbc8z5Ad2lGNwSgAXeNHTNpwdS6e3ila2tlbXN7c1itXaDZFakdfJkz7TzaO4kbVhnYHfTda9C3WCbtwMXHWxoCCc4Ws2CUHBsNSDHNRTYERTDGHFHDFGSDFGRTYFGHDFGDFSFGWERTFSDFGDFADWERTFGHDFGHFGH169} \end{align} for some constant $\delta>0$. \par \cole \begin{Theorem} \label{T02} Let $\Omega \subseteq \mathbb{R}^3$ be an exterior domain with an analytic boundary. Assume that the initial data $(p^\epsilon_0, v^\epsilon_0, S^\epsilon_0)$ satisfy \eqref{SDHNRTYERTDGHFHDFGSDFGRTYFGHDFGDFSFGWERTFSDFGDFADWERTFGHDFGHFGH109}--\eqref{SDHNRTYERTDGHFHDFGSDFGRTYFGHDFGDFSFGWERTFSDFGDFADWERTFGHDFGHFGH23} uniformly for fixed $\tau_0, M_0>0$ and the compatibility condition of all orders. Also, suppose that $(v_0^\epsilon, S_0^\epsilon)$ converge to $(v_0, S_0)$ in $H^{3}(\Omega)$,  with $S_0^\epsilon$ decaying in the sense of \begin{align} \vert S_0^\epsilon (x) \vert \les \vert x \vert^{-1-\zeta}, \textand |\nabla S_0^\epsilon (x) \vert          \les |x \vert^{-2-\zeta}, \llabel{LFEfjS4SGI4I4hqHh2nCaQ4nMpnzYoYE5fDsXhCHJzTQOcbKmvEplWUndVUorrqiJzRqTdIWSQBL96DFUd64k5gvQh0djrGlw795xV6KzhTl5YFtCrpybHH86h3qnLyzyycGoqmCbfh9hprBCQpFeCxhUZ2oJF3aKgQH8RyImF9tEksgPFMMJTAIyz3ohWjHxMR86KJONKTc3uyRNnSKHlhb11Q9Cwrf8iiXqyYL4zh9s8NTEve539GzLgvhDN7FeXo5kAWAT6VrwhtDQwytuHOa5UIOExbMpV2AHpuuCHWItfOruxYfFqsaP8ufHF16CEBXKtj6ohsuvT8BBPDNgGfKQg6MBK2x9jqRbHmjIUEKBIm0bbKacwqIXijrFuq9906Vym3Ve1gBdMy9ihnbA3gBo5aBKK5gfJSmNeCWwOMt9xutzwDkXIY7nNhWdDppZUOq2Ae0aW7A6XoIcTSLNDZyf2XjBcUweQTZtcuXIDYsDhdAu3VMBBBKWIcFNWQdOu3Fbc6F8VN77DaIHE3MZluLYvBmNZ2wEauXXDGpeKRnwoUVB2oMVVehW0ejGgbgzIw9FwQhNYrFI4pTlqrWnXzz2qBbalv3snl2javzUSncpwhcGJ0Di3Lr3rs6F236obLtDvN9KqApOuold3secxqgSQNZNfw5tBGXPdvW0k6G4Byh9V3IicOnR2obfx3jrwt37u82fwxwjSmOQq0pq4qfvrN4kFWhPHRmylxBx1zCUhsDNYINvLdtVDG35kTMT0ChPEdjSG4rWN6v5IIMTVB5ycWuYOoUSDHNRTYERTDGHFHDFGSDFGRTYFGHDFGDFSFGWERTFSDFGDFADWERTFGHDFGHFGH157} \end{align} for $0<\epsilon \leq \epsilon_0$ and some constant $\zeta>0$. Then $(v^\epsilon, p^\epsilon, S^\epsilon)$ converges to  $(v^{({\rm inc})}, 0, S^{({\rm inc})})\in L^{\infty}([0,T_0],X(\delta))$ in $L^2 ([0,T_0], X(\delta))$, where $\delta \in (0, \tau_0]$ is a sufficiently small constant and $(v^{({\rm inc})}, S^{({\rm inc})})$ is the solution to \eqref{SDHNRTYERTDGHFHDFGSDFGRTYFGHDFGDFSFGWERTFSDFGDFADWERTFGHDFGHFGH202}--\eqref{SDHNRTYERTDGHFHDFGSDFGRTYFGHDFGDFSFGWERTFSDFGDFADWERTFGHDFGHFGH204} with the initial data $(w_0, S_0)$, and $w_0$ is the unique solution of  \begin{align} &\dive w_0 = 0,     \llabel{6SevyecOTfZJvBjSZZkM68vq4NOpjX0oQ7rvMvmyKftbioRl5c4ID72iFH0VbQzhjHU5Z9EVMX81PGJssWedmhBXKDAiqwUJVGj2rIS92AntBn1QPR3tTJrZ1elVoiKUstzA8fCCgMwfw4jKbDberBRt6T8OZynNOqXc53PgfLK9oKe1pPrYBBZYuuiCwXzA6kaGbtwGpmRTmKviwHEzRjhTefripvLAXk3PkLNDg5odcomQj9LYIVawVmLpKrto0F6Ns7MmkcTL9Tr8fOT4uNNJvZThOQwCOCRBHRTxhSBNaIizzbKIBEcWSMYEhDkRtPWGKtUmo26acLbBnI4t2P11eRiPP99nj4qQ362UNAQaHJPPY1OgLhN8sta9eJzPgmE4zQgB0mlAWBa4Emu7mnfYgbNLzddGphhJV9hyAOGCNjxJ83Hg6CAUTnusW9pQrWv1DfVlGnWxMBbe9WwLtOdwDERmlxJ8LTqKWTtsR0cDXAfhRX1zXlAUuwzqnO2o7rtoiSMrOKLCqjoq1tUGGiIxuspoiitjaNRngtxS0r98rwXF7GNiepzEfAO2sYktIdgH1AGcRrd2w89xoOKyNnLaLRU03suU3JbS8dok8tw9NQSY4jXY625KCcPLyFRlSp759DeVbY5b69jYOmdfb99j15lvLvjskK2gEwlRxOtWLytZJ1yZ5Pit35SOiivz4F8tqMJIgQQiOobSpeprt2vBVqhvzkLlf7HXA4soMXjWdMS7LeRDiktUifLJHukestrvrl7mYcSOB7nKWMD0xBqkbxFgTTNISDHNRTYERTDGHFHDFGSDFGRTYFGHDFGDFSFGWERTFSDFGDFADWERTFGHDFGHFGH36} \\& \curl (r_0 w_0) = \curl (r_0 v_0) ,    \llabel{weyVIG6Uy3dL0C3MzFxsBE7zUhSetBQcX7jn22rr0yL1ErbpLRm3ida5MdPicdnMOiZCyGd2MdKUbxsaI9TtnHXqAQBjuN5I4Q6zz4dSWYUrhxTCuBgBUT992uczEmkqK1ouCaHJBR0Qnv1artFiekBu49ND9kK9eKBOgPGzqfKJ67NsKz3BywIwYxEoWYf6AKuyVPj8B9D6quBkFCsKHUDCksDYK3vs0Ep3gM2EwlPGjRVX6cxlbVOfAll7g6yL9PWyo58h0e07HO0qz8kbe85ZBVCYOKxNNLa4aFZ7mw7moACU1q1lpfmE5qXTA0QqVMnRsbKzHo5vX1tpMVZXCznmSOM73CRHwQPTlvVN7lKXI06KT6MTjO3Yb87pgozoxydVJHPL3k2KRyx3b0yPBsJmNjETPJi4km2fxMh35MtRoirNE9bU7lMo4bnj9GgYA6vsEsONRtNmDFJej96STn3lJU2u16oTEXogvMqwhD0BKr1CisVYbA2wkfX0n4hD5Lbr8l7ErfuN8OcUjqeqzCCyx6hPAyMrLeB8CwlkThixdIzviEWuwI8qKa0VZEqOroDUPGphfIOFSKZ3icda7Vh3ywUSzkkW8SfU1yHN0A14znyPULl6hpzlkq7SKNaFqg9Yhj2hJ3pWSmi9XgjapmMZ6HV8yjigpSNlI9T8eLhc1eRRgZ885eNJ8w3secl5ilCdozV1BoOIk9gDZNY5qgVQcFeTDVxhPmwPhEU41Lq35gCzPtc2oPugVKOp5Gsf7DFBlektobd2yuDtElXxmj1usDJJ6hj0SDHNRTYERTDGHFHDFGSDFGRTYFGHDFGDFSFGWERTFSDFGDFADWERTFGHDFGHFGH35} \end{align} with $r_0 = r(S_0, 0)$. \end{Theorem} \colb \par In the rest of the paper, we omit the superscript $\epsilon$, and write $S$, $u$, and $v$ for $S^\epsilon$, $u^\epsilon$, and $v^{\epsilon}$. The constant $C$ and the function $Q$ denote a generic constant and a positive increasing function, respectively, which depend only on $M_0$, $\tau_0$, and the exterior domain $\Omega$; they may vary from an inequality to an inequality.  The domain of dependence in the norms is understood to be $\Omega$ unless stated otherwise. We write $ 	a\les b $ if there exists a constant $C>0$ such that $ 	a\leq C b $. \par In order to prove Theorem~\ref{T01}, we establish  analytic a~priori estimates of the entropy $S$ and the (modified) velocity $u$.  The a~priori estimate needed to prove Theorem~\ref{T01} is the following. \par \cole \begin{Lemma} \label{L12} Let $M_0, \tau_0>0$. For sufficiently small parameters $\kappa$ and $\bar{\kappa}$ satisfying $0< \kappa \leq \bar{\kappa} \leq 1$, there exist constants $\tau_1$, $\epsilon_0$, $T_0>0$,  and a nonnegative continuous function $Q$ such that for all $\epsilon \in (0,\epsilon_0]$, the norm  \begin{align} M_{\epsilon, \kappa, \bar{\kappa}} (T) = \sup_{t\in [0,T]} (\Vert S(t) \Vert_{A(\tau(t))}  +  \Vert u(t) \Vert_{A(\tau(t))}) \llabel{HBVFanTvabFAVwM51nUH60GvT9fAjTO4MQVzNNAQiwSlSxf2pQ8qvtdjnvupLATIwym4nEYESfMavUgZoyehtoe9RTN15EI1aKJSCnr4MjiYhB0A7vnSAYnZ1cXOI1V7yja0R9jCTwxMUiMI5l2sTXnNRnVi1KczLG3MgJoEktlKoU13tsaqjrHYVzfb1yyxunpbRA56brW45IqhfKo0zj04IcGrHirwyH2tJbFr3leRdcpstvXe2yJlekGVFCe2aD4XPOuImtVoazCKO3uRIm2KFjtm5RGWCvkozi75YWNsbhORnxzRzw99TrFhjhKbfqLAbe2v5n9mD2VpNzlMnntoiFZB2ZjXBhhsK8K6cGiSbRkkwfWeYJXdRBBxyqjEVF5lr3dFrxGlTcsbyAENcqA981IQ4UGpBk0gBeJ6Dn9Jhkne5f518umOuLnIaspzcRfoC0StSy0DF8NNzF2UpPtNG50tqKTk2e51yUbrsznQbeIuiY5qaSGjcXiEl45B5PnyQtnUOMHiskTC2KsWkjha6loMfgZKG3nHph0gnNQ7q0QxsQkgQwKwyhfP5qFWwNaHxSKTA63ClhGBgarujHnGKf46FQtVtSPgEgTeY6fJGmB3qgXxtR8RTCPB18kQajtt6GDrKb1VYLV3RgWIrAyZf69V8VM7jHOb7zLvaXTTVI0ONKMBAHOwOZ7dPkyCgUS74HlnFZMHabr8mlHbQNSwwdomOL6q5wvRexVejvVHkCEdXm3cU54juZSKng8wcj6hR1FnZJbkmgKXJgFm5qZ5SubXvPKDSDHNRTYERTDGHFHDFGSDFGRTYFGHDFGDFSFGWERTFSDFGDFADWERTFGHDFGHFGH184} \end{align} satisfies the estimate \begin{align} M_{\epsilon, \kappa,\bar{\kappa}} (t) \les 1 + \left( t + \kappa + \bar{\kappa} + \epsilon + \tau(0) \right) Q(M_{\epsilon, \kappa, \bar{\kappa}}(t)) ,    \llabel{BOCGf4srh1a5FL0vYfRjJwUm2sfCogRhabxyc0RgavaRbkjzlteRGExbEMMhLZbh3axosCqu7kZ1Pt6Y8zJXtvmvPvAr3LSWDjbVPN7eNu20r8Bw2ivnkzMda93zWWiUBHwQzahUiji2TrXI8v2HNShbTKLeKW83WrQKO4TZm57yzoVYZJytSg2Wx4YafTHAxS7kacIPQJGYdDk0531u2QIKfREWYcMKMUT7fdT9EkIfUJ3pMW59QLFmu02YHJaa2Er6KSIwTBGDJYZwvfSJQby7fdFWdfT9zU27ws5oU5MUTDJzKFNojdXRyBaYybTvnhh2dV77oFFlt4H0RNZjVJ5BJpyIqAOWWcefdR27nGkjmoEFHjanXf1ONEcytoINtD90ONandawDRKi2DJzAqYHGCTB0pzdBa3OotPq1QVFvaYNTVz2sZJ6eyIg2N7PgilKLF9NzcrhuLeCeXwb6cMFExflJSE8Ev9WHgQ1Brp7ROMACwvAnATqGZHwkdHA5fbABXo6EWHsoW6HQYvvjcZgRkOWAbVA0zBfBaWwlIV05Z6E2JQjOeHcZGJuq90ac5Jh9h0rLKfIHtl8tPrtRdqql8TZGUgdNySBHoNrQCsxtgzuGAwHvyNxpMmwKQuJFKjtZr6Y4HdmrCbnF52gA0328aVuzEbplXZd7EJEEC939HQthaMsupTcxVaZ32pPdbPIj2x8AzxjYXSq8LsofqmgSqjm8G4wUbQ28LuAabwI0cFWNfGnzpVzsUeHsL9zoBLlg5jXQXnR0giRmCLErqlDIPYeYXduUSDHNRTYERTDGHFHDFGSDFGRTYFGHDFGDFSFGWERTFSDFGDFADWERTFGHDFGHFGH39} \end{align} for $t\in(0,T_0]$ and $\tau(0) \in (0, \tau_1]$, provided \begin{equation} \KK \geq   Q(M_{\epsilon,\kappa,\bar{\kappa}}(T_0)) \label{SDHNRTYERTDGHFHDFGSDFGRTYFGHDFGDFSFGWERTFSDFGDFADWERTFGHDFGHFGH149} \end{equation} holds, where $\tau$ and $K$ are as in~\eqref{SDHNRTYERTDGHFHDFGSDFGRTYFGHDFGDFSFGWERTFSDFGDFADWERTFGHDFGHFGH99}. \end{Lemma} \colb \par The constant  $\KK$ in \eqref{SDHNRTYERTDGHFHDFGSDFGRTYFGHDFGDFSFGWERTFSDFGDFADWERTFGHDFGHFGH149} depends on $M$, which implies that $K$ eventually depends on $M_0$. In the rest of the paper, we work on an interval of time such that  \begin{equation} T_0\leq \frac{\tau(0)}{2\KK} . \label{SDHNRTYERTDGHFHDFGSDFGRTYFGHDFGDFSFGWERTFSDFGDFADWERTFGHDFGHFGH148} \end{equation} Thus from \eqref{SDHNRTYERTDGHFHDFGSDFGRTYFGHDFGDFSFGWERTFSDFGDFADWERTFGHDFGHFGH99} we have $\tau(0)/2\leq \tau(t) \leq \tau(0)$ for $t\in [0, T_0]$. \par For the proof of Theorem~\ref{T01} given Lemma~\ref{L12}, cf.~\cite{JKL}. Sections~\ref{sec03}--\ref{sec06} are devoted to the proof of Lemma~\ref{L12}, thus completing the proof of Theorem~\ref{T01}. \par By \cite[Theorem~1.1]{A05} the $H^5$ norm of $(p,v,S)$ can be estimated by a constant on a time interval $[0,T_0]$, where $T_0$ only depends on the $H^5$ norm of the initial data. Thus we may assume \begin{align} \sup_{(j,k,i)\in \mathbb{N}_0^3, 0\leq j+k+i \leq 5} \Vert \partial_x^j \TT^k (\epsilon \partial_t)^i (p,v,S)(t) \Vert_{L^2} \les 1 \comma t\in [0, T_0] \commaone \epsilon \in (0,1] . \label{SDHNRTYERTDGHFHDFGSDFGRTYFGHDFGDFSFGWERTFSDFGDFADWERTFGHDFGHFGH24} \end{align} In particular, if $F$ is a smooth function of $u$ and $S$, then there exists a constant $C>0$ depending on the function $F$ such that \begin{align} \Vert F(\epsilon u(t),S(t)) \Vert_{L^\infty} \les 1 \comma t\in [0,T_0] \commaone \epsilon \in (0,1] . \llabel{JE0BsbkKbjpdcPLiek8NWrIjsfapHh4GYvMFbA67qyex7sHgHG3GlW0y1WD35mIo5gEUbObrbknjgUQyko7g2yrEOfovQfAk6UVDHGl7GV3LvQmradEUOJpuuztBBnrmefilt1sGSf5O0aw2Dc0hRaHGalEqIpfgPyNQoLHp2LAIUp77FygrjC8qBbuxBkYX8NTmUvyT7YnBgv5K7vq5NefB5ye4TMuCfmE2JF7hgqwI7dmNx2CqZuLFthzIlB1sjKA8WGDKcDKvabk9yp28TFP0rg0iA9CBD36c8HLkZnO2S6ZoafvLXb8gopYa085EMRbAbQjGturIXlTE0Gz0tYSVUseCjDvrQ2bvfiIJCdfCAcWyIO7mlycs5RjioIZt7qyB7pL9pyG8XDTzJxHs0yhVVAr8ZQRqsZCHHADFTwvJHeHOGvLJHuTfNa5j12ZkTvGqOyS8826D2rj7rHDTLN7Ggmt9Mzcygwxnj4JJeQb7eMmwRnSuZLU8qUNDLrdgC70bhEPgpb7zk5a32N1IbJhf8XvGRmUFdvIUkwPFbidJPLlNGe1RQRsK2dVNPM7A3YhdhB1R6N5MJi5S4R498lwY9I8RHxQKLlAk8W3Ts7WFUoNwI9KWnztPxrZLvNwZ28EYOnoufxz6ip9aSWnNQASriwYC1sOtSqXzot8k4KOz78LG6GMNCExoMh9wl5vbsmnnq6Hg6WToJun74JxyNBXyVpvxNB0N8wymK3reReEzFxbK92xELs950SNgLmviRC1bFHjDCke3SgtUdC4cONb4EF24D1VDSDHNRTYERTDGHFHDFGSDFGRTYFGHDFGDFSFGWERTFSDFGDFADWERTFGHDFGHFGH210} \end{align} In the rest of the paper, we work on the time interval $[0, T]$ where $0<T\leq T_0$. \par \startnewsection{Derivative reductions for the velocity}{sec03} Here we introduce the normal and tangential derivative reduction schemes for the velocity.  Throughout this section, we assume that $\Omega$ is an exterior domain with analytic boundary and the vector field $v$ satisfies  	\begin{align} 	v\cdot \nu |_{\partial \Omega}  	=  	0, 	\label{SDHNRTYERTDGHFHDFGSDFGRTYFGHDFGDFSFGWERTFSDFGDFADWERTFGHDFGHFGH301} 	\end{align} where $\nu$ is the unit outward normal vector to the boundary $\partial \Omega$. \par \subsection{Normal derivative reduction}  We start with the normal derivative reduction.  \par \cole \begin{Lemma} \label{L01} For $j\geq 2$, we have \begin{align} \begin{split} \Vert \partial_x^j \TT^k (\epsilon \partial_t)^i v\Vert_{L^2} & \lesssim \Vert \partial_x^{j-1} \TT^k (\epsilon \partial_t)^i \dive v \Vert_{L^2} + \Vert \partial_x^{j-1} \TT^k (\epsilon \partial_t)^i \curl v\Vert_{L^2} + \Vert \partial_x^{j-2} \TT^{k+1} (\epsilon \partial_t)^i v\Vert_{H^1} \\&\indeq\indeq + \Vert \partial_x^{j-2} \TT^k (\epsilon \partial_t)^i v \Vert_{L^2} + \Vert \partial_x^{j-1} [\TT^k, \dive] (\epsilon \partial_t)^i v\Vert_{L^2} \\&\indeq\indeq + \Vert \partial_x^{j-1} [\TT^k, \curl] (\epsilon \partial_t)^i v\Vert_{L^2} + \Vert [\TT, \partial_x^{j-2}] \TT^k (\epsilon \partial_t)^i v\Vert_{H^1}. \end{split} \label{SDHNRTYERTDGHFHDFGSDFGRTYFGHDFGDFSFGWERTFSDFGDFADWERTFGHDFGHFGH53} \end{align} For $j=1$ and $k=0$, we have \begin{align} \begin{split} \Vert \partial_x (\epsilon \partial_t)^i v \Vert_{L^2} & \lesssim \Vert  (\epsilon \partial_t)^i \dive v \Vert_{L^2}  + \Vert (\epsilon \partial_t)^i \curl v \Vert_{L^2}  + \Vert (\epsilon \partial_t)^i v \Vert_{L^2}  . \end{split} \label{SDHNRTYERTDGHFHDFGSDFGRTYFGHDFGDFSFGWERTFSDFGDFADWERTFGHDFGHFGH72} \end{align} \end{Lemma} \colb \par In order to prove this statement, we use a form of  $H^2$ regularity for the Laplacian. \par \cole \begin{Lemma}  \label{L14} Let $\Omega$ be an exterior domain in $\mathbb{R}^3$ with a smooth boundary $\partial \Omega$.  Then  \begin{align} \Vert v \Vert_{H^2(\Omega)} & \les \Vert \Delta v \Vert_{L^2(\Omega)} + \Vert \TT v \Vert_{H^1(\Omega)}  + \Vert v \Vert_{L^2(\Omega)}  , \label{SDHNRTYERTDGHFHDFGSDFGRTYFGHDFGDFSFGWERTFSDFGDFADWERTFGHDFGHFGH20} \end{align} for all $v \in H^2 (\Omega)$. \end{Lemma} \colb \par In the proof of Lemma~\ref{L01}, we also need the following classical div-curl estimate due to  Bourguignon and Brezis. \par \cole \begin{Lemma} \label{L13} (\cite{BoB}) Let $\Omega$ be an exterior domain in $\mathbb{R}^3$ with analytic boundary $\partial \Omega$ and outward unit normal vector $\nu$. Then there exists a constant $C>0$ such that \begin{align} \Vert \nabla v \Vert_{L^2(\Omega)}  & \les \Vert \dive v \Vert_{L^2(\Omega)} +  \Vert \curl v \Vert_{L^2(\Omega)} + \Vert v\cdot \nu \Vert_{H^{1/2}(\partial \Omega)} + \Vert v \Vert_{L^2(\Omega)}  \label{SDHNRTYERTDGHFHDFGSDFGRTYFGHDFGDFSFGWERTFSDFGDFADWERTFGHDFGHFGH07} , \end{align} for all $v \in H^1(\Omega)$.  \end{Lemma} \colb \par \begin{proof}[Proof of Lemma~\ref{L14}] Using the $H^2$ regularity for the Laplace equation \begin{align} \begin{split} \Vert v \Vert_{H^2(\Omega)}  \lesssim \Vert \Delta v \Vert_{L^2(\Omega)}  +  \Vert v\Vert_{H^{3/2}(\partial \Omega)} , \end{split}    \llabel{BHlWATyswjyDOWibTHqXt3aG6mkfGJVWv40lexPnIcy5ckRMD3owVBdxQm6CvLaAgxiJtEsSlZFwDoYP2nRYbCdXRz5HboVTU8NPgNViWeXGVQZ7bjOy1LRy9faj9n2iE1S0mci0YD3HgUxzLatb92MhCpZKLJqHTSFRMn3KVkpcFLUcF0X66ivdq01cVqkoQqu1u2Cpip5EV7AgMORcfZjLx7Lcv9lXn6rS8WeK3zTLDPB61JVWwMiKEuUZZ4qiK1iQ8N0832TS4eLW4zeUyonzTSofna74RQVKiu9W3kEa3gH8xdiOhAcHsIQCsEt0Qi2IHw9vq9rNPlh1y3wORqrJcxU4i55ZHTOoGP0zEqlB3lkwGGRn7TOoKfGZu5BczGKFeoyIBtjNb8xfQEKduOnJVOZh8PUVaRonXBkIjBT9WWor7A3WfXxA2f2VlXZS1Ttsab4n6R3BKX0XJTmlkVtcWTMCsiFVyjfcrzeJk5MBxwR7zzVOnjlLzUz5uLeqWjDul7OnYICGG9iRybTsYJXfrRnub3p16JBQd0zQOkKZK6DeVgpXRceOExLY3WKrXYyIe7dqMqanCCTjFW71LQ89mQw1gAswnYSMeWlHz7ud7xBwxF3m8usa66yr0nSdsYwuqwXdD0fRjFpeLOe0rcsIuMGrSOqREW5plybq3rFrk7YmLURUSSVYGruD6ksnLXBkvVS2q0ljMPpIL27QdZMUPbaOoLqt3bhn6RX9hPAdQRp9PI4fBkJ8uILIArpTl4E6jrUYwuFXiFYaDVvrDb2zVpvGg6zFSDHNRTYERTDGHFHDFGSDFGRTYFGHDFGDFSFGWERTFSDFGDFADWERTFGHDFGHFGH40} \end{align} combined with the trace theorem, we arrive at \begin{align} \Vert v \Vert_{H^2 (\Omega)} \lesssim \Vert \Delta v \Vert_{L^2(\Omega)} + \Vert \TT v \Vert_{H^1(\Omega)} + \Vert v \Vert_{L^2 (\Omega)},    \llabel{YojSbMBhr4pW8OwDNUao2mhDTScei90KrsmwaBnNUsHe6RpIq1hXFNPm0iVsnGkbCJr8Vmegl416tU2nnollOtcFUM7c4GC8ClaslJ0N8XfCuRaR2sYefjVriJNj1f2tyvqJyQNX1FYmTl5N17tkbBTPuF471AH0Fo71REILJp4VsqiWTTtkAd5RkkJH3RiRNKePesR0xqFqnQjGUIniVgLGCl2He7kmqhEV4PFdCdGpEP9nBmcvZ0pLYGidfn65qEuDfMz2vcq4DMzN6mBFRtQP0yDDFxjuZiZPE3Jj4hVc2zrrcROnFPeOP1pZgnsHAMRK4ETNF23KtfGem2kr5gf5u8NcuwfJCav6SvQ2n18P8RcIkmMSD0wrVR1PYx7kEkZJsJ7Wb6XIWDE0UnqtZPAqEETS3EqNNf38DEk6NhXV9c3sevM32WACSj3eNXuq9GhPOPChd7v1T6gqRinehWk8wLoaawHVvbU4902yObCT6zm2aNf8xUwPOilrR3v8RcNWEk7EvIAI8okPAYxPiUlZ4mwzsJo6ruPmYN6tylDEeeoTmlBKmnVuBB7HnU7qKn353SndtoL82gDifcmjLhHx3gi0akymhuaFTzRnMibFGU5W5x6510NKi85u8JTLYcbfOMn0auD0tvNHwSAWzE3HWcYTId2HhXMLiGiykAjHCnRX4uJJlctQ3yLoqi9ju7Kj84EFU49udeA93xZfZCBW4bSKpycf6nncmvnhKb0HjuKWp6b88pGC3U7kmCO1eY8jvEbu59zmGZsZh93NwvJYbkEgDpJBSDHNRTYERTDGHFHDFGSDFGRTYFGHDFGDFSFGWERTFSDFGDFADWERTFGHDFGHFGH43} \end{align} concluding the proof of \eqref{SDHNRTYERTDGHFHDFGSDFGRTYFGHDFGDFSFGWERTFSDFGDFADWERTFGHDFGHFGH20}. \end{proof} \par \begin{proof}[Proof of Lemma~\ref{L01}] Let $j\geq 2$. Appealing to the $H^2$ regularity \eqref{SDHNRTYERTDGHFHDFGSDFGRTYFGHDFGDFSFGWERTFSDFGDFADWERTFGHDFGHFGH20}, we obtain \begin{align} \begin{split} \Vert \partial_x^j \TT^k (\epsilon \partial_t)^i v\Vert_{L^2(\Omega)} & \les \Vert \Delta\partial_x^{j-2} \TT^k (\epsilon \partial_t)^i v\Vert_{L^2(\Omega)} +  \Vert\TT \partial_x^{j-2} \TT^k (\epsilon \partial_t)^i v\Vert_{H^{1}(\Omega)} +  \Vert \partial_x^{j-2} \TT^k (\epsilon \partial_t)^i v\Vert_{L^2(\Omega)} , \end{split} \llabel{jgQeQUH9kCaz6ZGpcpgrHr79IeQvTIdp35mwWmafRgjDvXS7aFgmNIWmjvopqUuxFrBYmoa45jqkRgTBPPKLgoMLjiwIZ2I4F91C6x9aeW7Tq9CeM62kef7MUbovxWyxgIDcL8Xszu2pZTcbjaK0fKzEyznV0WFYxbFOZJYzBCXtQ4uxU96TnN0CGBhWEFZr60rIgw2f9x0fW3kUB4AOfctvL5I0ANOLdw7h8zK12STKy2ZdewoXYPZLVVvtraCxAJmN7MrmIarJtfTddDWE9At6mhMPCVNUOOSZYtGkPvxpsGeRguDvtWTHMHf3Vyr6W3xvcpi0z2wfwQ1DL1wHedTqXlyojGIQAdEEKv7Tak7cAilRfvrlm82NjNg9KDSvNoQiNhng2tnBSVwd8P4o3oLqrzPNHZmkQItfj61TcOQPJblsBYq3NulNfrConZ6kZ2VbZ0psQAaUCiMaoRpFWfviTxmeyzmc5QsEl1PNOZ4xotciInwc6IFbpwsMeXxy8lJ4A6OV0qRzrSt3PMbvRgOS5obkaFU9pOdMPdjFz1KRXRKDVUjveW3d9shi3jzKBTqZkeSXqbzboWTc5yRRMoBYQPCaeZ23HWk9xfdxJYxHYuNMNGY4XLVZoPUQxJAliDHOKycMAcTpGHIktjlIV25YYoRC74thSsJClD76yxM6BRhgfS0UH4wXVF0x1M6IbemsTKSWlsG9pk95kZSdHU31c5BpQeFx5za7hWPlLjDYdKH1pOkMo1Tvhxxz5FLLu71DUNeUXtDFC7CZ2473sjERebaYt2sESDHNRTYERTDGHFHDFGSDFGRTYFGHDFGDFSFGWERTFSDFGDFADWERTFGHDFGHFGH06} \end{align} from where we obtain by using the vector calculus identity $ \Delta v  =  \nabla \dive v  - \curl \curl v $ the inequality \begin{align} \begin{split} \Vert \partial_x^j \TT^k (\epsilon \partial_t)^i v\Vert_{L^2(\Omega)} & \les \Vert \dive \partial_x^{j-2} \TT^k (\epsilon \partial_t)^i v\Vert_{\dot{H}^1(\Omega)} + \Vert \curl \partial_x^{j-2} \TT^k (\epsilon \partial_t)^i v\Vert_{\dot{H}^1(\Omega)}\\ &\indeq\indeq + \Vert \TT \partial_x^{j-2} \TT^k (\epsilon \partial_t)^i v\Vert_{H^1(\Omega)} + \Vert \partial_x^{j-2} \TT^k (\epsilon \partial_t)^i v \Vert_{L^2(\Omega)}\\ & \lesssim \Vert \partial_x^{j-2} \TT^k (\epsilon \partial_t)^i \dive v \Vert_{\dot{H}^1(\Omega)} + \Vert \partial_x^{j-2} \TT^k (\epsilon \partial_t)^i \curl v\Vert_{\dot{H}^1(\Omega)} \\ &\indeq\indeq + \Vert \partial_x^{j-2} \TT^{k+1} (\epsilon \partial_t)^i v\Vert_{H^1(\Omega)} + \Vert \partial_x^{j-2} \TT^k (\epsilon \partial_t)^i v \Vert_{L^2(\Omega)} \\ &\indeq\indeq + \Vert \partial_x^{j-1} [\TT^k, \dive] (\epsilon \partial_t)^i v\Vert_{L^2(\Omega)} + \Vert \partial_x^{j-1} [\TT^k, \curl] (\epsilon \partial_t)^i v\Vert_{L^2(\Omega)} \\ &\indeq\indeq + \Vert [\TT, \partial_x^{j-2}] \TT^k (\epsilon \partial_t)^i v\Vert_{H^1(\Omega)} , \end{split} \llabel{pV9wDJ8RGUqQmboXwJnHKFMpsXBvAsX8NYRZMwmZQctltsqofi8wxn6IW8jc68ANBwz8f4gWowkmZPWlwfKpM1fpdo0yTRIKHMDgTl3BUBWr6vHUzFZbqxnwKkdmJ3lXzIwkw7JkuJcCkgvFZ3lSo0ljVKu9Syby46zDjM6RXZIDPpHqEfkHt9SVnVtWdyYNwdmMm7SPwmqhO6FX8tzwYaMvjzpBSNJ1z368900v2i4y2wQjZhwwFUjq0UNmk8J8dOOG3QlDzp8AWpruu4D9VRlpVVzQQg1caEqevP0sFPHcwtKI3Z6nY79iQabga0i9mRVGbvlTAgV6PUV8EupPQ6xvGbcn7dQjV7Ckw57NPWUy9XnwF9elebZ8UYJDx3xCBYCIdPCE2D8eP90u49NY9Jxx9RI4Fea0QCjs5TLodJFphykczBwoe97PohTql1LMs37cKhsHO5jZxqpkHtLbFDnvfTxjiykLVhpwMqobqDM9A0f1n4i5SBc6trqVXwgQBEgH8lISLPLO52EUvi1myxknL0RBebO2YWw8Jhfo1lHlUMiesstdWw4aSWrYvOsn5Wn3wfwzHRHxFg0hKFuNVhjzXbg56HJ9VtUwalOXfT8oiFY1CsUCgCETCIvLR0AgThCs9TaZl6ver8hRtedkAUrkInSbcI8nyEjZsVOSztBbh7WjBgfaAFt4J6CTUCU543rbavpOMyelWYWhVBRGow5JRh2nMfUcoBkBXUQ7UlO5rYfHDMceWou3RoFWtbaKh70oHBZn7unRpRh3SIpp0Btqk5vhXCU9SDHNRTYERTDGHFHDFGSDFGRTYFGHDFGDFSFGWERTFSDFGDFADWERTFGHDFGHFGH04} \end{align} \colb and \eqref{SDHNRTYERTDGHFHDFGSDFGRTYFGHDFGDFSFGWERTFSDFGDFADWERTFGHDFGHFGH53} follows. On the other hand, the inequality \eqref{SDHNRTYERTDGHFHDFGSDFGRTYFGHDFGDFSFGWERTFSDFGDFADWERTFGHDFGHFGH72} follows from  the elliptic regularity \eqref{SDHNRTYERTDGHFHDFGSDFGRTYFGHDFGDFSFGWERTFSDFGDFADWERTFGHDFGHFGH07} and the boundary condition~\eqref{SDHNRTYERTDGHFHDFGSDFGRTYFGHDFGDFSFGWERTFSDFGDFADWERTFGHDFGHFGH301}. \end{proof} \par \subsection{Tangential derivative reduction}  The following lemma allows us to reduce the number of tangential derivatives.  \par \cole \begin{Lemma} \label{L02} For $j=1$ and $k\geq 1$, we have \begin{align} \begin{split} \Vert \partial_x \TT^k (\epsilon \partial_t)^i v \Vert_{L^2(\Omega)} & \lesssim \Vert \TT^{k} (\epsilon \partial_t)^i \dive v \Vert_{L^2(\Omega)} + \Vert \TT^{k} (\epsilon \partial_t)^i \curl v \Vert_{L^2(\Omega)} + \Vert \TT^{k} (\epsilon \partial_t)^i v \Vert_{L^2(\Omega)} \\ &\indeq + \sum_{l=1}^{k}  \binom{k}{l}   \bigl( \Vert \partial_x \TT^{k-l} (\epsilon \partial_t)^i v \Vert_{L^2(\Omega)} \Vert \TT^l \nu \Vert_{L^\infty (\Omega_{\delta_0})} \\&\indeq + \Vert \TT^{k-l} (\epsilon \partial_t)^i v \Vert_{L^2(\Omega)} \Vert \partial_x \TT^{l} \nu \Vert_{L^\infty(\Omega_{\delta_0})} + \Vert \TT^{k-l} (\epsilon \partial_t)^i v \Vert_{L^2(\Omega)} \Vert \TT^{l} \nu \Vert_{L^\infty(\Omega_{\delta_0})} \bigr) \\ &\indeq + \Vert [\TT^{k}, \dive] (\epsilon \partial_t)^i v \Vert_{L^2(\Omega)} + \Vert [\TT^{k}, \curl] (\epsilon \partial_t)^i v \Vert_{L^2(\Omega)} . \end{split} \label{SDHNRTYERTDGHFHDFGSDFGRTYFGHDFGDFSFGWERTFSDFGDFADWERTFGHDFGHFGH14} \end{align}  \end{Lemma} \colb \par Note that when $j=0$ and $k\geq1$, we may simply use \begin{align} \begin{split} \Vert \TT^k (\epsilon \partial_t)^i v \Vert_{L^2(\Omega)} \les  \Vert \partial_{x}\TT^{k-1} (\epsilon \partial_t)^i v \Vert_{L^2(\Omega)} \end{split} \label{SDHNRTYERTDGHFHDFGSDFGRTYFGHDFGDFSFGWERTFSDFGDFADWERTFGHDFGHFGH08} \end{align} and apply \eqref{SDHNRTYERTDGHFHDFGSDFGRTYFGHDFGDFSFGWERTFSDFGDFADWERTFGHDFGHFGH14} with the reduced $k$. \par \begin{proof}[Proof of Lemma~\ref{L02}] For $j=1$ and $k \geq 1$, we appeal to \eqref{SDHNRTYERTDGHFHDFGSDFGRTYFGHDFGDFSFGWERTFSDFGDFADWERTFGHDFGHFGH07} and obtain \begin{align} \begin{split} \Vert \partial_x \TT^k (\epsilon \partial_t)^i v \Vert_{L^2(\Omega)} & \lesssim \Vert \dive \TT^k (\epsilon \partial_t)^i v \Vert_{L^2(\Omega)} + \Vert \curl \TT^k (\epsilon \partial_t)^i v \Vert_{L^2(\Omega)} \\ &\indeq + \Vert \TT^k (\epsilon \partial_t)^i v \cdot \nu \Vert_{H^{1/2}(\partial \Omega)} + \Vert \TT^k (\epsilon \partial_t)^i v \Vert_{L^2(\Omega)} . \end{split} \label{SDHNRTYERTDGHFHDFGSDFGRTYFGHDFGDFSFGWERTFSDFGDFADWERTFGHDFGHFGH09} \end{align} Now, note that \begin{align} \begin{split} \TT^k (\epsilon \partial_t)^i v \cdot \nu = \TT^k (\epsilon \partial_t)^i (v \cdot \nu) - \sum_{l=1}^k  \binom{k}{l} \TT^{k-l} (\epsilon \partial_t)^i v \cdot \TT^l \nu , \end{split} \llabel{BHJFx7qPxB55a7RkOyHmSh5vwrDqt0nF7toPJUGqHfY5uAt5kQLP6ppnRjMHk3HGqZ0OBugFFxSnASHBI7agVfqwfgaAleH9DMnXQQTAAQM8qz9trz86VR2gOMMVuMgf6tGLZWEKqvkMEOgUzMxgN4CbQ8fWY9Tk73Gg90jy9dJbOvddVZmqJjb5qQ5BSFfl2tNPRC86tI0PIdLDUqXKO1ulgXjPVlfDFkFh42W0jwkkH8dxIkjy6GDgeM9mbTYtUS4ltyAVuor6w7InwCh6GG9Km3YozbuVqtsXTNZaqmwkzoKxE9O0QBQXhxN5Lqr6x7SxmvRwTSBGJY5uo5wSNGp3hCcfQNafXWjxeAFyCxUfM8c0kKkwgpsvwVe4tFsGUIzoWFYfnQAUT9xclTfimLCJRXFAmHe7VbYOaFBPjjeF6xI3CzOVvimZ32pt5uveTrhU6y8wjwAyIU3G15HMybdauGckOFnq6a5HaR4DOojrNAjdhSmhOtphQpc9jXX2u5rwPHzW032fi2bz160Ka4FDjd1yVFSMTzSvF1YkRzdzbYbI0qjKMNXBFtXoCZdj9jD5AdSrNBdunlTDIaA4UjYSx6DK1X16i3yiQuq4zooHvHqNgT2VkWGBVA4qeo8HH70FflAqTDBKi461GvMgzd7WriqtFq24GYcyifYkWHv7EI0aq5JKlfNDCNmWom3VyXJsNt4WP8yGgAoATOkVWZ4ODLtkza9PadGCGQ2FCH6EQppksxFKMWAfY0JdaSYgo7hhGwHttbb4z5qrcdc9CnAmxqY6m8uGSDHNRTYERTDGHFHDFGSDFGRTYFGHDFGDFSFGWERTFSDFGDFADWERTFGHDFGHFGH13} \end{align} with the first term on the right side vanishing by~\eqref{SDHNRTYERTDGHFHDFGSDFGRTYFGHDFGDFSFGWERTFSDFGDFADWERTFGHDFGHFGH301}. Thus we get \begin{align} \begin{split} \Vert \TT^k (\epsilon \partial_t)^i v \cdot \nu\Vert_{H^{1/2}(\partial \Omega)} \leq \sum_{l=1}^k \binom{k}{l} \Vert \TT^{k-l} (\epsilon \partial_t)^i v \cdot \TT^l \nu \Vert_{H^{1/2} (\partial \Omega)} . \llabel{f7DZQ6FBUPPiOxgsQ0CZlPYPBa75OiV6tZOBpfYuNcbj4VUpbTKXZRJf36EA0LDgAdfdOpSbg1ynCPUVoRWxeWQMKSmuh3JHqX15APJJX2v0W6lm0llC8hlss1NLWaNhRBAqfIuzkx2sp01oDrYsRywFrNbz1hGpq99FwUzlfcQkTsbCvGIIgmfHhTrM1ItDgCMzYttQRjzFxXIgI7FMAp1kllwJsGodXAT2PgoIp9VonFkwZVQifq9ClAQ4YBwFR4nCyRAg84MLJunx8uKTF3FzlGEQtl32y174wLXZm62xX5xGoaCHvgZFEmyDIzj3q10RZrsswByA2WlOADDDQVin8PTFLGmwi6pgRZQ6A5TLlmnFVtNiJbnUkLyvq9zSBP6eJJq7P6RFaim6KXPWaxm6W7fM83uKD6kNj7vhg4ppZ4ObMaSaPH0oqxABG8vqrqT6QiRGHBCCN1ZblTY4zq8lFqLCkghxDUuZw7MXCD4psZcEX9RlCwf0CCG8bgFtiUv3mQeLWJoyFkv6hcSnMmKbiQukLFpYAqo5Fjf9RRRtqS6XWVoIYVDMla5c7cWKJLUqcvtiIOeVCU7xJdC5W5bk3fQbyZjtUDmegbgI179dlU3u3cvWoAIowbEZ0xP2FBMSwazV1XfzVi97mmy5sTJK0hz9O6pDaGctytmHTDYxTUBALNvQefRQuF2OyokVsLJwdqgDhTTJeR7CuPczNLVj1HKml8mwLFr8Gz66n4uA9YTt9oiJGclm0EckA9zkElOB9Js7Gfwhqyglc2RQ9d52aYQvC8ArSDHNRTYERTDGHFHDFGSDFGRTYFGHDFGDFSFGWERTFSDFGDFADWERTFGHDFGHFGH95} \end{split} \end{align} Extending the unit normal vector to $\Omega_{\delta_0}$ as in Section~\ref{sec02}, and using the trace theorem, we obtain \begin{align} \begin{split} \Vert \TT^{k-l} (\epsilon \partial_t)^i v \cdot \TT^l \nu \Vert_{H^{1/2} (\partial \Omega)} \les \Vert \TT^{k-l} (\epsilon \partial_t)^i v \cdot \TT^l \nu \Vert_{H^1 (\Omega_{\delta_0})} . \end{split}    \llabel{K7aCLmENPYd27XImGC6L9gOfyL05HMtgR65lBCsWGwFKGBIQiIRBiT95N78wncbk7EFeiBRB216SiHoHJSkNgxqupJmZ1pxEbWcwiJX5NfiYPGD6uWsXTP94uaFVDZuhJH2d0PLOY243xMK47VP6FTyT35zpLxRC6tN89as3ku8eGrdMKWoMIU946FBjksOTe0UxZD4avbTw5mQ3Ry9AfJFjPgvLFKz0olfZdj3O07EavpWfbM3rBGSyOiuxpI4o82JJ42X1GIux8QFh3PhRtY9vjiSL6x76W9y2Zz3YASGRMp7kDhrgma8fWGG0qKLsO5oQr42t1jP1crM2fClRbETdqra5lVG1lKitbXqbdPKcaUV0lv4Lalo8VTXclaUqh5GWCzAnRnlNNcmwaF8ErbwX32rjiHleb4gXSjLROJgG2yb8OCAxN4uy4RsLQjD7U7enwcYCnZxiKdju74vpjBKKjRRl36kXXzvnX2JrD8aPDUWGstgb8CTWYnHRs6y6JCp8Lx1jzCI1mtG26y5zrJ1nFhX67wCzqF8uZQIS0dnYxPeXDyjBz1aYwzDXaxaMIZzJ3C3QRrahpw8sWLxrAsSqZP5Wvv1QF7JPAVQwuWu69YLwNHUPJ0wjs7RSiVaPrEGgxYaVmSk3Yo1wLn0q0PVeXrzoCIH7vxq5ztOmq6mp4drApdzhwSOlRPDpsClr8FoZUG7vDUYhbScJ6gJb8Q8emG2JG9Oja83owYwjozLa3DB500siGjEHolPuqe4p7T1kQJmU6cHnOo29oroOzTa3j31n8mDLSDHNRTYERTDGHFHDFGSDFGRTYFGHDFGDFSFGWERTFSDFGDFADWERTFGHDFGHFGH44} \end{align} By the Leibniz rule and H\"older's inequality, we conclude that \begin{align} \begin{split} & \Vert \TT^{k-j} (\epsilon \partial_t)^i v \cdot \TT^j \nu \Vert_{H^1(\Omega_{\delta_0})} \\&\indeq \leq \Vert \partial_x \TT^{k-j} (\epsilon \partial_t)^i v \Vert_{L^2(\Omega)} \Vert \TT^j \nu \Vert_{L^\infty(\Omega_{\delta_0})} + \Vert \TT^{k-j} (\epsilon \partial_t)^i v \Vert_{L^2(\Omega)} \Vert \partial_x \TT^j \nu \Vert_{L^\infty (\Omega_{\delta_0})} \\&\indeq\indeq + \Vert \TT^{k-j} (\epsilon \partial_t)^i v \Vert_{L^2(\Omega)} \Vert \TT^j \nu \Vert_{L^\infty(\Omega_{\delta_0})} . \label{SDHNRTYERTDGHFHDFGSDFGRTYFGHDFGDFSFGWERTFSDFGDFADWERTFGHDFGHFGH96} \end{split} \end{align}   Combining \eqref{SDHNRTYERTDGHFHDFGSDFGRTYFGHDFGDFSFGWERTFSDFGDFADWERTFGHDFGHFGH09}--\eqref{SDHNRTYERTDGHFHDFGSDFGRTYFGHDFGDFSFGWERTFSDFGDFADWERTFGHDFGHFGH96}, we then obtain \eqref{SDHNRTYERTDGHFHDFGSDFGRTYFGHDFGDFSFGWERTFSDFGDFADWERTFGHDFGHFGH14}. \end{proof} \par \subsection{Commutator estimates} Here we recall a Leibniz formula for $k$-folded commutators. Given two linear differential operators $Y$, $Z$, the adjoint operator ad $Y(Z)$ is defined as \begin{align*} \ad Y(Z) = [Y, Z] = YZ - ZY. \end{align*} Recall a Leibniz-type formula  \begin{align} [\TT^k, Z] = \sum_{m=1}^k \binom{k}{m} ((\ad \TT)^m (Z))  \TT^{k-m} \comma k\in{\mathbb N} \label{SDHNRTYERTDGHFHDFGSDFGRTYFGHDFGDFSFGWERTFSDFGDFADWERTFGHDFGHFGH16} \end{align}  from \cite[Lemma~3.4]{CKV}, which holds for any differential operator~$Z$. From \cite[Lemma~5.3]{Ko01}, we also recall an analytic estimate for the adjoint operator. \par \cole \begin{Lemma}[\cite{Ko01}] \label{L04} Let $Y_1, \ldots, Y_m$ and $Y_0$ be analytic vector fields defined on $\Omega \subseteq \mathbb{R}^3$ such that \begin{align*} Y_n = \sum_{i=1}^3 a_n^i \partial_i \comma n=0, 1, \ldots, m , \end{align*} where  \begin{align*} \max_{\vert \alpha \vert = k} \vert \partial^\alpha a_n^i \vert  \lesssim k! K_1^k \comma i=1, 2, 3 \commaone n=0, 1,\ldots, m \commaone k\in \mathbb{N}_0 , \end{align*} for some $K_1 \geq 1$. Then there exists $\bar{K}_1, \bar{K}_2 \geq 1$ such that \begin{align*} (\ad Y_m \dots \ad Y_1) (Y_0) = \sum_{i=1}^3 b_m^i \partial_i, \end{align*}   where  \begin{align*} \max_{\vert \alpha \vert = k} \vert \partial^\alpha b_m^i \vert  \lesssim (k+m)! \bar{K}_1^k \bar{K}_2^m , \end{align*} for $i=1,2,3$   and $k\in{\mathbb N}_0$. \end{Lemma}   \par \colb   In the following lemma, we derive  commutator estimates for various spatial derivative operators. \par \cole \begin{Lemma}  There exists a constant $\bar{K} \geq1$ such that  \label{L03} \begin{align} \Vert \partial_x^j [\TT^k, \dive] (\epsilon \partial_t)^i v \Vert_{L^2} &\lesssim \sum_{k'=1}^k \sum_{j'=0}^j \binom{k}{k'} \binom{j}{j'} (j' + k')! \bar{K}^{j'+k'} \Vert \partial_x^{j-j'+1} \TT^{k-k'} (\epsilon \partial_t)^i v \Vert_{L^2}, \label{SDHNRTYERTDGHFHDFGSDFGRTYFGHDFGDFSFGWERTFSDFGDFADWERTFGHDFGHFGH18} \\ \Vert \partial_x^j [\TT^k, \nabla] (\epsilon \partial_t)^i v  \Vert_{L^2} &\lesssim \sum_{k'=1}^k \sum_{j'=0}^j \binom{k}{k'} \binom{j}{j'} (j' + k')! \bar{K}^{j'+k'} \Vert \partial_x^{j-j'+1} \TT^{k-k'} (\epsilon \partial_t)^i v \Vert_{L^2} \label{SDHNRTYERTDGHFHDFGSDFGRTYFGHDFGDFSFGWERTFSDFGDFADWERTFGHDFGHFGH19} , \\ \Vert \partial_x^j [\TT^k, \curl] (\epsilon \partial_t)^i v  \Vert_{L^2} &\lesssim \sum_{k'=1}^k \sum_{j'=0}^j \binom{k}{k'} \binom{j}{j'} (j' + k')! \bar{K}^{j'+k'} \Vert \partial_x^{j-j'+1} \TT^{k-k'} (\epsilon \partial_t)^i v \Vert_{L^2} , \label{SDHNRTYERTDGHFHDFGSDFGRTYFGHDFGDFSFGWERTFSDFGDFADWERTFGHDFGHFGH50} \\ \Vert \partial_x^j [\partial_x^k, \TT] \TT^l (\epsilon \partial_t)^i v \Vert_{L^2} &\lesssim \sum_{k'=1}^k \sum_{j'=0}^j \binom{k}{k'} \binom{j}{j'} (j' + k')! \bar{K}^{j'+k'} \Vert \partial_x^{k-k'+j-j'+1} \TT^{l} (\epsilon \partial_t)^i v \Vert_{L^2}, \label{SDHNRTYERTDGHFHDFGSDFGRTYFGHDFGDFSFGWERTFSDFGDFADWERTFGHDFGHFGH22} \end{align} for $i, j, l\in \mathbb{N}_0$ and $k \in \mathbb{N}$. \end{Lemma} \par \begin{proof}[Proof of Lemma~\ref{L03}] By \eqref{SDHNRTYERTDGHFHDFGSDFGRTYFGHDFGDFSFGWERTFSDFGDFADWERTFGHDFGHFGH16}, we have the expansion 	\begin{align} 	[\TT^k, \dive]          v 	=  	\sum_{s=1}^3 	[\TT^k, \partial_s]          v_s 	=  	\sum_{s=1}^3 \sum_{k'=1}^k  	\binom{k}{k'} ((\ad\TT)^{k'} (\partial_s)) \TT^{k-k'}         v_s 	.    \llabel{7CIvCpKZUs0jVrb7vHIH7NTtYY7JKvVdGLhA1ONCWoQW1fvjmlH7lSlIm8T1QSdUWhTiMPKDZmm4V7ofRW1dnlqg0Ah1QRjdtKZVzEBNE1eXiRRSLLQPESEDeXbiMMFfxC5FI1zviyNsYHPsGxfGiIuhDPDi0OIHuBTTHOCHyCTkABxuCjgOZs965wfeFwvfRpNLLT3EvgKgkO9jyyvotRRlpDTdn9H5ZnqwWr4OUkIlxtsk0RZdODnsoYid6ctgwwQrxQk1S8ajpPiZJlp5pIAT1t482KxtvQ6D1TVzQ7F3xoz6Hw2phWDlCJg7VcEix6XFIdlOlcNbgODKp86tCHVGrzEcVnBk99sq5XGd1DNFANeggJYjfBWjAbJSchyEuVlENawP0DWoZWKuP4IPtvZbmnRL0472K3bBQIH5SpPxtXy5NJjoWceA7FeT7IwpivQdqLaeZE0QfiMW1KozkdUtRsGH6ryobMpDbfLt0Z2FAXbR3QQwuIizgZFQ4Gh4lY5pt9RMTieqBIkdXI979BGU2yYtJSanOMsDLWydCQfolxJWbbIdbEggZLBKbFmKXoRMcUyM8NlGnWyuERUtbAs4ZRPHdIWtlbJRtQwoddmlZhI3I8A9K8SyflGzcVjCqGkZnaZrxHNxIcMaeGQdXXxGHFi6AeYBAlo4Q9HZIjJjtOhl4VLmVvcphmMESM8ltxHQQUHjJhYyf5Ndc0i8mHOTNS7yx5hNrJCyJ1ZFj4QeIom7wczw98Bn6SxxoqPtnXp4FyiEb2MCyj2AHaB8FejdIRhqQVfRSDHNRTYERTDGHFHDFGSDFGRTYFGHDFGDFSFGWERTFSDFGDFADWERTFGHDFGHFGH45} 	\end{align} Using Theorem~\ref{P01} and Lemma~\ref{L04}, we obtain \begin{align} (\ad \TT)^{k'} (\partial_s) = \sum_{l=1}^{3} b_{k',s}^l    \partial_l \comma k'\in {\mathbb N} \commaone s\in \{1,2,3\}   , \label{SDHNRTYERTDGHFHDFGSDFGRTYFGHDFGDFSFGWERTFSDFGDFADWERTFGHDFGHFGH17} \end{align} where  \begin{align} \max_{\vert \beta \vert = k} \vert \partial^\beta b_{k',s}^l \vert \lesssim (k+k')! \bar{K}_1^k \bar{K}_2^{k'} \comma k \in \mathbb{N}_0 \comma k'\in \mathbb{N} \comma s,l \in \{1,2,3\} , \label{SDHNRTYERTDGHFHDFGSDFGRTYFGHDFGDFSFGWERTFSDFGDFADWERTFGHDFGHFGH159} \end{align} for some constants $\bar{K}_1, \bar{K}_2 \geq 1$. From the expression \eqref{SDHNRTYERTDGHFHDFGSDFGRTYFGHDFGDFSFGWERTFSDFGDFADWERTFGHDFGHFGH17} and the Leibniz rule, we arrive at \begin{align} \begin{split} \partial_x^j [\TT^k, \dive] (\epsilon \partial_t)^i v  & = \partial_x^j \sum_{s=1}^3 \sum_{k'=1}^k \sum_{l=1}^3 \binom{k}{k'} b_{k',s}^l \partial_l \TT^{k-k'} (\epsilon \partial_t)^i v_s \\ & = \sum_{s=1}^3 \sum_{k'=1}^k \sum_{l=1}^3 \sum_{j'=0}^j \binom{k}{k'} \binom{j}{j'} \partial_x^{j'} b_{k',s}^l \partial_x^{j-j'} \partial_l \TT^{k-k'} (\epsilon \partial_t)^i v_s . \end{split}    \llabel{8rEtz0mq54IZtbSlXdBmEvCuvAf5bYxZ3LEsJYEX8eNmotV2IHlhJE70cs45KVwJR1riFMPEsP3srHa8pqwVNAHusohYIrkNwekfRbDVLm2axu6caKkTXrgBgnQhUA1z8X6MtqvksUfAFVLgTmqPntrgIggjfJfMGfCuByBS7njWfYRNhpHsjFCzM4f6cRDgjPZkbSUHQBnzQwEnS9CxSfn00xmAfwlTv4HIZIZAyXIs4hPOPjQ3v93iTL0JtNJ8baBBWcY18vifUiGKvSQ4gEkZ10yS5lXCwI4oX2gPBisFp7TjKupgVn5oi4uxKt2QP4kbrChS5ZnuWXWep0mOjW1r2IaXvHle8ksF2XQ529gTLs3uvAOf64HOVIqrbLoG5I2n0XskvcKYFIV8yP9tfMEVPR7F0ipDaqwgQxro5EtIWr3tEaSs5CjzfRRALgvmyMhIztVKjStP744RC0TTPQpn8gLVtzpLzEQe2Rck9WuM7XHGA7O7KGwfmZHLhJRNUDEQeBrqfKIt0Y4RW49GKEHYptgLH4F8rZfYCvcf1pOyjk8iTES0ujRvFpipcwIvLDgikPuqqk9REdH9YjRUMkr9byFJKLBex0SgDJ2gBIeCX2CUZyyRtGNY3eGOaDp3mwQyV1AjtGLgSC1dDpQCBcocMSM4jqbSWbvx6aSnuMtD05qpwNDlW0tZ1cbjzwU5bUdCGAghCw0nICDFKHRkphbtA6nYld6c5TSkDq3Qxo2jhDxQbmb8nPq3zNZQFJJyuVm1C6rzRDCB1meQy4TtYr5jQVWoOfbrSDHNRTYERTDGHFHDFGSDFGRTYFGHDFGDFSFGWERTFSDFGDFADWERTFGHDFGHFGH46} \end{align} Using \eqref{SDHNRTYERTDGHFHDFGSDFGRTYFGHDFGDFSFGWERTFSDFGDFADWERTFGHDFGHFGH159}, we obtain \begin{align} \begin{split} \Vert \partial_x^j [\TT^k, \dive] (\epsilon \partial_t)^i v \Vert_{L^2} & \lesssim \sum_{s=1}^3 \sum_{k'=1}^k \sum_{j'=0}^j \sum_{l=1}^3 \binom{k}{k'} \binom{j}{j'} \Vert \partial_x^{j'} b_{k',s}^l \partial_x^{j-j'+1}  T^{k-k'} (\epsilon \partial_t)^i v_s  \Vert_{L^2} \\ & \lesssim \sum_{k'=1}^k \sum_{j'=0}^j \binom{k}{k'} \binom{j}{j'} (j' + k')! \bar{K}_1^{j'} \bar{K}_2^{k'} \Vert \partial_x^{j-j'+1} T^{k-k'} (\epsilon \partial_t)^i v  \Vert_{L^2}, \end{split}    \llabel{YQ6qakZepHb2b5w4KN3mEHtQKAXsIycbakyID9O8YCmRlEW7fGISs6xazbM6PSBN2Bjtb65zz2NuYo4kUlpIqJVBC4DzuZZN6Zkz0oommnswebstFmlxkKEQEL6bsoYzxx08IQ5Ma7InfdXLQ9jeHSTmigttk4vP7778Hp1o67atRbfcrS2CWzwQ9j0Rjr0VL9vlvkkk6J9bM1XgiYlay8ZEq39Z53jRnXh5mKPPa5tFw7E0nE7CuFIoVlFxguxB1hqlHeOLdb7RKfl0SKJiYekpvRSYnNFf7UVOWBvwpN9mtgGwh2NJCY53IdJXPpYAZ1B1AgSxn61oQVtg7W7QcPC42ecSA5jG4K5H1tQs6TNphOKTBIdGkFSGmV0kzAxavQzjeXGbiSjg3kYZ5LxzF3JNHknrmy4smJ70whEtBeXkSTWEujcAuS0NkHloa7wYgMa5j8g4gi7WZ77Ds5MZZMtN5iJEaCfHJ0sD6zVuX06BP99Fga9GgYMv6YFVOBERy3Xw2SBYZDxixxWHrrlxjKA3fokPh9Y758fGXEhgbBw82C4JCStUeozJfIuGjPpwp7UxCE5ahG5EGJF3nRLM8CQc00TcmXISIyZNJWKMIzkF5u1nvD8GWYqBt2lNxdvzbXj00EEpUTcw3zvyfab6yQoRjHWRFJzPBuZ61G8w0SAbzpNLIVjWHkWfjylXj6VZvjsTwO3UzBosQ7erXyGsdvcKrYzZGQeAM1u1TNkybHcU71KmpyahtwKEj7Ou0A7epb7v4FdqSAD7c02cGvsiW444pFeh8OdjSDHNRTYERTDGHFHDFGSDFGRTYFGHDFGDFSFGWERTFSDFGDFADWERTFGHDFGHFGH47} \end{align} and  \eqref{SDHNRTYERTDGHFHDFGSDFGRTYFGHDFGDFSFGWERTFSDFGDFADWERTFGHDFGHFGH18} follows by setting $\bar{K} \geq \max(\bar{K}_1, \bar{K}_2)$. \par The proofs of \eqref{SDHNRTYERTDGHFHDFGSDFGRTYFGHDFGDFSFGWERTFSDFGDFADWERTFGHDFGHFGH19}--\eqref{SDHNRTYERTDGHFHDFGSDFGRTYFGHDFGDFSFGWERTFSDFGDFADWERTFGHDFGHFGH22} are analogous. \end{proof} \par The next lemma provides an analytic estimate for the unit normal vector to the boundary. \begin{Lemma} \label{L32} \rm There exists a constant $\tilde{\eta}>0$ such that  \begin{align} \sum_{j=0}^\infty \sum_{k=0}^\infty  \frac{\tilde{\eta}^{j+k}}{(j+k-3)!}  \Vert \partial_{x}^j \TT^k \nu \Vert_{L^\infty (\Omega_{\delta_0})} \les 1 , \label{SDHNRTYERTDGHFHDFGSDFGRTYFGHDFGDFSFGWERTFSDFGDFADWERTFGHDFGHFGH187} \end{align} where $\nu$ is the unit normal vector to the boundary satisfying \eqref{SDHNRTYERTDGHFHDFGSDFGRTYFGHDFGDFSFGWERTFSDFGDFADWERTFGHDFGHFGH87}. \end{Lemma} \par \begin{proof}[Proof of Lemma~\ref{L32}] We claim that there exists a constant $\bar{\eta}>0$, such that for all $k\in \mathbb{N}_0$, we have \begin{align} 	\sum_{j=0}^\infty 	\frac{\bar{\eta}^{j+k} }{(j+k-3)!}  	\Vert \partial_{x}^j \TT^k \nu \Vert_{L^\infty (\Omega_{\delta_0})} 	\leq 	C_0^{k+1}, 	\label{SDHNRTYERTDGHFHDFGSDFGRTYFGHDFGDFSFGWERTFSDFGDFADWERTFGHDFGHFGH500} \end{align} for some constant $C_0\geq1$. In \eqref{SDHNRTYERTDGHFHDFGSDFGRTYFGHDFGDFSFGWERTFSDFGDFADWERTFGHDFGHFGH500}, we then choose $\tilde{\eta} =  \bar{\eta}/2C_0$ to get \begin{align} \sum_{j=0}^\infty \sum_{k=0}^\infty	 \frac{\tilde{\eta}^{j+k}}{(j+k-3)!} \Vert \partial_{x}^j \TT^k \nu \Vert_{L^\infty (\Omega_{\delta_0})} \leq \sum_{j=0}^\infty \sum_{k=0}^\infty	 \frac{\bar{\eta}^{j+k}}{2^k C_0^k (j+k-3)!} \Vert \partial_{x}^j \TT^k \nu \Vert_{L^\infty (\Omega_{\delta_0})} \les 1 ,    \llabel{wM7olsSQoeyZXota8wXr6NSG2sFoGBel3PvMoGgamq3YkaatLidTQ84LYKFfAF15vlZaeTTxvru2xlM2gFBbV80UJQvkebsTqFRfmCSVe34YVHOuKokFXYI2MTZj8BZX0EuD1dImocM93NjZPlPHqEll4z66IvF3TOMb7xuVRYjlVEBGePNUgLqSd4OYNeXudaDQ6BjKUrIpcr5n8QTNztBho3LC3rc30it5CN2TmN88XYeTdqTLPlS97uLMw0NAsMphOuPNisXNIlWfXBGc2hxykg50QTN75t5JNwZR3NH1MnVRZj2PrUYveHPEljGaTIx4sCFzKB0qp3PleK68p85w44l5zZl07brv61KkiAuTSA5dkwYS3F3YF3e1xKEJWoAvVOZVbwNYgF7CKbSi92R0rlWh2akhCoEppr6O2PZJDZN8ZZD4IhHPTMvSDTgOy1lZ0Y86n9aMgkWdeuOZjOi2Fg3ziYaSRCjlzXdQKbcnb5pKTqrJp1P6oGyxc9vZZRZeFr5TsSZzGl7HWuIGM0yReYDw3lMuxgAdF6dpp8ZVRcl7uqH8OBMbzL6dKBflWCWdlVhycV5nEpv2JSkD0ccMpoIR38QpeZj9j0ZoPmqXRTxBs8w9Q5epR3tN5jbvbrbSK7U4W4PJ0ovnB0opRpCYNPso834PwtSRqvir4DRqujaJq32QUTG1Pgbp6nJM2CUnENdJCr3ZGBHEgBtdsTd84gM22gKBN7QnmRtJgKUIGEeKx64yAGKGezeJNmpeQkLR389HH9fXLBcE6T4GjVZLIdLQIiQtSDHNRTYERTDGHFHDFGSDFGRTYFGHDFGDFSFGWERTFSDFGDFADWERTFGHDFGHFGH38} \end{align} obtaining \eqref{SDHNRTYERTDGHFHDFGSDFGRTYFGHDFGDFSFGWERTFSDFGDFADWERTFGHDFGHFGH187}. In the remainder of the proof, we proceed by induction to prove \eqref{SDHNRTYERTDGHFHDFGSDFGRTYFGHDFGDFSFGWERTFSDFGDFADWERTFGHDFGHFGH500} for all $k\in \mathbb{N}_0$. \par Firstly, we use \eqref{SDHNRTYERTDGHFHDFGSDFGRTYFGHDFGDFSFGWERTFSDFGDFADWERTFGHDFGHFGH87} to obtain \eqref{SDHNRTYERTDGHFHDFGSDFGRTYFGHDFGDFSFGWERTFSDFGDFADWERTFGHDFGHFGH500} for $k=0$ by taking $\bar{\eta} = \eta$. Now we assume that \eqref{SDHNRTYERTDGHFHDFGSDFGRTYFGHDFGDFSFGWERTFSDFGDFADWERTFGHDFGHFGH500} holds for some $k\in \mathbb{N}_0$ and aim to show that it also holds for $k+1$. Using \eqref{SDHNRTYERTDGHFHDFGSDFGRTYFGHDFGDFSFGWERTFSDFGDFADWERTFGHDFGHFGH10}, the Leibniz rule, and H\"older's inequality, we arrive at \begin{align} \begin{split} & \sum_{j=0}^\infty	 \frac{\bar{\eta}^{j+k+1}}{(j+k-2)!}  \Vert \partial_{x}^j \TT^{k+1} \nu \Vert_{L^\infty (\Omega_{\delta_0})} \\ &\indeq \leq C \sum_{i=1}^3 \sum_{l=1}^3 \sum_{j=0}^\infty	 \sum_{j'=0}^j	 \binom{j}{j'} \frac{\bar{\eta}^{j+k+1}}{(j+k-2)!}  \Vert \partial_{x}^{j'} b_{il} \Vert_{L^\infty (\Omega_{\delta_0})} \Vert \partial_{x}^{j-j'+1} \TT^{k} \nu \Vert_{L^\infty (\Omega_{\delta_0})} \\ &\indeq \leq C \sum_{j=0}^\infty	 \sum_{j'=0}^j	 \bar{\eta}^{j'}  C^{j'}  \frac{j! (j-j'+k-2)!}{(j-j')!(j+k-2)!}  \times \left( \frac{\bar{\eta}^{j-j'+k+1}}{(j-j'+k-2)!} \Vert \partial_{x}^{j-j'+1} \TT^{k} \nu \Vert_{L^\infty (\Omega_{\delta_0})} \right) \\ &\indeq \leq C \sum_{j''=1}^\infty	 \frac{ \bar{\eta}^{j''+k}}{(j''+k-3)!} \Vert \partial_{x}^{j''} \TT^{k} \nu \Vert_{L^\infty (\Omega_{\delta_0})} \times \sum_{j=j''-1}^\infty	 (C_1 \bar{\eta})^{j+1-j''} \frac{j! (j''+k-3)!}{(j''-1)! (j+k-2)!} , \label{SDHNRTYERTDGHFHDFGSDFGRTYFGHDFGDFSFGWERTFSDFGDFADWERTFGHDFGHFGH501} \end{split} \end{align} where $C_1\geq 1$ is a constant. Note that the sum in $j$ is dominated by a constant $C$ uniformly for all $j'', k\in \mathbb{N}_0$, by taking $\bar{\eta} \leq 1/2C_1$. Thus from \eqref{SDHNRTYERTDGHFHDFGSDFGRTYFGHDFGDFSFGWERTFSDFGDFADWERTFGHDFGHFGH501} and the induction hypothesis for $k$, we have \begin{align} \begin{split} & \sum_{j=0}^\infty	 \frac{\bar{\eta}^{j+k+1}}{(j+k-2)!}  \Vert \partial_{x}^j \TT^{k+1} \nu \Vert_{L^\infty (\Omega_{\delta_0})} \leq C C_0^{k+1} ,    \llabel{kBk9G9FzHWIGm91M7SW029tzNUX3HLrOUtvG5QZnDqyM6ESTxfoUVylEQ99nTCSkHA8sfxrONeFp9QLDnhLBPibiujcJc8QzZ2KzDoDHg252clhDcaQ1cnxG9aJljFqmADsfDFA0wDO3CZrQ1a2IGtqKbjciqzRSd0fjSJA1rsie9iqOr5xgVljy6afNuooOyIVlT21vJWfKUdeLbcq1MwF9NR9xQnp6TqgElSk50p43HsdCl7VKkZd12Ijx43vI72QyQvUm77BV23a6Wh6IXdP9n67StlZllbRiDyGNr0g9S4AHAVga0XofkXFZwgGtsW2J492NC7FAd8AVzIE0SwEaNEI8v9ele8EfNYg3uWVH3JMgi7vGf4N0akxmBAIjpx4dXlxQRGJZerTMzBxY9JAtmZCjH9064Q4uzKxgmpCQg8x06NYx02vknEtYX5O2vgP3gcspGswFqhX3apbPWsf1YOzHivDia1eODMILTC2mPojefmEVB9hWwMaTdIGjm9PdpHVWGV4hXkfK5Rtci05ekzj0L8Tme2JPXpDI8EbcqV4FdxvrHIeP8CdORJpTiMVEbAunSGsUMWPts4uBv2QSiXIb7B8zo7bp9voEwNRuXJ4ZxuRZYhc1h339THRXVFw5XVW8gaB39mFSv6MzeznkbLHrtZ73hUuaqLvPhgTlNnVpo1ZggmnRAqM3X31ORYSj8RktS8VGOjrz1iblt3uOuEs8Q3xJ1cA2NKoF8o6U3mW2Hq5y6jposxJgwWZ4Exd79JvlcwauoRDCYZzmpabV09jgumebSDHNRTYERTDGHFHDFGSDFGRTYFGHDFGDFSFGWERTFSDFGDFADWERTFGHDFGHFGH37} \end{split} \end{align} concluding the proof of \eqref{SDHNRTYERTDGHFHDFGSDFGRTYFGHDFGDFSFGWERTFSDFGDFADWERTFGHDFGHFGH500} for $k+1$ by simply taking $C_0\geq C$. \end{proof} \par \startnewsection{Analytic estimate of the entropy}{sec04} The following lemma provides an analytic estimate for the entropy $S$. \par \cole \begin{Lemma} \label{L07} Let $M_0>0$. For any $\kappa$, $\bar{\kappa} \in (0,1]$, there exists $\tau_1 \in (0,1]$ such that if $0<\tau(0) \leq \tau_1$, then \begin{align} \Vert S(t) \Vert_{A(\tau(t))}  \les 1 + tQ(M_{\epsilon, \kappa, \bar{\kappa}}(t)) \comma t\in (0,T_0] , \label{SDHNRTYERTDGHFHDFGSDFGRTYFGHDFGDFSFGWERTFSDFGDFADWERTFGHDFGHFGH227} \end{align} for all $\epsilon \in (0,1]$, provided $K$ in \eqref{SDHNRTYERTDGHFHDFGSDFGRTYFGHDFGDFSFGWERTFSDFGDFADWERTFGHDFGHFGH99} satisfies \begin{align} K \geq Q(M_{\epsilon, \kappa, \bar{\kappa}}(T_0)) ,    \llabel{zcbugpatf9yU9iBEyv3UhS79XdImPNEhN64Rs9iHQ847jXUCAufFmsnUudD4Sg3FMLMWbcBYs4JFyYzlrSfnkxPjOHhsqlbV5eBld5H6AsVtrHgCNYn5aC028FEqoWaKSs9uu8xHrbn1eRIp7sL8JrFQJatogZc54yHZvPxPknqRqGw7hlG6oBkzlEdJSEigf0Q1BoCManS1uLzlQ3HnAuqHGPlcIadFLRkdjaLg0VAPAn7c8DqoV8bRCvOzqk5e0Zh3tzJBWBORSwZs9CgFbGo1EFAK7EesLXYWaOPF4nXFoGQlh3pG7oNtG4mpTMwEqV4pO8fMFjfgktnkwIB8NP60fwfEhjADF3bMqEPV9U0o7fcGqUUL10f65lThLWyoXN4vuSYes96Sc2HbJ0hugJMeB5hVaEdLTXrNo2L78fJmehCMd6LSWqktpMgskNJq6tvZOkgp1GBBqG4mA7tMVp8Fn60ElQGMxjoGWCrvQUYV1KYKLpPzVhhuXVnWaUVqLxeS9efsAi7LmHXCARg4YJnvBe46DUuQYkdjdz5MfPLHoWITMjUYM7Qryu7W8Er0Ogj2fKqXSclGmIgqXTam7J8UHFqzvbVvxNiuj6Ih7lxbJgMQYj5qtgaxbMHwbJT2tlBsib8i7zj6FMTLbwJqHVIiQ3O0LNnLypZCTVUM1bcuVYTejG3bfhcX0BVQl6Dc1xiWVK4S4RW5PyZEVW8AYt9dNVSXaOkkGKiLHhzFYYK1qNGGEEU4FxdjaS2NRREnhHmB8Vy446a3VCeCkwjCMe3DGfMiFopvSDHNRTYERTDGHFHDFGSDFGRTYFGHDFGDFSFGWERTFSDFGDFADWERTFGHDFGHFGH48} \end{align} where $T_0 >0$ is a sufficiently small constant depending on $M_0$. \end{Lemma} \colb \par \begin{proof}[Proof of Lemma~\ref{L07}] Let \begin{equation} \tau_1=\kappa^{4}    . \label{SDHNRTYERTDGHFHDFGSDFGRTYFGHDFGDFSFGWERTFSDFGDFADWERTFGHDFGHFGH25} \end{equation} Now, we fix $(j,k,i) \in \mathbb{N}_0^3$, apply $\partial_x^j \TT^k (\epsilon \partial_t)^i$ to the equation \eqref{SDHNRTYERTDGHFHDFGSDFGRTYFGHDFGDFSFGWERTFSDFGDFADWERTFGHDFGHFGH02}, and take the $L^2$-inner product with $\partial_x^j \TT^k (\epsilon \partial_t)^i S$, obtaining \begin{align} \frac{1}{2} \frac{d}{dt} \Vert \partial_x^j \TT^k (\epsilon \partial_t)^i S \Vert_{L^2}^2  + \bigl\langle  v \cdot \nabla \partial_x^j \TT^k (\epsilon \partial_t)^i S, \partial_x^j \TT^k (\epsilon \partial_t)^i S \bigr\rangle  =  \bigl\langle [v\cdot \nabla , \partial_x^j \TT^k (\epsilon \partial_t)^i] S, \partial_x^j \TT^k (\epsilon \partial_t)^i S \bigr\rangle , \llabel{lzLp5r0zdXrrBDZQv9HQ7XJMJogkJnsDxWzIN7FUfveeL0ljk83TxrJFDTvEXLZYpEq5emBawZ8VAzvvzOvCKmK2QngMMBAWcUH8FjSJthocw4l9qJTVGsq8yRw5zqVSpd9ArUfVDcDl8B1o5iyUR4KNqb84iOkIQGIczg2ncttxdWfLQlNnsg3BBjX2ETiPrpqigMOSw4CgdGPfiG2HNZhLeaQwywsiiAWrDjo4LDbjBZFDrLMuYdt6k6Hn9wp4Vk7tddFrzCKidQPfCRKUedzV8zISvntBqpu3cp5q7J4FgBq59pSMdEonG7PQCzMcWlVR0iNJhWHVugWPYdIMgtXB2ZSaxazHeWp7rfhk4qrAbJFFG0lii9MWIl44js9gNlu46CfP3HvS8vQxYw9cEyGYXi3wi41aIuUeQXEjG3XZIUl8VSPJVgCJ3ZOliZQLORzOFVKqlyz8D4NB6M5TQonmBvikY88TJONaDfE2uzbcvfL67bnJUz8Sd7yx5jWroXdJp0lSymIK8bkKzqljNn4KxluFhYLg0FrO6yRztwFTK7QRN01O21ZcHNKgRM7GZ9nB1Etq8sqlAsfxotsl927c6Y8IY8T4x0DRhoh0718MZJoo1oehVLr8AEaLKhyw6SnDthg2HMt9D1jUF5b4wcjllAvvOShtK806ujYa0TYO4pcVXhkOOJVtHN98Qqq0J1HkNcmLS3MApQ75AlAkdnMyJMqACerDl5yPys44a7cY7sEp6LqmG3V53pBs2uPNUM7pX6sy95vSv7iIS8VGJ08QKhAS3jIDSDHNRTYERTDGHFHDFGSDFGRTYFGHDFGDFSFGWERTFSDFGDFADWERTFGHDFGHFGH51} \end{align} where $\langle \cdot,\cdot\rangle$ denotes the scalar product in $L^2$. Using the Cauchy-Schwarz inequality and the boundary condition \eqref{SDHNRTYERTDGHFHDFGSDFGRTYFGHDFGDFSFGWERTFSDFGDFADWERTFGHDFGHFGH03}, we obtain \begin{align} \frac{d}{dt} \Vert \partial_x^j \TT^k (\epsilon \partial_t)^i S \Vert_{L^2}  \les \Vert \nabla v \Vert_{L_x^\infty} \Vert \partial_x^j \TT^k (\epsilon \partial_t)^i S \Vert_{L^2}  + \Vert [v\cdot \nabla, \partial_x^j \TT^k (\epsilon \partial_t)^i] S \Vert_{L^2} . \label{SDHNRTYERTDGHFHDFGSDFGRTYFGHDFGDFSFGWERTFSDFGDFADWERTFGHDFGHFGH54} \end{align} Using the notation \eqref{SDHNRTYERTDGHFHDFGSDFGRTYFGHDFGDFSFGWERTFSDFGDFADWERTFGHDFGHFGH97} and  \begin{align} \begin{split} \Vert u \Vert_{\tilde A(\tau)}  = \sum_{ j+k+i \geq 4}  \Vert \partial_x^j \TT^k (\epsilon \partial_t)^i u \Vert_{L^2} \frac{\kappa^{(j-1)_+} \bar{\kappa}^{k} (j+k+i-3) \tau(t)^{j+k+i-4}}{(j+k+i-3)!} , \end{split} \llabel{NTJsfbhIiUNfeH9Xf8WeCxmBLgzJTIN5NLhvdBOzPmopxYqM4VhkybtYga3XVTTqLyAHyqYqofKP58n8qR9AYrRRetBFxHGg7pduM8gm1TdplRKIW9gi5ZxEEAHDeAsfP5hbxAxbWCvpWk9caqNibi5A5NY5IlVAS3ahAaB8zzUTuyK55glDL5XO9CpORXwrEV1IJG7wEgpOag9zbJiGeT6HEmcMaQpDfyDxheTNjwfwMx2CipkQeUjRUVhCfNMo5DZ4h2adEjZTkOx946EeUIZv7rFL6dj2dwgRxgbObqJsYmsDqQAssn9g2kCb1MsgKfx0YjK0GlrXO7xI5WmQHozMPfCXTmDk2Tl0oRrnZvAsFr7wYEJHCd1xzCvMmjeR4ctk7cS2fncvfaN6AO2nIh6nkVkN8tT8aJdb708jZZqvL1ZuT5lSWGo08cLJ1q3TmAZF8qhxaoYJC6FWRuXHMx3Dcw8uJ87Q4kXVac6OOPDZ4vRtsP01hKUkdaCLBiPSAtLu9WLoyxMaBvixHyadnqQSJWgSCkF7lHaO2yGRIlK3aFZenCWqO9EyRofYbkidHQh1G2vohcMPoEUzp6f14NioarvW8OUc426ArsSo7HiBUKdVs7cOjaV9KEUtKne4VIPuZc4bPRFB9ABfqclU2ct6PDQudt4VOzMMUNrnzJXpxkE2NB8pfJiM4UNg4Oi1gchfOU62avNrpcc8IJm2WnVXLD672ltZTf8RDwqTvBXEWuH2cJtO1INQUlOmEPvj3OOvQSHxiKc8RvNnJNNCC3KXp3J8w50WSDHNRTYERTDGHFHDFGSDFGRTYFGHDFGDFSFGWERTFSDFGDFADWERTFGHDFGHFGH57} \end{align}    the estimate \eqref{SDHNRTYERTDGHFHDFGSDFGRTYFGHDFGDFSFGWERTFSDFGDFADWERTFGHDFGHFGH54} implies \begin{align} \begin{split} \frac{d}{dt} \Vert S \Vert_{A(\tau)} & = \dot{\tau}(t) \Vert S \Vert_{\tilde A(\tau)}  +  \sum_{(j,k,i) \in \mathbb{N}_0^3} \frac{\kappa^{(j-1)_+}\bar{\kappa}^{k} \tau^{(j+k+i-3)_+}}{(j+k+i-3)!}  \frac{d}{dt} \Vert \partial_x^j \TT^k (\epsilon \partial_t)^i S\Vert_{L^2} \\ & \les \dot{\tau}(t) \Vert S \Vert_{\tilde A(\tau)} + \Vert \nabla v \Vert_{L_x^\infty} \Vert S \Vert_{A(\tau)}  + \sum_{(j,k,i) \in \mathbb{N}_0^3}  \sum_{\substack{(j',k',i') \leq (j,k,i) \\j'+k'+i' \geq 1}} \mathcal{C}_{j,k,i,j',k',i'} \\ &\indeq + \Vert v \Vert_{L^\infty_x}  \sum_{(j,k,i)\in \mathbb{N}_0^3} \frac{\kappa^{(j-1)_+} \bar{\kappa}^{k}\tau^{(j+k+i-3)_+}}{(j+k+i-3)!} \Vert \partial_{x}^j [\TT^k, \nabla] (\epsilon \partial_t)^i S \Vert_{L^2} , \end{split} \label{SDHNRTYERTDGHFHDFGSDFGRTYFGHDFGDFSFGWERTFSDFGDFADWERTFGHDFGHFGH55} \end{align} where  \begin{align} \mathcal{C}_{j,k,i,j',k',i'} = \frac{\kappa^{(j-1)_+} \bar{\kappa}^{k}\tau^{(j+k+i-3)_+}}{(j+k+i-3)!}  \binom{j}{j'} \binom{k}{k'} \binom{i}{i'}  \Vert \partial_x^{j'} \TT^{k'} (\epsilon \partial_t)^{i'} v \cdot \partial_x^{j-j'} \TT^{k-k'}  (\epsilon \partial_t)^{i-i'} \nabla S \Vert_{L^2} . \llabel{sOTXHHhvL5kBpKr5urqvVFv8upqgPRPQbjCxme33uJUFhYHBhYMOd01Jt7ySfVpF0z6nCK8grRahMJ6XHoLGu4v2o9QxONVY88aum7cZHRNXHpG1a8KYXMayTxXIkO5vV5PSkCp8PBoBv9dBmepms7DDUaicXY8Lx8IBjFBtke2yShNGE7a0oEMFyAUUFkRWWheDbHhAM6Uh373LzTTTxxm6ybDBsIIIAPHhi837ra970Fam4O7afXUGrf0vWe52e8EPyBFZ0wxBzptJf8LiZkdTZSSPpSzrbGEpxb4KXLHLg1VPaf7ysvYsFJb8rDpAMKnzqDg7g2HwCruQNDBzZ5SNMayKB6RIePFIHFQawrRHAx38CHhoBGVIRvxSMYf0g8hacibKG3CuSl5jTKl42o6gAOYYHUB2SVO3Rc4whR8pwkrrANA4j7MfcEMal4HwKPTgZaZ9G8sevuwIAhkhR8WgafzJA0FVNmSCwUB0QJDgRjCSVSrsGMbWABxvzOMMycNSOylZzwFiNPcmfQhwZPanLp01EUVHMA2dE0nLuNVxKxco7opbQzRaAlowoVtorqU5eUXtEl1qh1IPCPEuVPxcnTwZvKJTpHZpqxXOFaGsrQNPnuqc9CMD8mJZoMOCy6XHjWAfEqI95ZjgcPdV8maWwkHlM30VwDjXlx1QfgyLFxe5iwJDnJIrUKL6eCVth3gX8RLACCC2q5kMYXo8NsDfAn3OsapZT8U3FdMIEDKqMo0YvbwCGMR66XTyIOtLUuC6cmcOFsvpWTniQmu0PeHEF9ImolIuSDHNRTYERTDGHFHDFGSDFGRTYFGHDFGDFSFGWERTFSDFGDFADWERTFGHDFGHFGH56} \end{align} We split the third term on the far right side of \eqref{SDHNRTYERTDGHFHDFGSDFGRTYFGHDFGDFSFGWERTFSDFGDFADWERTFGHDFGHFGH55} according to the low and high values of $j'+k'+i'$.  We claim that there exists a constant $C>0$ such that \begin{align} & I_1 = \sum_{0\leq j+k+i \leq 4}  \sum_{\substack{(j', k', i') \leq (j,k,i)\\ j'+k'+i' \geq 1}} \mathcal{C}_{j,k,i,j',k',i'} \les 1 , \label{SDHNRTYERTDGHFHDFGSDFGRTYFGHDFGDFSFGWERTFSDFGDFADWERTFGHDFGHFGH104}   \\ & I_2 = \sum_{j+k+i \geq 5}  \sum_{\substack{(j', k', i') \leq (j,k,i)\\ 1 \leq j'+k'+i'  \leq [ (j+k+i)/2 ]}} \mathcal{C}_{j,k,i,j',k',i'}  \les \Vert v \Vert_{A(\tau)} \Vert S \Vert_{\tilde{A}(\tau)}  + \Vert v \Vert_{A(\tau)} , \label{SDHNRTYERTDGHFHDFGSDFGRTYFGHDFGDFSFGWERTFSDFGDFADWERTFGHDFGHFGH102}   \\ & I_3 = \sum_{j+k+i \geq 5}  \sum_{\substack{(j', k', i') \leq (j,k,i)\\  [( j+k+i )/2] + 1\leq  j'+k'+i'  \leq  j+k+i-3}} \mathcal{C}_{j,k,i,j',k',i'}  \les \Vert v \Vert_{A(\tau)} \Vert S \Vert_{\tilde{A}(\tau)}  + \Vert v \Vert_{A(\tau)} , \label{SDHNRTYERTDGHFHDFGSDFGRTYFGHDFGDFSFGWERTFSDFGDFADWERTFGHDFGHFGH107} \\ & I_4 = \sum_{j+k+i \geq 5}  \sum_{\substack{(j', k', i') \leq (j,k,i)\\   j+k+i-2 \leq  j'+k'+i'  }} \mathcal{C}_{j,k,i,j',k',i'}  \les \Vert v \Vert_{A(\tau)} , \label{SDHNRTYERTDGHFHDFGSDFGRTYFGHDFGDFSFGWERTFSDFGDFADWERTFGHDFGHFGH103} \\ & I_5	 = \sum_{(j,k,i)\in \mathbb{N}_0^3} \frac{\kappa^{(j-1)_+} \bar{\kappa}^{k}\tau^{(j+k+i-3)_+}}{(j+k+i-3)!} \Vert \partial_{x}^j [\TT^k, \nabla] (\epsilon \partial_t)^i S \Vert_{L^2} \les 1+\Vert S \Vert_{\tilde{A}(\tau)} . \label{SDHNRTYERTDGHFHDFGSDFGRTYFGHDFGDFSFGWERTFSDFGDFADWERTFGHDFGHFGH505} \end{align} \par The proofs of \eqref{SDHNRTYERTDGHFHDFGSDFGRTYFGHDFGDFSFGWERTFSDFGDFADWERTFGHDFGHFGH104}--\eqref{SDHNRTYERTDGHFHDFGSDFGRTYFGHDFGDFSFGWERTFSDFGDFADWERTFGHDFGHFGH103} are analogous to those in Section 4 in \cite{JKL}.  Here we only outline necessary modifications. \par \emph{Proof of \eqref{SDHNRTYERTDGHFHDFGSDFGRTYFGHDFGDFSFGWERTFSDFGDFADWERTFGHDFGHFGH104}}\/:  We apply H\"older and Sobolev inequalities on the factor $\Vert \partial_x^{j'} \TT^{k'} (\epsilon \partial_t)^{i'} v \cdot \partial_x^{j-j'} \TT^{k-k'}  (\epsilon \partial_t)^{i-i'} \nabla S \Vert_{L^2}$ and then use \eqref{SDHNRTYERTDGHFHDFGSDFGRTYFGHDFGDFSFGWERTFSDFGDFADWERTFGHDFGHFGH19} to estimate $\Vert \partial_{x}^j \TT^k (\epsilon \partial_t)^i \nabla S \Vert_{L^2}$. Then \eqref{SDHNRTYERTDGHFHDFGSDFGRTYFGHDFGDFSFGWERTFSDFGDFADWERTFGHDFGHFGH104} follows by appealing to \eqref{SDHNRTYERTDGHFHDFGSDFGRTYFGHDFGDFSFGWERTFSDFGDFADWERTFGHDFGHFGH24}. \par \emph{Proof of \eqref{SDHNRTYERTDGHFHDFGSDFGRTYFGHDFGDFSFGWERTFSDFGDFADWERTFGHDFGHFGH102}}\/: Using H\"older and Sobolev inequalities, we obtain \begin{align} \begin{split} 	& \mathcal{C}_{j,k,i,j',k',i'} \mathbbm{1}_{\{j+k+i\geq 5\}} \mathbbm{1}_{\{ 1 \leq j'+k'+i'\leq [(j+k+i)/2] \}} \\ & \les \left(\frac{\kappa^{j'+1} \bar{\kappa}^{k'} \tau^{(j'+k'+i'-1)_+}}{(j'+k'+i'-1)!} \Vert \partial_x^{j'+2} \TT^{k'} (\epsilon \partial_t)^{i'} v \Vert_{L^2} \right)^{3/4} \\&\indeqtimes \left(\frac{\kappa^{(j'-1)_+} \bar{\kappa}^{k'} \tau^{(j'+k'+i'-3)_+}}{(j'+k'+i'-3)!}\Vert \partial_x^{j'} \TT^{k'} (\epsilon \partial_t)^{i'} v \Vert_{L^2} \right)^{1/4}  \\&\indeqtimes \left(\frac{\kappa^{j-j'} \bar{\kappa}^{k-k'} \tau^{(j+k+i-j'-k'-i'-3)_+}}{(j+k+i-j'-k'-i'-3)!} \Vert \partial_x^{j-j'} \TT^{k-k'}  (\epsilon \partial_t)^{i-i'} \nabla S \Vert_{L^2} \right) . \label{SDHNRTYERTDGHFHDFGSDFGRTYFGHDFGDFSFGWERTFSDFGDFADWERTFGHDFGHFGH63} \end{split} \end{align} In \eqref{SDHNRTYERTDGHFHDFGSDFGRTYFGHDFGDFSFGWERTFSDFGDFADWERTFGHDFGHFGH63}, we bounded the rest of the powers  of $\kappa$, $\bar{\kappa}$, and $\tau$  by $C$, which is possible by \eqref{SDHNRTYERTDGHFHDFGSDFGRTYFGHDFGDFSFGWERTFSDFGDFADWERTFGHDFGHFGH25}, and bounded the rest of the factors involving combinatorial symbols by $C$ since $j'+k'+i' \leq [(j+k+i)/2]$. From \eqref{SDHNRTYERTDGHFHDFGSDFGRTYFGHDFGDFSFGWERTFSDFGDFADWERTFGHDFGHFGH63}, we use the discrete H\"older and Young inequalities to get \begin{align} \begin{split} I_2 & \les \Vert v \Vert_{A(\tau)} \Vert S \Vert_{\tilde{A}(\tau)}  + \Vert v \Vert_{A(\tau)} \sum_{j+k+i\geq 3} \frac{\kappa^{j} \bar{\kappa}^{k} \tau^{(j+k+i-3)_+}}{(j+k+i-3)!} \Vert \partial_x^j [\TT^k, \nabla] (\epsilon \partial_t)^i S \Vert_{L^2}  . \label{SDHNRTYERTDGHFHDFGSDFGRTYFGHDFGDFSFGWERTFSDFGDFADWERTFGHDFGHFGH101} \end{split} \end{align} Appealing to \eqref{SDHNRTYERTDGHFHDFGSDFGRTYFGHDFGDFSFGWERTFSDFGDFADWERTFGHDFGHFGH19}, the sum in the second term on the right side of above may be estimated by \begin{align} \begin{split} & \sum_{j+k+i\geq 3} \sum_{k'=1}^k \sum_{j'=0}^j  \binom{k}{k'} \binom{j}{j'} \frac{\kappa^{j} \bar{\kappa}^{k} \tau^{(j+k+i-3)_+}}{(j+k+i-3)!} (j'+k')! \bar{K}^{j'+k'} \Vert \partial_x^{j-j'+1} \TT^{k-k'} (\epsilon \partial_t)^i S \Vert_{L^2} \\ &\indeq \les \sum_{j+k+i\geq 3} \sum_{k'=1}^k \sum_{j'=0}^j  (\kappa \bar{K})^{j'} (\bar{\kappa} \bar{K})^{k'} \mathbbm{1}_{\{i \geq 3 \}} \\ &\indeqtimes \left(\frac{\kappa^{j-j'} \bar{\kappa}^{k-k'} \tau^{(j+k+i-j'-k'-3)_+}}{(j+k+i-j'-k'-3)!} \Vert \partial_x^{j-j'+1} \TT^{k-k'} (\epsilon \partial_t)^i S \Vert_{L^2}  \right) \mathcal{A}_{j,k,i,j',k'} \\ &\indeq\indeq + \sum_{j+k+i\geq 3} \sum_{k'=1}^k \sum_{j'=0}^j  (\kappa \bar{K})^{j'} (\bar{\kappa} \bar{K})^{k'} \mathbbm{1}_{\{0\leq i \leq 2 \}} \\ &\indeqtimes \left(\frac{\kappa^{j-j'} \bar{\kappa}^{k-k'} \tau^{(j+k+i-j'-k'-3)_+}}{(j+k+i-j'-k'-3)!} \Vert \partial_x^{j-j'+1} \TT^{k-k'} (\epsilon \partial_t)^i S \Vert_{L^2}  \right) \mathcal{B}_{j,k,i,j',k'} . \label{SDHNRTYERTDGHFHDFGSDFGRTYFGHDFGDFSFGWERTFSDFGDFADWERTFGHDFGHFGH106} \end{split} \end{align} In \eqref{SDHNRTYERTDGHFHDFGSDFGRTYFGHDFGDFSFGWERTFSDFGDFADWERTFGHDFGHFGH106}, we denote \begin{align} \mathcal{A}_{j,k,i,j',k'} = \binom{k}{k'} \binom{j}{j'} \frac{(j+k+i-j'-k'-3)!(j'+k')!}{(j+k+i-3)!} \mathbbm{1}_{\{i \geq 3 \}} \label{SDHNRTYERTDGHFHDFGSDFGRTYFGHDFGDFSFGWERTFSDFGDFADWERTFGHDFGHFGH213} \end{align} and \begin{align} \mathcal{B}_{j,k,i,j',k'} = \binom{k}{k'} \binom{j}{j'} \frac{(j+k+i-j'-k'-3)!(j'+k')!}{(j+k+i-3)!} \mathbbm{1}_{\{0 \leq i \leq 2 \}} . \llabel{ThHWHwh8Jz4hC0rK2GdNzLXiEY7VuQfRbXpiQnPps9gMA8mWkyXsYFLoiRtl2Kl2pI9bSnyi07mUZqhEsBOCgI4F5AFFdjX3wf0Wu2Yqddp2ZUkjeFMAxnDlsut9qzbyRgDWrHldNZewzEK1cSwWJZywloSof6zVDAB6er0o2HZY1trZhBuL5zYzrAUdMKXVKGWKIHOqqx1zj8tlpxuUD83eLUerjxfHNMZlaqZVhT6Jku15FdLvdeo087AsGC8WdoMnf4dToHw47hglTqjKtAlwR9ufOhLKTDgWZhxHFX5gU5uN2S6esPlxKpXzBmgyWUy5D01WD88a4YmWRfdmev1dBvHOmhTBqurAgTC6yrrRBPn9QfZ9T4mwIh3xjAtkiMlAlTd6fSQ5iQBBY6OErT3gf0DKeECnjgcTXAL8grKBpfcJvq4fpIhWGFSdh6LOqg0ao9AjakqEZKgv95BAqvCSJJgo1Lzsv5yhPQkMpPWnsXvHNZAUQt1DpO47V7AR6JCTR9fnH6QVYwjdZTR3TZSCdOcidDYXYiztVEgNK6hZWtLooE11MiqywCo9kUjdnt1CCc80eJHxMWn7GlDI9yCp9xs9zknCbFsjlKydvYatKpvKJpySVeTP1zRB9N91vA0XXSacVlNzZX3jRgbDjcscBBveadZerkNNT2ni9PDo6HyDNuAvUoYNaIuQ6oUCEi5k5kBw1fwktQDSG4Ky7U2nXSKlOez0PJ0v1SNnMkUdmxxN5t7vhLQxWUulVutc6EMFPa0mIkRDVwluKitmTncyT58CDjRPSDHNRTYERTDGHFHDFGSDFGRTYFGHDFGDFSFGWERTFSDFGDFADWERTFGHDFGHFGH214} \end{align} Applying the combinatorial inequality \begin{align} \binom{k}{k'} \binom{j}{j'}  \leq \binom{k+j}{k'+j'}    \llabel{unXB74DYSJKWEPYi0YxlyAd7HGWyyCeJzd8ht2NnHGfOmsDLbWqhYk2v3J7j2nbB4taYDMN0OkJTdkNPO7JvkTRFYwud2MZ91SZPVQcLlvrOcIN92COu4QpaM7ShSsg1qs8uijWXMMnX9760otPu5BJwtxkMVH4wuj37tRdB7Za2FeTvXLlkC0kZ1ZRCZvbcVw9SRuUimZYbIyqO0qKkkirgpvLzBS44Rwj1NZRJHOafvDTKV3ePSJJ0wuXjKzgeXa11GuCRiRVPRSUNxSqionXMk3f8cKO4inK7IfRUJ0MW6ZdcMT3LaSjZ4IqtQIFDukYcYz700unb8AtSmg2kMKAAL7DB4cgDSBFHX5GPcHAqCytP3oJRryBaylYND2HGJZrTgUURyUwzCli2F9vvyflNg1nR5Y1nxJC235JxySQAmanUoZh1VvDoMy3RL9pUIemC79uRdoMV0hz3sKl3uSB9WEO7EFbVYWJMeDYZe6UuJQ2rbloFIc6ca1hFEdw2d4sSTrHnhAIlQ1o9RphHtZ9C4DInJMbOmYJatZfwEAKGDpQb2Mx5o8ndJnyvrUaP1lkNOGdleO90C3QpE1fEXgQ5Y2APzGVPRJn6rPVXg9U8d9updt0YZpJ4i1h1WTmixB5H0ndNf7UYbKCXm1RHvRl6IgFtPgkhxdX95jIOZ0qtx9xEmzdF1L5sNPY0CKoPBSK6S7faLCutDrBVCtB3NykryNUqACBCto7OVcAjjsNMjfpDA9w2r6CzmQhf2xmwR0HvIhiIjhHoZ4obATsWkJgUBX5JbTtFSDHNRTYERTDGHFHDFGSDFGRTYFGHDFGDFSFGWERTFSDFGDFADWERTFGHDFGHFGH49} \end{align} to \eqref{SDHNRTYERTDGHFHDFGSDFGRTYFGHDFGDFSFGWERTFSDFGDFADWERTFGHDFGHFGH213}, we obtain \begin{align} \mathcal{A}_{j,k,i,j',k'} \les \frac{(j+k+i-j'-k'-3)!(j+k)!}{(j+k+i-3)!(j+k-j'-k')!} \les 1 ,   \label{SDHNRTYERTDGHFHDFGSDFGRTYFGHDFGDFSFGWERTFSDFGDFADWERTFGHDFGHFGH52} \end{align} since $i\geq 3$. Similarly,  \begin{align} \mathcal{B}_{j,k,i,j',k'} \les \frac{(j+k+i-j'-k'-3)!(j+k)!}{(j+k+i-3)!(j+k-j'-k')!} \mathbbm{1}_{\{0 \leq i \leq 2 \}} . \llabel{r3fDq6DjRCxJo6kZKTUfJNw8uCuyAAxqjBCsvLawMMS37OhxIQcaW0SnnMPdZQLjtIIpnIneWsWiUQpO69palKXNRlv9bYG06ptJE7WLpWiXneaUhEszoGmQquLLn5bXdyFia5iLrg7PABBN4v8QapYAv8hGhMd7E10KVuQl78KRe38xdUEFe15K72PTLwYDkutECGADCEBMbcFvLgnnrbdwldq6CS8wVyBzPGZaC3fGkrvmhMj9uvRdnSYXXg12IkXVNSN3601pWdykiU6kaUwDUZ2G8r5XiZXMQ7AGrplYAPlwn11dm00Jodc1h7zFn5rLbVOHMDh0QggiSOKll3vzZ0A6hDO56OyuNBgfzTkNNyR28PsJUDsofPaXgqBUKo9tXhTwgIFAxgS43mPTRh5QLfBBrLyi8we8dXmJnkR7DcCII4Df1yrovtwK6Zq8Fay5DrbFolZgiNNUKQko199y43VW46dUhwt4dEWECfqszQjuZcFNqRA8EBKTjxXvfBSuPr9IelzgCVxKdowNHtbT6jKdPL1MYPfjfmZiLHDZUfPBhjt64XadgQnI2ylWjZqHTp3H9LIGIdsX0liHs3aE1qHcNYwAfL2aJAMnqOiSdxF6cG4bEx1aZjpeJsDcYCD98Z7vMYv77Vfx3eVQakI2aPi3UhSEKfsDNOley7xrpP4S9FoFg8deOZeMJ5PQWSMloZjHqNXrtCh7pqHNFuq1MF5tPbBVgGwFtRYhAi2q52Rw4dFk76zGcROdFIXft1LK3fxk0xMVmqt2h7rqf9OlF4gjjR2BSDHNRTYERTDGHFHDFGSDFGRTYFGHDFGDFSFGWERTFSDFGDFADWERTFGHDFGHFGH58} \end{align} If $j'+k' \geq [(j+k)/2]+1$, then there exists a constant $C\geq 1$ such that \begin{align} \mathcal{B}_{j,k,i,j',k'} \les C^{j'+k'} , \llabel{8FxKi9pwg5yRY8XMXIWISOjbcsRKniRqLwJzkAU7oq6tBKpEv4ESlNOy9u1tciXJCAYUOps6OAh6W4ZxV8lo2dueBiLmZrIw6dFUG4w8PcdiJSFmSwE9KhkgSZ75dycS17zskwXxP9UVlmzJwi5E0CZMvlQLeFZblh53SABDvgFzmlZH5lJ42UOcRojWmp7FtO4atvrDjVQbvSNkhDbNl8R21vnhX3ILYS65pr9OAdnA3j6KEL5tNaNVmsKvlGmy2kMyVvsb73RcLUNGzi6wlc1uFWkYrpXEmfkPv4tQozW2CHUXS2kCG3CfbZ7cLMohJXIukasc5F0vF8GoNZkeI6DEv2OlWLJpuuOiSXadleN1cgyW4u0bp8TDtYFQI8kp29nLjdvVTPfrKxCVr7pdxnVwdHHQubf5O4ixYrddMbrhd60rN8GL1TGfyeCQmNaJN3fJgn0we17GkQAg1WAj6l87vzmOzdKQ1HZ8qATPMo1KAULCHIjKsRXwKT5xYB9iwCmzcM5nlfbkMBhe7ONI5U2NrE5WXsFl5mkw3TzbJicLnBrAjcGa9wJZGTj7YmbbEr0cgxOxs75shL8m7RUxHRQBPihQ1Zg0p7UIyPsopiOaehpYI4z5HbVQxsth4RUeV3BsgK0xxzeyEFCVmbOLA52noFxBk8r1BQcFeK5PIKE9uvUNX3B77LAvk5Gdu4dVuPux0h7zEtOxlvInnvDyGe1qJiATORZ29Nb6q8RI4V3Dv4fBsJDvp6ago5FalZhU1Yx2evycybq7Jw4eJ9oewgCma6lFCjsOSDHNRTYERTDGHFHDFGSDFGRTYFGHDFGDFSFGWERTFSDFGDFADWERTFGHDFGHFGH216} \end{align} uniformly for all $i \in \{0,1,2\}$; if $1\leq j'+k' \leq [(j+k)/2]$, then we have \begin{align} \mathcal{B}_{j,k,i,j',k'} \les 1 , \label{SDHNRTYERTDGHFHDFGSDFGRTYFGHDFGDFSFGWERTFSDFGDFADWERTFGHDFGHFGH217} \end{align} uniformly for all $i \in \{0,1,2\}$. Combining \eqref{SDHNRTYERTDGHFHDFGSDFGRTYFGHDFGDFSFGWERTFSDFGDFADWERTFGHDFGHFGH106}, \eqref{SDHNRTYERTDGHFHDFGSDFGRTYFGHDFGDFSFGWERTFSDFGDFADWERTFGHDFGHFGH52}--\eqref{SDHNRTYERTDGHFHDFGSDFGRTYFGHDFGDFSFGWERTFSDFGDFADWERTFGHDFGHFGH217}, and switching indices from $j-j'$ to $j''$ and $k-k'$ to $k''$, we obtain \begin{align} \begin{split} & \sum_{j+k+i\geq 3}  \frac{\kappa^{j} \bar{\kappa}^{k} \tau^{j+k+i-3}}{(j+k+i-3)!} \Vert \partial_x^j [\TT^k, \nabla] (\epsilon \partial_t)^i S \Vert_{L^2} \\ &\indeq \les \sum_{j''=0}^\infty \sum_{k''=0}^\infty \sum_{i=3}^\infty \left(\frac{\kappa^{j''} \bar{\kappa}^{k''} \tau^{(j''+k''+i-3)_+}}{(j''+k''+i-3)!} \Vert \partial_x^{j''+1} \TT^{k''} (\epsilon \partial_t)^i S \Vert_{L^2}  \right) \\ & \indeqtimes \sum_{j=j''}^\infty \sum_{k=k''+1}^\infty (\kappa \bar{K})^{j-j''} (\bar{\kappa} \bar{K})^{k-k''}  \\ &\indeq\indeq + \sum_{j''=0}^\infty \sum_{k''=0}^\infty \sum_{i=0}^2 \left(\frac{\kappa^{j''} \bar{\kappa}^{k''} \tau^{(j''+k''+i-3)_+}}{(j''+k''+i-3)!} \Vert \partial_x^{j''+1} \TT^{k''} (\epsilon \partial_t)^i S \Vert_{L^2}  \right) \\ & \indeqtimes \sum_{j=j''}^\infty \sum_{k=k''+1}^\infty  (C\kappa \bar{K} )^{j-j''} (C\bar{\kappa} \bar{K})^{k-k''}  . \label{SDHNRTYERTDGHFHDFGSDFGRTYFGHDFGDFSFGWERTFSDFGDFADWERTFGHDFGHFGH100} \end{split} \end{align} Note that the sums in $j$ and $k$ in \eqref{SDHNRTYERTDGHFHDFGSDFGRTYFGHDFGDFSFGWERTFSDFGDFADWERTFGHDFGHFGH100}   are bounded by $C$ for $\kappa \leq 1/C\bar{K}$ and $\bar{\kappa} \leq 1/C\bar{K}$, uniformly for all $j', k', i \in \mathbb{N}_0$. Therefore, by \eqref{SDHNRTYERTDGHFHDFGSDFGRTYFGHDFGDFSFGWERTFSDFGDFADWERTFGHDFGHFGH101} and \eqref{SDHNRTYERTDGHFHDFGSDFGRTYFGHDFGDFSFGWERTFSDFGDFADWERTFGHDFGHFGH100} we conclude \begin{align} \begin{split} I_2 & \les \Vert v \Vert_{A(\tau)} \Vert S \Vert_{\tilde{A}(\tau)}  + \Vert v \Vert_{A(\tau)} \sum_{1 \leq j+k+i \leq 3} \Vert \partial_x^j \TT^k (\epsilon \partial_t)^i S \Vert_{L^2}  \les \Vert v \Vert_{A(\tau)} \Vert S \Vert_{\tilde{A}(\tau)}  + \Vert v \Vert_{A(\tau)} , \end{split}    \llabel{yzeoXOyIagDo4rJPDvdVdAIPfvaxOIsle7l60zfITnPR5IE34RJ1dqTXj0SVuTpTrmkFSn2gIWtUvMdtZIWIZTo2aJpeFRGmLIa5YG6yn0LboerwMPeyvJZHroCE38u158qCnrhQWMU6v8vncMbFGS2vM3vW5qwbO6UKlUBS9Y2oL2juyOlk3XJ7mfyQMGcRyAVSc0Yk8bizEBIGDYdOGoW6emVaisrFm5Ehk6nwh98Pn1pr7Od6qGjlJObLuDe0UqZTHY6w0iZ67Kfw5cPZ1ZFpvG095DYKQTHJ7t5HPbc8WevwqlBtMEFsBgDSaccQONR6sBXLp8nu9ygZ47NBUCknTcXgssyAgKe8nTCWZjy1Ygl83LTLRtp6iSFIDk1BtU6O9pxcNWt3FHEPLmD7TXtkNZr0rYhB8frpuqccVbXrPCHjJbJoqMK7O6CvwuATercnN2SpjTAaMyFNpa9GeKl351nDs3KVr6WMCTYS0zsABO0SwHhCqG0qk62kpIM5YjYCeM76VcZ0cFJZTEHZy5LjzlsDRtfvNE1e4QcGKXy7y7Hp2oXfzX6DZbs8nmLkLREmO32NF7SezN3kn7UZI29gSEw2zk5YBa0aVY6umS1ABaKt5Ahkx6Qz8B2SpqdGFQV8bqASLCGwcdkXBaIL44gxbitC1txjgfQ6AKORWGJKF2ym4xY9EDRne7ODExbjdiFdVElumlJwtcNTiKoLEffm6QI1hmfEXqg4Sd3b8GBBQG7Fsyq6g8ISR17dioqXqfE2jRYpBJkWP2HfthwNrhDsCJ6Xj2O6SDHNRTYERTDGHFHDFGSDFGRTYFGHDFGDFSFGWERTFSDFGDFADWERTFGHDFGHFGH59} \end{align} where the last inequality follows from \eqref{SDHNRTYERTDGHFHDFGSDFGRTYFGHDFGDFSFGWERTFSDFGDFADWERTFGHDFGHFGH24}. \par \emph{Proof of \eqref{SDHNRTYERTDGHFHDFGSDFGRTYFGHDFGDFSFGWERTFSDFGDFADWERTFGHDFGHFGH107}}\/: Using H\"older and Sobolev inequalities, we obtain \begin{align} \begin{split} & \mathcal{C}_{j,k,i,j',k',i'} \mathbbm{1}_{\{j+k+i\geq 5\}} \mathbbm{1}_{\{[(j+k+i)/2] + 1 \leq j'+k'+i'\leq j+k+i-3\}} \\ & \les \left(\frac{\kappa^{(j'-1)_+} \bar{\kappa}^{k'} \tau^{(j'+k'+i'-3)_+}}{(j'+k'+i'-3)!} \Vert \partial_x^{j'} \TT^{k'} (\epsilon \partial_t)^{i'} v \Vert_{L^2} \right) \\ &\indeqtimes \left( \frac{\kappa^{j-j'+2} \bar{\kappa}^{k-k'} \tau^{(j+k+i-j'-k'-i'-1)_+}}{(j+k+i-j'-k'-i')!}  \Vert \partial_x^{j-j'+2} \TT^{k-k'}  (\epsilon \partial_t)^{i-i'} \nabla S \Vert_{L^2} \right)^{3/4}  \\ & \indeqtimes \left(\frac{\kappa^{j-j'} \bar{\kappa}^{k-k'} \tau^{(j+k+i-j'-k'-i'-3)_+}}{(j+k+i-j'-k'-i'-3)!} \Vert \partial_x^{j-j'} \TT^{k-k'}  (\epsilon \partial_t)^{i-i'} \nabla S \Vert_{L^2} \right)^{1/4} . \label{SDHNRTYERTDGHFHDFGSDFGRTYFGHDFGDFSFGWERTFSDFGDFADWERTFGHDFGHFGH105} \end{split} \end{align} Similarly to \eqref{SDHNRTYERTDGHFHDFGSDFGRTYFGHDFGDFSFGWERTFSDFGDFADWERTFGHDFGHFGH63}, by \eqref{SDHNRTYERTDGHFHDFGSDFGRTYFGHDFGDFSFGWERTFSDFGDFADWERTFGHDFGHFGH25}, the remaining factors of $\kappa$, $\bar{\kappa}$, and $\tau$ are bounded by $C$; the product of factors involving combinatorial symbols may be bounded by $C$ since $j'+k'+i' \geq [(j+k+i)/2] + 1$.  By \eqref{SDHNRTYERTDGHFHDFGSDFGRTYFGHDFGDFSFGWERTFSDFGDFADWERTFGHDFGHFGH105}, using the discrete H\"older and the discrete Young inequalities we arrive at \begin{align} \begin{split} I_3 & \les \Vert v \Vert_{A(\tau)} \Vert S \Vert_{\tilde{A}(\tau)}  + \Vert v \Vert_{A(\tau)} \sum_{j+k+i\geq 3} \frac{\kappa^{j} \bar{\kappa}^{k} \tau^{(j+k+i-3)_+}}{(j+k+i-3)!} \Vert \partial_x^j [\TT^k, \nabla] (\epsilon \partial_t)^i S \Vert_{L^2}  . \end{split}    \llabel{ONi9jWzM1HHOlFm2TfMmujPhiVujhpDSS5vLdDjayX8BpBeeKnwz6Po6KUz01etrb2z8jwiJy9GB3F7RNKgEUWkFVuxBJpntmcqZF7NIVIiSWXEy2B7tRn6afnGtzp09oO1UsLvofZRiQcxV7tFgjBZ9mEssTsKSENx0N4gQ0ubFQGfRGM1dMTXJ53WWic7AlLkC76gLZm4Sn9zosCBe9fDlmE6Lq9lJY0heWK3oKvFizSvWqOwh8NvAIcoWO4QfR4CxSOGYXBYo2zigpP0nYii2uivvjhp9zpnfsROtkKm3afLxSh6D4eOoQ0TcRScNt9ZGUvxH0aRPrHhuZrJtGLHt533m8j3jJeSbCYI84loykM7iaY1pC09AJJ8puhnm6yTQO1WjgC5NvJJcAuK0Ib3YxCEleEHjdLl8L677tACt5w4MTfsZXC730oY7Lm9tGfKZQ3UFhPK7LqBEB1tnde3A05ZPCaVVMO4Wp1vl5NWYVAS3A6ppl8D702kndT3RzBToCoeQzgwwkUcHv6n8yGW6ADNujYLGNTMO4OrlG24pCbTFbaEWwLfvUlF7kpNGq0kmJd9voUPmdlK4yovDYNYPQhxO4g1sAZRKOsjpI9NvaBCW4EssGGiWLGR9RPwoE9NROT90cKFKyx5ZwzWBEnqOyITIPyL1ILo7eJfQdKiR1VJwiLx70OjSKHIShGIXvmm10FDxlINtLdHGzZ4OBmclD6VUmJvbO5I3EI1imd1hjxMLf9WNnq0cfLTeFzxKEQHL5u9r94VqDSQs1OhLadtGVbTmqRCtSDHNRTYERTDGHFHDFGSDFGRTYFGHDFGDFSFGWERTFSDFGDFADWERTFGHDFGHFGH60} \end{align} Proceeding as in \eqref{SDHNRTYERTDGHFHDFGSDFGRTYFGHDFGDFSFGWERTFSDFGDFADWERTFGHDFGHFGH106}--\eqref{SDHNRTYERTDGHFHDFGSDFGRTYFGHDFGDFSFGWERTFSDFGDFADWERTFGHDFGHFGH100}, we obtain \eqref{SDHNRTYERTDGHFHDFGSDFGRTYFGHDFGDFSFGWERTFSDFGDFADWERTFGHDFGHFGH107}. \par \emph{Proof of \eqref{SDHNRTYERTDGHFHDFGSDFGRTYFGHDFGDFSFGWERTFSDFGDFADWERTFGHDFGHFGH103}}: Using H\"older and Sobolev inequality, we obtain \begin{align} \begin{split} I_4  & \les \sum_{j+k+i \geq 5}  \sum_{\substack{(j', k', i') \leq (j,k,i)\\   j+k+i-2 \leq  j'+k'+i'  }} \left( \frac{\kappa^{(j'-1)_+} \bar{\kappa}^{k'} \tau^{j'+k'+i'-3}}{(j'+k'+i'-3)!} \Vert \partial_x^{j'} \TT^{k'} (\epsilon \partial_t)^{i'} v \Vert_{L^2} \right) \\&\indeqtimes \Vert \partial_x^{j-j'} \TT^{k-k'} (\epsilon\partial_t)^{i-i'} \nabla S \Vert_{H^2} \\ & \les \Vert v \Vert_{A(\tau)} , \end{split}    \llabel{4RQKTJXDb5X7vc4gRO9owaiUdwAc4uIKr2eA0fjhHu7HukD1pbwm7YzSrpRxdBpLYAGi7aIh0aZot6yMJxsfBpw3JbbH4lIXf9s3jFDyyYSC8wNFPYZiO2AJBZKdaAD3hu4WI0h1zMut1AyXXJ1IWtVOUvgjMEdr3H6ZASxZeGCNEG9U8XLLC8otnqMrr5llwno3yV0tSlPW1LjWvkj8W91VA8T6CXqWMs7WjvwZre71Buv1RkHBPRkBwhsucRyoHn8BL0fGPm3ANnQX1MqDgLrFJmPZ1sTLo46ZhffTFKGFDSIMoV0UtBu6d0TIkEwRX13S45chg6JyD2aHOEAbD98fDOYbXljuO3gIzi2BaEYczmnuOx94YNAzGeE0iHiqoXEtl8Ljgxh3GO5duEBzl0uoVOZLI0xiECVeNFhcQXgqqN2eSobAT4IXuflDOankxzPUd37eYwmvrud8eluLCBkwtTomGSSpWw7m10SlOpCJqqm6hQLau4qM8oyxlORGCT1csPUcGqV6zE9cPoOKWCTPvz9PMUTRd9e0f5g8B7nPURfC1t0u4HPVzeEgMRZ9ZVv1rWPIsPJf1FLiIa2TyZq8haK9V3qN8QsZt7HROzBAheG35fAkpgPKmIuZsYWbSQw6Dg25gM5HG2RWpuipXgEHDco7Ue6aFNVdPf5d27rYWJCOoz1ystngc9xlbcXZBlzAM9czXn1Eun1GxINAXz44Qh8hpfeXqdkdzuIRYMNMcWhCWiqZ5sIVyj9rh1A47uLkTwfyfryoiDx1enHSeip1vsO4KM6ISDHNRTYERTDGHFHDFGSDFGRTYFGHDFGDFSFGWERTFSDFGDFADWERTFGHDFGHFGH61} \end{align} where the last inequality follows from \eqref{SDHNRTYERTDGHFHDFGSDFGRTYFGHDFGDFSFGWERTFSDFGDFADWERTFGHDFGHFGH24}.  \par \emph{Proof of \eqref{SDHNRTYERTDGHFHDFGSDFGRTYFGHDFGDFSFGWERTFSDFGDFADWERTFGHDFGHFGH505}}: For low values of $j+k+i$, from \eqref{SDHNRTYERTDGHFHDFGSDFGRTYFGHDFGDFSFGWERTFSDFGDFADWERTFGHDFGHFGH24} and \eqref{SDHNRTYERTDGHFHDFGSDFGRTYFGHDFGDFSFGWERTFSDFGDFADWERTFGHDFGHFGH19}, we obtain \begin{align} \sum_{0\leq j+k+i \leq 4} \frac{\kappa^{(j-1)_+} \bar{\kappa}^k \tau^{(j+k+i-3)_+}}{(j+k+i-3)!} \Vert \partial_{x}^j [\TT^k, \nabla] (\epsilon \partial_t)^i S \Vert_{L^2} \les	 1. \label{SDHNRTYERTDGHFHDFGSDFGRTYFGHDFGDFSFGWERTFSDFGDFADWERTFGHDFGHFGH508} \end{align} For high values of $j+k+i$, by \eqref{SDHNRTYERTDGHFHDFGSDFGRTYFGHDFGDFSFGWERTFSDFGDFADWERTFGHDFGHFGH19}, we have \begin{align} \begin{split} 	& \sum_{ j+k+i \geq 5} \frac{\kappa^{(j-1)_+} \bar{\kappa}^k \tau^{(j+k+i-3)_+}}{(j+k+i-3)!} \Vert \partial_{x}^j [\TT^k, \nabla] (\epsilon \partial_t)^i S \Vert_{L^2} \\ &\indeq \les \sum_{ j+k+i \geq 5} \sum_{k'=1}^k \sum_{j'=0}^j \frac{\kappa^{(j-1)_+} \bar{\kappa}^k \tau^{(j+k+i-3)_+}}{(j+k+i-3)!}  \binom{k}{k'} \binom{j}{j'} (j' + k')! \bar{K}^{j'+k'}  \Vert \partial_x^{j-j'+1} \TT^{k-k'} (\epsilon \partial_t)^i S \Vert_{L^2} \\ &\indeq \les \sum_{ j+k+i \geq 5} \sum_{k'=1}^k \sum_{j'=0}^j \frac{(j+k)!(j+k+i-j'-k'-3)!}{(j+k-j'-k')!(j+k+i-3)!} (\kappa \bar{K})^{j'+k'}  \\ & \indeqtimes \left( \frac{\kappa^{j-j'} \bar{\kappa}^{k-k'} \tau^{(j+k+i-j'-k'-3)_+}}{(j+k+i-j'-k'-3)!}  \Vert \partial_x^{j-j'+1} \TT^{k-k'} (\epsilon \partial_t)^i S \Vert_{L^2} \right) , \end{split}	    \llabel{5nql8i6XMK0rks0J59i2gzeHrjL1GbnITIBUV33u1He2frwh8SlMFfxURXLEeSH3mRUZPo4qNIEiCDpZPg9jUAiCbJ3yZnYviypDY6UdI0D0Uvh8B5pzubybLzK6Okxiv1n6JFvXnAnLZkB2NBL7XrJXMMKs8CZsG212rnKHzisWyrbRdhAevTuCXC238cSHDCDRuSfLZ6OcD9ILJinlaki39ZA5i8M5PHQYRu9c4V1mhLVzeTTXVpgYHoa1Xj0tDpTKKrzAmUdJKYBVzYf2fuZHOiTzct88pOcWnFTvpWY7rmUX9mStlDy20kd9AusOJf9flOWhZwlWiHSoUyCdlxpPLdWIUhiT3LHokLLeFWnzMXpp1rcNzMdnmqvOOqqu8r9aHjZugynu1WR2Kt1bLAm4xFIJ56unzROpYf9QL0MlwkoJUIPqs7J4FuANaH5PLCga54bD9TmxtzHi82iDykpSWvkcvxuDk7rWAaUjbhIKIz4yML1wz7iFCxXO8nis8uKjb3uglgImLMaYSJSdAKU9AVHZ8G1n2uexpWzyZdj2gvCTeT5702uSeZ9BCuRQVhFJQNMZ3S2esCCUqV3CHgWPGx9gzqE69aQz5VOSuRPnXm57MllKGVzazF7XDlGQJ2Gawooq71nMVyltPDG7L0Ld0OErw0kLAYXZzR6XaeileVHuk7xcwyBOlV2JrQxWjPbeLaiIYqHqpbcvoLZVz1ytufBpcJ4MVecfeep7bry2UjHFLwzWFuMMdhVHxDX8Gsv0RQLK8YqbATHFEX0UkxjLiSKH3Rr5SDHNRTYERTDGHFHDFGSDFGRTYFGHDFGDFSFGWERTFSDFGDFADWERTFGHDFGHFGH65} \end{align} where the last inequality follows from \eqref{SDHNRTYERTDGHFHDFGSDFGRTYFGHDFGDFSFGWERTFSDFGDFADWERTFGHDFGHFGH25}. We then proceed as in \eqref{SDHNRTYERTDGHFHDFGSDFGRTYFGHDFGDFSFGWERTFSDFGDFADWERTFGHDFGHFGH106}--\eqref{SDHNRTYERTDGHFHDFGSDFGRTYFGHDFGDFSFGWERTFSDFGDFADWERTFGHDFGHFGH100} to obtain \begin{align} \sum_{ j+k+i \geq 5} \frac{\kappa^{(j-1)_+} \bar{\kappa}^k \tau^{(j+k+i-3)_+}}{(j+k+i-3)!} \Vert \partial_{x}^j [\TT^k, \nabla] (\epsilon \partial_t)^i S \Vert_{L^2} \les 1+\Vert S \Vert_{\tilde{A}(\tau)} . \label{SDHNRTYERTDGHFHDFGSDFGRTYFGHDFGDFSFGWERTFSDFGDFADWERTFGHDFGHFGH507} \end{align} Thus \eqref{SDHNRTYERTDGHFHDFGSDFGRTYFGHDFGDFSFGWERTFSDFGDFADWERTFGHDFGHFGH505} follows from \eqref{SDHNRTYERTDGHFHDFGSDFGRTYFGHDFGDFSFGWERTFSDFGDFADWERTFGHDFGHFGH508} and \eqref{SDHNRTYERTDGHFHDFGSDFGRTYFGHDFGDFSFGWERTFSDFGDFADWERTFGHDFGHFGH507}. \par Combining \eqref{SDHNRTYERTDGHFHDFGSDFGRTYFGHDFGDFSFGWERTFSDFGDFADWERTFGHDFGHFGH55} and \eqref{SDHNRTYERTDGHFHDFGSDFGRTYFGHDFGDFSFGWERTFSDFGDFADWERTFGHDFGHFGH104}--\eqref{SDHNRTYERTDGHFHDFGSDFGRTYFGHDFGDFSFGWERTFSDFGDFADWERTFGHDFGHFGH505}, we obtain \begin{align} \begin{split} \frac{d}{dt} \Vert S \Vert_{A(\tau)} & \les \Vert S \Vert_{\tilde{A}(\tau)}  \left( \dot{\tau}(t)  + \Vert v \Vert_{A(\tau)} + 1 \right) + \Vert S \Vert_{A(\tau)} + \Vert v \Vert_{A(\tau)} + 1 , \label{SDHNRTYERTDGHFHDFGSDFGRTYFGHDFGDFSFGWERTFSDFGDFADWERTFGHDFGHFGH108} \end{split} \end{align} where we used \eqref{SDHNRTYERTDGHFHDFGSDFGRTYFGHDFGDFSFGWERTFSDFGDFADWERTFGHDFGHFGH24} to bound $\Vert \nabla v \Vert_{L^\infty_x}$ and $\Vert v \Vert_{L^\infty_x}$.  Now, we choose $K$ in \eqref{SDHNRTYERTDGHFHDFGSDFGRTYFGHDFGDFSFGWERTFSDFGDFADWERTFGHDFGHFGH99} to be sufficiently large so that the term next to $\Vert S \Vert_{\tilde{A}(\tau)}$ is less than or equal to $0$. The lemma then follows by integrating \eqref{SDHNRTYERTDGHFHDFGSDFGRTYFGHDFGDFSFGWERTFSDFGDFADWERTFGHDFGHFGH108} on $[0,T_0]$ and using the Gronwall lemma. \end{proof} \par \startnewsection{Velocity estimates}{sec06} In this section, we use derivative reductions for the divergence, curl, and pure time derivative components of the velocity.  First, we split $\Vert v \Vert_{A(\tau)}$ as \begin{align} \Vert v \Vert_{A(\tau)} =  \sum_{l=1}^5 \sum_{J_l}  \frac{\kappa^{(j-1)_+} \bar{\kappa}^k \tau^{(i+j+k-3)_+}}{(i+j+k-3)!} \Vert \partial_x^{j} \TT^k (\epsilon \partial_t)^i v \Vert_{L^2}  = \sum_{l=1}^5 U_l , \label{SDHNRTYERTDGHFHDFGSDFGRTYFGHDFGDFSFGWERTFSDFGDFADWERTFGHDFGHFGH183} \end{align} where \begin{align} U_{l} = \sum_{J_l} \frac{\kappa^{(j-1)_+} \bar{\kappa}^k \tau^{(i+j+k-3)_+}}{(i+j+k-3)!} \Vert \partial_x^j \TT^k (\epsilon \partial_t)^i v \Vert_{L^2}  \comma l=1,\ldots, 5,    \llabel{umoxSEQd0Myjy6DdKVFmCU7TfBI0xlsOWk20HafdxvncOzFEKskwmIHNHlUNhKYaDnErMs9hNUpc2L8ipcs8XwvjOeX35ehokgG6nyzyevqmZYDfdhQXEFk8qgF3J8SySTL1emZ1hqVGLs1rRD8Fs6u4NCoeyUCERC0xh4x67erCg8lf8J1Go29fRN6HMuUIwacxm2Af9sfbNov4CqQKcFloP7BRW85rUNd2BgxWkOpDr11KRKZlaHXQoN7OaohaotKxJGOUty96bCDK33hIEKUogH2RwEYISuKzS4pVz7BzdHrIjRgCvSJUs1x1D3RRqk9icqLUQ9TeTiCnifS9NpGBJFlPNUFbB5IU956UG9PrSw9oNNnvK4g3GFnUxtWPNXVvue5rAIN0eoRq2XpZNDRpQUwIjhrAXb13L7G4OSy8fznjOJBbaVLPazvFW8uCY5if8VkV0vydXHDDIU8ga2EpiAyYYWhIKYvQ5Xjd850u0AyrPIbkOEaxPZDHmR249GiAnrkNORVg45eQC6GzWVplxUGMUR37sseSzNrlavQzMRpUHdQ1jeJFhRqYlOTMQ1sTicKRHoxsD02Ase6ebYrNxUezPV1IhRwGxDgjslQdbWcoICqaGYMERvugV6kyVQl1l68rTm1okgKupWNIfnHZ4RmbdWEJGBaQ4CIXFDzckkQ1BqWuyqxN3N6kj6LoxcvV46tplSOTvyAQ4vQ0ZZVeCxCBvu6O4OkIGj6IXwLBiaouAtrSlFz795kew70VnFBnjEfMfiSep8xiXv83I36SZgoH8sv2SDHNRTYERTDGHFHDFGSDFGRTYFGHDFGDFSFGWERTFSDFGDFADWERTFGHDFGHFGH62} \end{align} and \begin{align*} & J_1 = \{(j,k,i) \in \mathbb{N}_0^3 \colon j\geq 2 \} \comma &&   J_2 = \{(j,k,i) \in \mathbb{N}_0^3 \colon j = 1, k \geq 1 \}  , \\ & J_3 = \{(j,k,i) \in \mathbb{N}_0^3 \colon j = 1, k =0 \}  \comma && J_4 = \{(j,k,i) \in \mathbb{N}_0^3 \colon j = 0, k \geq 1 \} \comma \\ & J_5 = \{(j,k,i) \in \mathbb{N}_0^3 \colon j = 0, k =0\} . \end{align*} For simplicity of notation, we abbreviate, for $(j,k,i)\in \mathbb{N}_0^3$,   \begin{align}    U_{j,k,i}    =    \frac{\kappa^{(j-1)_+} \bar{\kappa}^k \tau^{(i+j+k-3)_+}}{(i+j+k-3)!}    \Vert \partial_x^{j} \TT^k (\epsilon \partial_t)^i v \Vert_{L^2}     \llabel{BrtlOu8jTJ0ugcT08dORqnA2UzjWI4L5JjM3LQdgsTyV3hNC72Q89Qb5bxur3w6rjyNl9kDte492PJf3TxTDd17VbnAPKptpusRQuHiBwvO39DUmCifhow4RUvFoejBw1JOq8mHPwCfIfh3uUxOVXLoz1E1d0V2KbeCLL9M3pMAkLuAcb6KUHnOXqdPGATF1Lrsh5tpx3OZIAD0H5nDGHlXBqTqgyAEA6pbVZcNjRE1BkH3omJFFjm9TJA7NUBmtg5ppAwIh190lJCmYeih6JWiCfyDWAByAbg61BS14QNzvZSQgqLjGvIW7Vz37vT4dZ92l3DdxftojupoeKITksc2uF2YnBJtMrrNsDJ0hPEK2h7KFiQmbEzr1TClt50d3LR9HDyUIetgwmyKv6NMDAzFDvMoxIYoMiKt9ZPAdaDYug53lgYerHdqgXX70KpLETmzD24CryGGRgzrNZVs78R7S0k0ji59HqHYiB0KvtZrqjtqBxRqbcLz69t0O7ceIlLvMkDc725QVgUMo6uO72ftjKNmef1dr5yiZKljQgHKgliZZQtkheh8wZ7ZRoSGVax9RKHQwgBEvzRwdGNcEPCGYhatb2jZ6XmUdNpYdg7l7sBfrhGmlueYkHA4vIesnrWIqxLSPzSnDFx7RPCfsGeziz44BxQVkm9edBAEhMdruUNIa03TUvnhkRmygOHTcTfiLkfUgYhruxfmyqLiFRozIcHkCC0kXQJo5C3jsfLMJe0vbW57SYQbbWW1ovoL2RRm7yIdmwpXqneEKrwBvjvor3pvUVDD8SDHNRTYERTDGHFHDFGSDFGRTYFGHDFGDFSFGWERTFSDFGDFADWERTFGHDFGHFGH64}   \end{align} for the velocity coefficients arising in expansions below.  Similarly, we denote $E= \{(j,k,i)\in \mathbb{N}_0^3 \colon j\geq 1\}$, and we write \begin{align} D_{j,k,i} = \frac{\kappa^{(j-1)_+} \bar{\kappa}^k \tau^{(i+j+k-3)_+}}{(i+j+k-3)!} \Vert \partial_x^{j-1} \TT^k (\epsilon \partial_t)^i \dive v \Vert_{L^2}    \llabel{BBsuq9OobZWmwy1ql338ivREeCOsJK7Tjax80TdcrxrU0WkrjidNQHLObsVEMR5ih9SmKjrgZCMDUD28M5TL5qTCSbIe9MYP3ARqOLxQH60dvw1XCvSnpTBcaBAMObspru0BDxjVUWWfMp10GNLtfRogR2TlC4KKhCbYsRCGh7FNpjgdSCHUzEKRsAcK2iZG9wFpgwCPUi7jh7fw5OfIJ6iw8EU7ECMBzzhKBEAOXFSz3qsSMv5pZldBJ6pQZcQ8ISDMzGh2GQDm2Rco69Z0zEEPSymYZK9vztjRhCxx980eqRB3C9DimZOyq936zi2X9pjxaf4T8T1yNdIFsdPMiU7gXdeHctkQGkJfjXs4uoBxZA2vY4sWxhqgNJoUNvlSsDq3xWsOjrc1Voi6eienlnGHf8UYCUexIgLa3wPm7A37TKTNSTsv5s2pLuyK6jXKqclcpnPEmofpXpXzVtuwpFO8Usppuucs3ZyVMGgG4zKXz1STngNjSRReL07cWjObWU2d1wq5k2UTAHguogqzVj9zR5cltRIRgGaJgQ2RRlMvdikH6UQiHfOJoTsCM5W2y6iNn1xpobqmpgwjyxoRsdmPboozEAyUoTmw1tQmqP9EwyL8HaGxAc0UhK2znvVWcv4xF5xAYXXjpJmE5fV2hUn6jt2Lt2npfeIdgLKL0xSUDMR51Rzg6r6qUCMcV24DAShdpmrmvene8eSZaP5BbbPPSB6eOsVJOV4dmUiZ2mEwiwTNt1JDDwWjT4mfF4jZC9hko7giefw6w1vmlCzxs6ijEv9M3OtrSDHNRTYERTDGHFHDFGSDFGRTYFGHDFGDFSFGWERTFSDFGDFADWERTFGHDFGHFGH66} \end{align} for the coefficients containing the divergence and \begin{align} C_{j,k,i} = \frac{\kappa^{(j-1)_+} \bar{\kappa}^k \tau^{(i+j+k-3)_+}}{(i+j+k-3)!} \Vert \partial_x^{j-1} \TT^k (\epsilon \partial_t)^i \curl v \Vert_{L^2}  ,    \llabel{OxW2pkMWi7zFKYpA0DcFZpnNS1Lk1rPOTNRZNpe0OWgJK5agU63LxqcHw0iSBcYqNRaD38nPvGHselXU3x8bUZuje0xiRm3iRAbwtJBlaV4M2PJMq7jb3KYDNMbl3CVKrZZDRnV4A65HtZiRodJP0T4NFm1G3ukOBMuK52ljg5v9tBsFjgCRa68HjYa694ZjFJ6CVArUC2OVqxQcnOtix4DXkXaU86p3k1dQqfJUcTYg9MU1RGZFJPXYBCZpgFVfYWG2vVeUqL8QD5IkcRtJuuS2jV2j9re9nJq1j1LPnGXB77qxq0Y2azEuZ4u441ZDhP2c6ltZfe7sMqQEKxXmdTmc3NG1cbV4nmp3bGwOzqXPVe7sEDvoeN9DmlBeZGcrMyDBgyvyBb0YaeHbb2PlECWXTTua3s6XI7ft2Ah6eEUiUur3aiIjiIQx2c3fDbvdyB7wo4oR3i1lrvQS8HjwHsbZzQ47M5Upnd0q5kKW6ZWnGPJtIEdnmGALnep3huZwDANUo8G1vcBzxXVdpxFBpUJlItAEBDsjvloldGbWbqs2ZYgwIZ5URD3kPbYDmbFN9fu026aB6pz5DgpDwrEGQ0Fqd0SE0sQdQAr8OTgO4fpavWhGZygRhreK5l6IFBWjy60VkeHEnXbH5mSL45RJe4oS9aMVTcaRLS768nDfgaCeH9JEvoHAo0guEz8XScEWRDNN53wJbB2uFhfL6zuTGzMbitUjWdxZ6a2TIfhtIyuR5IgpZCA5HZMZT9SvLIqSwDe8gBxlEdFkyj2QyuYzwlwpxYas2Xz8SDHNRTYERTDGHFHDFGSDFGRTYFGHDFGDFSFGWERTFSDFGDFADWERTFGHDFGHFGH67} \end{align} for those with the curl.  \par \emph{The term $U_1$}: First we estimate the sum $U_1$.  By \eqref{SDHNRTYERTDGHFHDFGSDFGRTYFGHDFGDFSFGWERTFSDFGDFADWERTFGHDFGHFGH53}, \def\eprtudfsigjsdkfgjsdfg{\partial}we have \begin{align} \begin{split} U_1  & \lesssim \frac{\kappa}{\bar{\kappa}}   U_1 + \kappa U_1 + \sum_{E} D_{j,k,i} + \sum_{E} C_{j,k,i} \\&\indeq\indeq + \sum_{J_1} \frac{\kappa^{(j-1)_+} \bar{\kappa}^k \tau^{(i+j+k-3)_+}}{(i+j+k-3)!} \Vert \eprtudfsigjsdkfgjsdfg_x^{j-1} [\TT^k, \dive] (\epsilon \eprtudfsigjsdkfgjsdfg_t)^i v \Vert_{L^2} \\ &\indeq\indeq + \sum_{J_1} \frac{\kappa^{(j-1)_+} \bar{\kappa}^k \tau^{(i+j+k-3)_+}}{(i+j+k-3)!} \Vert \eprtudfsigjsdkfgjsdfg_x^{j-1} [\TT^k, \curl] (\epsilon \eprtudfsigjsdkfgjsdfg_t)^i v \Vert_{L^2} \\ &\indeq\indeq + \sum_{J_1} \frac{\kappa^{(j-1)_+} \bar{\kappa}^k \tau^{(i+j+k-3)_+}}{(i+j+k-3)!} \Vert [\TT, \eprtudfsigjsdkfgjsdfg_x^{j-2}] \TT^k (\epsilon \eprtudfsigjsdkfgjsdfg_t)^i v \Vert_{L^2} \\ &\indeq\indeq + \sum_{J_1} \frac{\kappa^{(j-1)_+} \bar{\kappa}^k \tau^{(i+j+k-3)_+}}{(i+j+k-3)!} \Vert  \eprtudfsigjsdkfgjsdfg_x [\TT, \eprtudfsigjsdkfgjsdfg_x^{j-2}] \TT^k (\epsilon \eprtudfsigjsdkfgjsdfg_t)^i v \Vert_{L^2} \\ & = I_{11} + I_{12} + I_{13} + I_{14} + I_{15} + I_{16} + I_{17} + I_{18} . \label{SDHNRTYERTDGHFHDFGSDFGRTYFGHDFGDFSFGWERTFSDFGDFADWERTFGHDFGHFGH160} \end{split} \end{align} For the term $I_{15}$, we use \eqref{SDHNRTYERTDGHFHDFGSDFGRTYFGHDFGDFSFGWERTFSDFGDFADWERTFGHDFGHFGH18} to write \begin{align} \begin{split} I_{15} & \lesssim \sum_{J_1} \sum_{k'=1}^k \sum_{j'=0}^{j-1}  \binom{k}{k'} \binom{j-1}{j'} \frac{\kappa^{j-1} \bar{\kappa}^k \tau^{(i+j+k-3)_+}}{(i+j+k-3)!} \\&\indeqtimes (j'+k')! \bar{K}^{j'+k'} \Vert \eprtudfsigjsdkfgjsdfg_x^{j-j'} \TT^{k-k'} (\epsilon \eprtudfsigjsdkfgjsdfg_t)^i v \Vert_{L^2} \\ & \lesssim \sum_{i=0}^\infty \sum_{k''=0}^\infty \sum_{j''=0}^\infty \frac{\kappa^{j''} \bar{\kappa}^{k''} \tau^{(i+j''+k''-2)_+}}{(i+j''+k''-2)!} \Vert \eprtudfsigjsdkfgjsdfg_x^{j''+1} \TT^{k''} (\epsilon \eprtudfsigjsdkfgjsdfg_t)^i v \Vert_{L^2} \\ &\indeqtimes \sum_{j=j''+1}^\infty \sum_{k=k''+1}^\infty \frac{(k+j-1)! (i+j''+k''-2)!}{(j''+k'')! (i+j+k-3)!}  \kappa^{j-j''-1}\bar{\kappa}^{k-k''}  \bar{K}^{k-k''+j-j''-1} . \label{SDHNRTYERTDGHFHDFGSDFGRTYFGHDFGDFSFGWERTFSDFGDFADWERTFGHDFGHFGH81} \end{split} \end{align} Note that the sum in $j,k$ is dominated by \begin{align} \bar{\kappa} \bar{K} \sum_{j=j''+1}^\infty \sum_{k=k''+1}^\infty \frac{(k+j-1)! (i+j''+k''-2)!}{(j''+k'')! (i+j+k-3)!}  \kappa^{j-j''-1}\bar{\kappa}^{k-k''-1}   \bar{K}^{k-k''+j-j''-2} \les \bar{\kappa}  , \label{SDHNRTYERTDGHFHDFGSDFGRTYFGHDFGDFSFGWERTFSDFGDFADWERTFGHDFGHFGH82} \end{align} uniformly for all $i, k'', j'' \in \mathbb{N}_0$, by taking $\bar{\kappa} \leq 1/CK$. Thus, from \eqref{SDHNRTYERTDGHFHDFGSDFGRTYFGHDFGDFSFGWERTFSDFGDFADWERTFGHDFGHFGH81}--\eqref{SDHNRTYERTDGHFHDFGSDFGRTYFGHDFGDFSFGWERTFSDFGDFADWERTFGHDFGHFGH82} we obtain \begin{align} I_{15} \lesssim \bar{\kappa} U_1 + \bar{\kappa} U_2 + \bar{\kappa} U_3 . \label{SDHNRTYERTDGHFHDFGSDFGRTYFGHDFGDFSFGWERTFSDFGDFADWERTFGHDFGHFGH161} \end{align} The term $I_{16}$ can be treated analogously as $I_{15}$. Namely, using \eqref{SDHNRTYERTDGHFHDFGSDFGRTYFGHDFGDFSFGWERTFSDFGDFADWERTFGHDFGHFGH50}, we get \begin{align} I_{16} \lesssim \bar{\kappa} U_1 + \bar{\kappa} U_2 + \bar{\kappa} U_3 . \label{SDHNRTYERTDGHFHDFGSDFGRTYFGHDFGDFSFGWERTFSDFGDFADWERTFGHDFGHFGH162} \end{align} For the term $I_{17}$, we appeal to \eqref{SDHNRTYERTDGHFHDFGSDFGRTYFGHDFGDFSFGWERTFSDFGDFADWERTFGHDFGHFGH22} obtaining \begin{align} \begin{split} I_{17} & \lesssim \sum_{J_1} \sum_{j'=1}^{j-2}  \binom{j-2}{j'} j'! \bar{K}^{j'}  \frac{\kappa^{j-1} \bar{\kappa}^k \tau^{(i+j+k-3)_+}}{(i+j+k-3)!} \Vert \eprtudfsigjsdkfgjsdfg_x^{j-j'-1} \TT^k (\epsilon \eprtudfsigjsdkfgjsdfg_t)^i  v \Vert_{L^2}  \\ & \lesssim \sum_{i=0}^\infty \sum_{k=0}^\infty \sum_{j''=0}^\infty  \left( \frac{\kappa^{j''} \bar{\kappa}^k \tau^{(i+j''+k-2)_+}}{(i+j''+k-2)!} \Vert \eprtudfsigjsdkfgjsdfg_x^{j''+1} \TT^k (\epsilon \eprtudfsigjsdkfgjsdfg_t)^i v \Vert_{L^2}  \right) \\ &\indeqtimes \sum_{j=j''+3}^\infty  \frac{(i+j''+k-2)!(j-2)!}{(i+j+k-3)!j''!} \kappa^{j-j''-1} \bar{K}^{j-j''-2} . \label{SDHNRTYERTDGHFHDFGSDFGRTYFGHDFGDFSFGWERTFSDFGDFADWERTFGHDFGHFGH83} \end{split} \end{align} Note that the sum in $j$ is dominated by \begin{align} 	\kappa 	\sum_{j=j''+3}^\infty  	\frac{(i+j''+k-2)!(j-2)!}{(i+j+k-3)!j''!}  	\kappa^{j-j''-2} \bar{K}^{j-j''-2}  	\les 	\kappa 	, 	\label{SDHNRTYERTDGHFHDFGSDFGRTYFGHDFGDFSFGWERTFSDFGDFADWERTFGHDFGHFGH84} \end{align} uniformly for all $i,k, j''\in \mathbb{N}_0$, by taking $\kappa \leq 1/C\bar{K}$. From \eqref{SDHNRTYERTDGHFHDFGSDFGRTYFGHDFGDFSFGWERTFSDFGDFADWERTFGHDFGHFGH83}--\eqref{SDHNRTYERTDGHFHDFGSDFGRTYFGHDFGDFSFGWERTFSDFGDFADWERTFGHDFGHFGH84}, we obtain \begin{align} I_{17}  \lesssim \kappa U_1 + \kappa U_2 + \kappa U_3 . \label{SDHNRTYERTDGHFHDFGSDFGRTYFGHDFGDFSFGWERTFSDFGDFADWERTFGHDFGHFGH163} \end{align} The term $I_{18}$ can be estimated analogously as $I_{17}$, and we arrive at \begin{align} I_{18} \lesssim \kappa U_1 + \kappa U_2 + \kappa U_3 . \label{SDHNRTYERTDGHFHDFGSDFGRTYFGHDFGDFSFGWERTFSDFGDFADWERTFGHDFGHFGH164} \end{align} Collecting the estimates \eqref{SDHNRTYERTDGHFHDFGSDFGRTYFGHDFGDFSFGWERTFSDFGDFADWERTFGHDFGHFGH160}, \eqref{SDHNRTYERTDGHFHDFGSDFGRTYFGHDFGDFSFGWERTFSDFGDFADWERTFGHDFGHFGH161}--\eqref{SDHNRTYERTDGHFHDFGSDFGRTYFGHDFGDFSFGWERTFSDFGDFADWERTFGHDFGHFGH162}, and \eqref{SDHNRTYERTDGHFHDFGSDFGRTYFGHDFGDFSFGWERTFSDFGDFADWERTFGHDFGHFGH163}--\eqref{SDHNRTYERTDGHFHDFGSDFGRTYFGHDFGDFSFGWERTFSDFGDFADWERTFGHDFGHFGH164}, we obtain \begin{align} U_1  \les \left(  \frac{\kappa}{\bar{\kappa}} + \kappa + \bar{\kappa} \right) U_1 + (\kappa+ \bar{\kappa} ) U_2 + (\kappa+ \bar{\kappa}) U_3 + \sum_{E} D_{j,k,i} + \sum_{E} C_{j,k,i} . \label{SDHNRTYERTDGHFHDFGSDFGRTYFGHDFGDFSFGWERTFSDFGDFADWERTFGHDFGHFGH174} \end{align}   \par \emph{The term $U_2$}: Next, we estimate the sum $U_2$. By \eqref{SDHNRTYERTDGHFHDFGSDFGRTYFGHDFGDFSFGWERTFSDFGDFADWERTFGHDFGHFGH14}, we have \begin{align} \begin{split} U_2 & \lesssim \sum_{E} D_{j,k,i} + \sum_{E} C_{j,k,i} + U_4 + \sum_{k=1}^\infty \sum_{i=0}^\infty \sum_{l=1}^k   \binom{k}{l} \frac{\bar{\kappa}^k \tau^{(k+i-2)_+}}{(k+i-2)!} \Vert \eprtudfsigjsdkfgjsdfg_x \TT^{k-l} (\epsilon \eprtudfsigjsdkfgjsdfg_t)^i v \Vert_{L^2} \Vert \TT^l \nu \Vert_{L^\infty(\Omega_{\delta_0})} \\ &\indeq + \sum_{k=1}^\infty \sum_{i=0}^\infty \sum_{l=1}^k   \binom{k}{l} \frac{\bar{\kappa}^k \tau^{(k+i-2)_+}}{(k+i-2)!} \Vert \TT^{k-l} (\epsilon \eprtudfsigjsdkfgjsdfg_t)^i v \Vert_{L^2} \Vert \eprtudfsigjsdkfgjsdfg_x \TT^l \nu \Vert_{L^\infty(\Omega_{\delta_0})} \\ &\indeq + \sum_{k=1}^\infty \sum_{i=0}^\infty \sum_{l=1}^k   \binom{k}{l} \frac{\bar{\kappa}^k \tau^{(k+i-2)_+}}{(k+i-2)!} \Vert \TT^{k-l} (\epsilon \eprtudfsigjsdkfgjsdfg_t)^i v \Vert_{L^2} \Vert \TT^l \nu \Vert_{L^\infty(\Omega_{\delta_0})} \\ &\indeq + \sum_{k=1}^\infty \sum_{i=0}^\infty \frac{\bar{\kappa}^k \tau^{(k+i-2)_+}}{(k+i-2)!} \Vert [\TT^k, \dive] (\epsilon \eprtudfsigjsdkfgjsdfg_t)^i v \Vert_{L^2}  \\ &\indeq + \sum_{k=1}^\infty \sum_{i=0}^\infty \frac{\bar{\kappa}^k \tau^{(k+i-2)_+}}{(k+i-2)!} \Vert [\TT^k, \curl] (\epsilon \eprtudfsigjsdkfgjsdfg_t)^i v \Vert_{L^2}  \\ & = I_{21} + I_{22} + I_{23} + I_{24} + I_{25} + I_{26} + I_{27} + I_{28} . \label{SDHNRTYERTDGHFHDFGSDFGRTYFGHDFGDFSFGWERTFSDFGDFADWERTFGHDFGHFGH151} \end{split} \end{align} We split the sum $I_{24}$ according to the values of $k$ and $i$, obtaining \begin{align} \begin{split} I_{24} & = \sum_{k=1}^\infty \sum_{i=2}^\infty \sum_{l=1}^k \binom{k}{l} \frac{ \bar{\kappa}^{k} \tau^{(k+i-2)_+}}{(k+i-2)!} \Vert \eprtudfsigjsdkfgjsdfg_x \TT^{k-l} (\epsilon \eprtudfsigjsdkfgjsdfg_t)^i v \Vert_{L^2} \Vert \TT^l \nu \Vert_{L^\infty(\Omega_{\delta_0})} \\&\indeq + \sum_{k=2}^\infty \sum_{l=1}^k \binom{k}{l} \frac{ \bar{\kappa}^{k} \tau^{(k-2)_+}}{(k-2)!} \Vert \eprtudfsigjsdkfgjsdfg_x \TT^{k-l} v \Vert_{L^2} \Vert \TT^l \nu \Vert_{L^\infty(\Omega_{\delta_0})} \\&\indeq + \sum_{k=2}^\infty \sum_{l=1}^k \binom{k}{l} \frac{ \bar{\kappa}^{k} \tau^{(k-1)_+}}{(k-1)!} \Vert \eprtudfsigjsdkfgjsdfg_x \TT^{k-l} (\epsilon \eprtudfsigjsdkfgjsdfg_t) v \Vert_{L^2} \Vert \TT^l \nu \Vert_{L^\infty(\Omega_{\delta_0})} \\&\indeq + \sum_{i=0}^1 \frac{\bar{\kappa}\tau^{(i-1)_+}}{(i-1)!} \Vert \eprtudfsigjsdkfgjsdfg_x (\epsilon \eprtudfsigjsdkfgjsdfg_t)^i v \Vert_{L^2} \Vert \TT \nu \Vert_{L^\infty (\Omega_{\delta_0})} \\& = I_{241} + I_{242} + I_{243} + I_{244} . \end{split}    \llabel{vCmy2027Jm0LHDMi8X5I255a2UfDR4mcoTUjWxZ3ZJ03alP15KPWffdCPQvOfakMNcpuVDt4cjR1oY7dqRGJvGT0ikreblMGbaNXGb8OEx6aIsztj7eeTt9OKLBFuMQbPLyqTdkjPdF7DNGeSMBOEYZ0aHFSUGw6GNIl3rCv2gKZvonEoA4IixK1D6RmzZgMEtWfXdAB9HMHCjO0woSnFR5ouCA2ug9Qr9w68xDKat9r3pnFjuBuaLmuZkyxCOZpfAbl4tXoggaAaws7Te4IOHrgzIvJq03hLQ585CRSLUZAxXz6RjOpeSx7BRT7txkA3dYoipdOpjMYASRDUXR02w2BAyxHnxBn5huR3MgQzB8mVlx3oBYOTTUMygvzKk6xrgiJZBGwQeoe3ToBsNhgzl64FVnOdf74cinPCeqU7hVu03CSpkpK0SZ9faBxRdRWaIz7Qlt63qkh0bbrEZ2FzUfvSD29l3cwm9uA1zYlTAONVdcOOfYcDS6Hgl4QNmCuPdgFEpuc4nz4vO8N6HC52kSn7ZMrzuz2zK9tk6DAjaUW6vghspP3MDTE9TD8WbIreHTM8VaqwE1kMM3vDcFVic5wDMXThn1NdqgIJZEwPXjazE6sa9UMIlsAZS4EF2vehcxfoAuLVxunTvRTKawGcHoviaXbtCuyMVT7DQdbuGB6WemA8AO9XEFkO8f6ArkAbi397pw9UoR7FCtPGytKq01TEmnbbNVQOuWDUJdgzsaNJLPMAg1EYKs7vNI45FZxQzabloG5bT4UMeFFsgh4PzviDrNUmfuDSDHNRTYERTDGHFHDFGSDFGRTYFGHDFGDFSFGWERTFSDFGDFADWERTFGHDFGHFGH68} \end{align} For the term $I_{241}$, we estimate \begin{align} \begin{split} I_{241} & \lesssim \sum_{k=1}^\infty \sum_{i=2}^\infty \sum_{l=1}^k \frac{\bar{\kappa}^l  \tau^l}{\tilde{\eta}^l} \left( \frac{\bar{\kappa}^{k-l} \tau^{(k+i-l-2)_+}}{(k+i-l-2)!} \Vert \eprtudfsigjsdkfgjsdfg_{x} \TT^{k-l} (\epsilon\eprtudfsigjsdkfgjsdfg_t)^i v \Vert_{L^2} \right) \left( \frac{\tilde{\eta}^{l}}{(l-3)!} \Vert \TT^l \nu \Vert_{L^\infty (\Omega_{\delta_0})} \right) \\ &\indeqtimes \frac{k!(k+i-l-2)!(l-3)!}{l!(k-l)!(k+i-2)!}  \\& \lesssim \bar{\kappa} \sum_{k=1}^\infty \sum_{i=2}^\infty \sum_{l=1}^k \frac{ \bar{\kappa}^{l-1}}{\tilde{\eta}^{l-1}} \left( \frac{\bar{\kappa}^{k-l} \tau^{(k+i-l-2)_+}}{(k+i-l-2)!} \Vert \eprtudfsigjsdkfgjsdfg_{x} \TT^{k-l} (\epsilon\eprtudfsigjsdkfgjsdfg_t)^i v \Vert_{L^2} \right) \left( \frac{\tilde{\eta}^{l}}{(l-3)!} \Vert \TT^l \nu \Vert_{L^\infty (\Omega_{\delta_0})} \right) . \llabel{EumNFpNWbgBdstqQkxVgebWnkXxwEx5mnT5ClbHNY3dv7bvSXqVAKwZq9OZCVhcnFplqS46GIDIS6gByGMKMuec2QsaXEt6nqtuzWxEMcmWcY7Yo1TKsEhXIlBg4dU6fdJT1HYhGyNpqvl2LQazFBoUQpdn8OiN1cGW8lFI3sKQR4EpX4Sfs7d6xxt8hoY8FtaiiwWROI3r1shgHgY6CFK2738bWZCkOb3oNdMQ3Ex3PGeYrE1574bD7yGiOpfQeHH5HO7RQtOJ9ZZpbtYpaHlL9C0pbmVtRSVOzKACwQMj4Q4o4eJArRdciatphUeJADtQdkWk6m70KApmTV15qYnboDBktnk4ukg8SPIjOm3loj895BhImn8bXAUAmLWRxp2cOFiWNjXPoqOHNnZN9WqkEUNHh13JHCpYI5qYrOwvGfgUxhi4UTdN3VMSiRTvvH7xsEtlI5jkmHa35re1gHeq4LViMpiUNoWhTp9zvYhGulekd0rsCBZfjeHQGyhWiEQX1Jgrj0L9iprpDqFNiZsFSBY0Jz81ojF3V4dZnehUR0JXtNXswcRq7BspAp8od99NsP3ssZex8aSlCgDdurUaYHI1SdqCnaMKPrTWQCROzhuZqHxPEGLM9BBfabHUiSq4Qy1LPhQWaIWYo9IakJbRNCxHhp6eiWv7p7JUJgSBNgUSVZsHUJ3oeV1uWw9Aeeixbp8Axt5ZgQPREAQsA0iwxlNK0td3JrPygKrvUR0YXdIlWwJdEao9WMuReYfQUqUoebEgHTLIcL62xVJZhnGxEHyBRff4wSDHNRTYERTDGHFHDFGSDFGRTYFGHDFGDFSFGWERTFSDFGDFADWERTFGHDFGHFGH150} \end{split} \end{align} We choose $\bar{\kappa} \leq \tilde{\eta}$ and use \eqref{SDHNRTYERTDGHFHDFGSDFGRTYFGHDFGDFSFGWERTFSDFGDFADWERTFGHDFGHFGH187} to get \begin{align} \begin{split} I_{241} \lesssim \bar{\kappa} U_2 + \bar{\kappa} U_3 . \end{split}    \llabel{T9LuePpwjrpRogWJ0Iv6sVWMIDiLzotM3ZrY4UKtDEJo2JotOUhTMblm3oNAr4FUYddaug68TX9BgoI0N0C5ySekTx6JxUvudb3KmMocy1WtAyC1uaC9rD3l579WJJ59LQaO7hxpDHsE1LtilFTP6jhASQePrSiosf5TsUVjOPX5qZZo0m0dzwu6ctAw6a3boSxFqDWJauvpSgcgUo3aPoBwWE2BJ9GAlK3c0ujcsQFNb4yCzpRUiQeZEDLIkzqkkEGLKC45BngCXim0yQL5UY0jWzkX2gvw2ui3S30GXpRRuRIOYERuRMD0bBNOb6yvIpeNpLKyV1DqnTur1ckV4yM5Z0dh7u9KmsrPGrxRkRPb09I1YXI9IMf6kPUDYOz9P535CU39tLxBPHIq5gL4wd3S87hiGY0diFePnJBIRNwR60scpaOS5GkEpCyWMEXyMH0v86FRcr8aQYquO5ZZda1IfuGzD9TN1PQLNYFtxB6NX3yOIAULFEswa9VuVUzvoN7i8Au67gdRnyN3ntXN75CDXmhDVWSDzclZWQ3D7FniEs4KRwvAMdkgXKS8feURSISDF3KoWBCeW8fkYNgH8ESnonm8fw8AhCulfSYR4LvuF5xYNuBFM0s8Q9fVrVjS6qkEdI9B7AaABnFEHyQ4FE3p2zNYbeCWWsjICZdMiQxwExryPWHPEdyQb0vtY4GOLLAXSYFDdvhbwojup5CPrRfoTWAfFdeQy4QwgGS8Ss4K6UrdkhKQBOIMpMAPsjTHPwz4WqITmhrrrrh2ZJ2AmbTsaJHj3QxISDHNRTYERTDGHFHDFGSDFGRTYFGHDFGDFSFGWERTFSDFGDFADWERTFGHDFGHFGH71} \end{align} The terms $I_{242}$, $I_{243}$, and $I_{244}$ can be estimated analogously to $I_{241}$. Thus we arrive at \begin{align} I_{24}  \lesssim \bar{\kappa} U_2 + \bar{\kappa} U_3 . \label{SDHNRTYERTDGHFHDFGSDFGRTYFGHDFGDFSFGWERTFSDFGDFADWERTFGHDFGHFGH152} \end{align} Similarly, the terms $I_{25}$ and $I_{26}$ are treated analogously to $I_{24}$, obtaining \begin{align} I_{25} + I_{26} \lesssim \bar{\kappa} U_2 + \bar{\kappa} U_3 . \label{SDHNRTYERTDGHFHDFGSDFGRTYFGHDFGDFSFGWERTFSDFGDFADWERTFGHDFGHFGH153} \end{align} For the term $I_{27}$, we use \eqref{SDHNRTYERTDGHFHDFGSDFGRTYFGHDFGDFSFGWERTFSDFGDFADWERTFGHDFGHFGH18} and obtain \begin{align} \begin{split} I_{27} & \lesssim \sum_{k=1}^\infty\sum_{i=0}^\infty\sum_{k'=1}^k  \binom{k}{k'}  \frac{\bar{\kappa}^k \tau^{(i+k-2)_+}}{(k+i-2)!} k'! \bar{K}^{k'} \Vert \eprtudfsigjsdkfgjsdfg_x \TT^{k-k'} (\epsilon \eprtudfsigjsdkfgjsdfg_t)^i v \Vert_{L^2} \\ & \lesssim \sum_{i=0}^\infty \sum_{k''=0}^\infty  \frac{\bar{\kappa}^{k''} \tau^{(i+k''-2)_+}}{(i+k''-2)!} \Vert \eprtudfsigjsdkfgjsdfg_x \TT^{k''} (\epsilon \eprtudfsigjsdkfgjsdfg_t)^i v \Vert_{L^2} \times \sum_{k=k''+1}^\infty \frac{k! (i+k''-2)!}{k''! (k+i-2)!}  \bar{\kappa}^{k-k''} \bar{K}^{k-k''} . \label{SDHNRTYERTDGHFHDFGSDFGRTYFGHDFGDFSFGWERTFSDFGDFADWERTFGHDFGHFGH114} \end{split} \end{align} Note that the sum in $k$ is bounded by \begin{align} \bar{\kappa} \bar{K} \sum_{k=k''+1}^\infty \frac{k! (i+k''-2)!}{k''! (k+i-2)!}  \bar{\kappa}^{k-k''-1} \bar{K}^{k-k''-1} \lesssim \bar{\kappa} , \label{SDHNRTYERTDGHFHDFGSDFGRTYFGHDFGDFSFGWERTFSDFGDFADWERTFGHDFGHFGH112} \end{align} uniformly for all $k'', i\in \mathbb{N}$, by taking $\bar{\kappa} \leq 1/CK$. Thus from \eqref{SDHNRTYERTDGHFHDFGSDFGRTYFGHDFGDFSFGWERTFSDFGDFADWERTFGHDFGHFGH114}--\eqref{SDHNRTYERTDGHFHDFGSDFGRTYFGHDFGDFSFGWERTFSDFGDFADWERTFGHDFGHFGH112} we estimate \begin{align} I_{27} \lesssim \bar{\kappa}  U_2 + \bar{\kappa}  U_3 .          \label{SDHNRTYERTDGHFHDFGSDFGRTYFGHDFGDFSFGWERTFSDFGDFADWERTFGHDFGHFGH155} \end{align} The term $I_{28}$ can be estimated analogously as $I_{27}$ by using \eqref{SDHNRTYERTDGHFHDFGSDFGRTYFGHDFGDFSFGWERTFSDFGDFADWERTFGHDFGHFGH50}, and we arrive at \begin{align} I_{28} \lesssim \bar{\kappa} U_2 + \bar{\kappa} U_3 . \label{SDHNRTYERTDGHFHDFGSDFGRTYFGHDFGDFSFGWERTFSDFGDFADWERTFGHDFGHFGH156} \end{align} Collecting the estimates \eqref{SDHNRTYERTDGHFHDFGSDFGRTYFGHDFGDFSFGWERTFSDFGDFADWERTFGHDFGHFGH151}, \eqref{SDHNRTYERTDGHFHDFGSDFGRTYFGHDFGDFSFGWERTFSDFGDFADWERTFGHDFGHFGH152}--\eqref{SDHNRTYERTDGHFHDFGSDFGRTYFGHDFGDFSFGWERTFSDFGDFADWERTFGHDFGHFGH153}, and \eqref{SDHNRTYERTDGHFHDFGSDFGRTYFGHDFGDFSFGWERTFSDFGDFADWERTFGHDFGHFGH155}--\eqref{SDHNRTYERTDGHFHDFGSDFGRTYFGHDFGDFSFGWERTFSDFGDFADWERTFGHDFGHFGH156}, we conclude  \begin{align} U_2 \lesssim \bar{\kappa} U_2 + \bar{\kappa} U_3 + U_4 + \sum_{E} D_{j,k,i} + \sum_{E} C_{j,k,i} , \llabel{xTN0HyXIgrXl7sQdnBWrzolZxCCwSlP6b52rqIT8yOeZKUW8iASz8WKiu3bjhg1zCxce2fpagwSyoiD3YxMRrbrprB8oFpJMtItcgQCJD8qQPX3tFyOyK9kqYsp1He43FFjfTYAoH9TrsZulJLfz5PN9MwOwlhKXLn4BWKHpc5r7rV3Ivf2qAk0ckyWiPt2Oo9BTfM5cXjaltIM7PPtTpRX1OysUMBp8XCuzgAd0NKAvdil19GwZVPrGLoGONxeU03Sm0BA1cxyfKGITFLNfgWPMpfRmO63LFMDUmz7qGhj190n34Oc9gieJVzMNOgAhfGTqBpw53q4eyb5J8ZeUNT8R38gKs708pFfj1ELgKiz0EACdAlPNN4okghcRHpmUuoeX4JJ2jtEXNXPRfb9FbN29ZtYUeX5EK9Lv22fvW5Q0VI0sq1IR1a5aojpZWyEKcOLaRUd2wxsHXws4XSrIVgWKywytyA00FoMN8est52u9fFKESCiP7WBvbjOhWuZIU5qDYQZOYVc1m2uyyWrATV9JIVCcILzr0XfXvLRjHkQ7dvQ5qQkKbUddVj2NgD9SbOe2NDztnIjaLQEEnLY4K0Do66oKZnE5tg9p8WgFJrymnIR4JuqanKB0Tu9gTTdhmFWJr7SkcmRgSVZBHXMKPdve2TKbYwe1w6LjLt3fqNqPNrkL9xZ312OlYPJzYHIR9FOwtOda3azKYIZmfOpiDP2UY1xdXMxnx0FHt0Z7KK5meIiZqgEIAVqkXuCIau0ge7syt7wPL1Somn8trVNAP5vOeApza4OmSDHNRTYERTDGHFHDFGSDFGRTYFGHDFGDFSFGWERTFSDFGDFADWERTFGHDFGHFGH115} \end{align} from where we arrive at \begin{align} U_2 \lesssim \bar{\kappa} U_3 + U_4 + \sum_{E} D_{j,k,i} + \sum_{E} C_{j,k,i} , \label{SDHNRTYERTDGHFHDFGSDFGRTYFGHDFGDFSFGWERTFSDFGDFADWERTFGHDFGHFGH166} \end{align} by taking $\bar{\kappa} \leq 1/C$. \par \emph{The terms $U_3$ and $U_4$}: For $U_3$, we have by \eqref{SDHNRTYERTDGHFHDFGSDFGRTYFGHDFGDFSFGWERTFSDFGDFADWERTFGHDFGHFGH72}, \begin{align} \begin{split}   U_3 \lesssim \sum_{E} D_{j,k,i}  + \sum_{E} C_{j,k,i} + U_5 . \label{SDHNRTYERTDGHFHDFGSDFGRTYFGHDFGDFSFGWERTFSDFGDFADWERTFGHDFGHFGH171} \end{split} \end{align} For $U_4$, we have by \eqref{SDHNRTYERTDGHFHDFGSDFGRTYFGHDFGDFSFGWERTFSDFGDFADWERTFGHDFGHFGH08}, \begin{align} \begin{split} U_4 \lesssim \bar{\kappa} \sum_{k=1}^\infty \sum_{ i=0}^\infty   \frac{ \bar{\kappa}^{k-1} \tau^{(k+i-3)_+}}{(k+i-3)!} \Vert \eprtudfsigjsdkfgjsdfg_x \TT^{k-1} (\epsilon \eprtudfsigjsdkfgjsdfg_t)^i v \Vert_{L^2} = \bar{\kappa} U_2 + \bar{\kappa} U_3 . \llabel{cKk2KSIDKbjSH9Rt0UKyyFShxSWRzxSsuMN2vFJgdTwYNCC7fYPYG52l9DuZ30jM3nxAAjYnsbz3ZREXuoEuVh6vUiKQXl9xplK4pv4r1SOF3gu9JSquV3aSxxP4r2sqmh8Vnc5lAnLI692OJ3I9hfl4tXzaJe6GPjnSEJW9IKlRk7tDlXhKdgKqL5vJLR7DGF4USAIttBcO5DLELRkZR6VJCCuPGSsCMuMUqIge2P0YmcB39mmPK7m3iNo3LPlKT6DxnGuwVqMZrP9zdW0MIXlRsCvy2x5hWn7UxgtJMGkQSm5W7JEQfvmB2Gi6cXhyQt1WjTdz4Samd0QLh4lfRnGntTOT6YaQsAyNoew52Km6kidUykKV0tntPTpBygKcOhcx77EkXSLocdpouOwaCQ3fNM4vtwIJvNz7dwFNviTOv3Medg8iLhaM3kdcjDT30paBswxFVIaXyjw0jOSbROdY3e3Phl4qm1wtaM08QrmGBISKqyoygoFISt0tQp0pZ1QQJknImQg436gpzdYuwl5msooOh9l4Z25dXGcSBaavBCelg09qWH49vDRFPE6SRgMSrOEiFDqSmv7MqC8n0B4OujjipuVia053hgXhQqEIzH7cbOABdHwOdTEXvoykC4wrajhzqnquDojXl0Pz7UtybiOOsEINnWEXehc70MvZ7n6V9qYzryOYuiqBEmXaktOJI8ZAdNdvleWiAh0gwDT6fWH6x8uSAjGkKuV6bwiUfcfOWjyzWdRn7KpFuCAHIogHdU3WkbDq1yV8oMfJskTSJAFV4ehtSDHNRTYERTDGHFHDFGSDFGRTYFGHDFGDFSFGWERTFSDFGDFADWERTFGHDFGHFGH113} \end{split} \end{align} \par \emph{The term $U_5$}: We claim that there exists $t_0>0$ sufficiently small depending only on $M_0$ and $\epsilon_1 >0$ sufficiently small depending on $M_{\epsilon, \kappa, \bar{\kappa}}(T)$, such that \begin{align} \sum_{ i=0}^\infty   \frac{ \tau^{(i-3)_+}}{(i-3)!} \Vert (\epsilon \eprtudfsigjsdkfgjsdfg_t)^i (p,v) \Vert_{L^2}	 \les 1+t Q(M_{\epsilon, \kappa, \bar{\kappa}}(t) ) , \label{SDHNRTYERTDGHFHDFGSDFGRTYFGHDFGDFSFGWERTFSDFGDFADWERTFGHDFGHFGH555} \end{align} for all $\epsilon \in (0, \epsilon_1)$. The proof of \eqref{SDHNRTYERTDGHFHDFGSDFGRTYFGHDFGDFSFGWERTFSDFGDFADWERTFGHDFGHFGH555} relies on the energy estimate of the partially linearized equation as in the case $\Omega = \mathbb{R}^3$, cf.~\cite[Lemma~6.2]{JKL}.  Here we only outline the modification needed for the presence of boundary. Instead of using the elliptic regularity, we appeal to Lemma~\ref{L13} to estimate the dissipative term. Namely, we have \begin{align} \begin{split} \Vert (\epsilon \eprtudfsigjsdkfgjsdfg_t)^i \nabla v\Vert_{L^2} & \les \Vert (\epsilon \eprtudfsigjsdkfgjsdfg_t)^i \dive v \Vert_{L^2} + \Vert (\epsilon \eprtudfsigjsdkfgjsdfg_t)^i \curl v \Vert_{L^2} + \Vert (\epsilon \eprtudfsigjsdkfgjsdfg_t)^i v \cdot \nu \Vert_{H^{1/2} (\eprtudfsigjsdkfgjsdfg \Omega)} + \Vert (\epsilon \eprtudfsigjsdkfgjsdfg_t)^i v \Vert_{L^2} \\ & \les \Vert (\epsilon \eprtudfsigjsdkfgjsdfg_t)^i \dive v \Vert_{L^2} + \Vert (\epsilon \eprtudfsigjsdkfgjsdfg_t)^i \curl v \Vert_{L^2} + \Vert (\epsilon \eprtudfsigjsdkfgjsdfg_t)^i v \Vert_{L^2} , \end{split}    \llabel{kSEpCWQaSgnZkk3aeesYU7TC7DjdycEy6YhXSt7PPC75VyXVX0033ywEoy37OVWeFB6f5A4zI5U5FYFTvRFmZbF16anRnJcMwFJuXCgDVU7FkanwUZqfhivmvYiUAZaiJLeDmZbMKcvxRNJpwxkgIOqH6pTR821OCO97sijRlLagXkCrxnH5nGq2b8TD1QLjwMUG7wLB8hV0tcfNtxrtAKTbTV70Cus2Z9zWbi3fBhUBueLctDXuk88jYpxJzuARu5GQSMoTAOBuuByUhESR0IVrVGFcH0rtSLRCVouoU5QbOUKP8WaFjZ0RLHqyX8xXBLOwUtkiEwWG9fmn7Ifh4MwTpuLWjOVCmpvGnoyHPV3Q2xzY5iDGuMIyY2zHhLTMRUHz9PrIp7eqfjbQMhsqm7f21zapbnCNGWxjGJAHfCykEOXMfRkQMSOnRmHDslDBlTsX1OYIh6VnhtORVgXFclHw0MsU8taIoG1w40uBD1sOoHWK3IngjGMdIjSUJarrbi4PVOmzDgPpUKxjKv1BOITyQEs2wltfw1QmqU5NgZs044dVr2FikdG6eqDEgvajQFWQPmQeNscdxSAyhqL2LfTUWHJoEsYAre3LvakeYHXXZEbdJodq2jjufZFwkFwMLCiBAzKslRU0wObDZ08O1IxMMACdT6Cuj3WOgHbMfWXYKctRoIofd0EECQVbFHq3GVsjjJoPiGcrg9Oa2u5s6j2W5AphHhdSrLt2Ctx8DHzuZXSEjHgy397e86ALC5AbRC7FQ4wELlpDvHtoumXBqq92Lgd2nSRcSDHNRTYERTDGHFHDFGSDFGRTYFGHDFGDFSFGWERTFSDFGDFADWERTFGHDFGHFGH69} \end{align} since $v \cdot \nu =0$ on $\eprtudfsigjsdkfgjsdfg \Omega$. Then we may proceed as in \cite[Lemma~6.2]{JKL} since the  term $\Vert (\epsilon \eprtudfsigjsdkfgjsdfg_t)^i v \Vert_{L^2}$ can be absorbed in the $A$~norm. It then follows from \eqref{SDHNRTYERTDGHFHDFGSDFGRTYFGHDFGDFSFGWERTFSDFGDFADWERTFGHDFGHFGH555} that \begin{align} 	U_5 	\les 	1 	+ 	t Q(M_{\epsilon, \kappa, \bar{\kappa}}(t)) 	\comma t\in(0,t_0) 	, 	\label{SDHNRTYERTDGHFHDFGSDFGRTYFGHDFGDFSFGWERTFSDFGDFADWERTFGHDFGHFGH176} \end{align} \par \subsection{Nonhomogeneous transport equation} \par As shown in \cite{JKL} and \cite{A05}, the curl component of the velocity satisfies the non-homogeneous transport equation. Thus, we consider \begin{align} \eprtudfsigjsdkfgjsdfg_t \tilde{S} + v\cdot \nabla \tilde{S}  = G,    \llabel{RG0a4DbmyLRpMSH3BXeD4ty54XXqEMbCkHuuytF38almrJcZeHlP8cgwNXGTZXeLavohyQaSEh1B7TEz54dRZGzhPJZtYAFK4PoB5vc5JDphrxfvSFcW2hSRzpnyvlqRWwhzhZ7fcd766Gh9DEWg5OcwF3FcGMUe42nZF1yw047XIG22Jd1INGkQLrXcvIW54aU33tnhD9NqKOLtMKkt5UCl2CmLLGkDZs8pORclJ3BqG5hKneVXEUFbnERg2k1yzfokQoCsnOZVSZk0X7sUjP5RruAQiLb6j2y3XRi9uzWFZZQnfC3ZisnU24uqcESJ9ATwmRCmbjfWBbjvQbZ1GmaYOiFK2s7Z20eOVLynbcwRtelbgInPM0f6Bd17EoyU9WOzSvNJh6ruX82bA5FvGC0sjhJsR110KLyNnrwYeVxgGE2h6LZS6AuVvmZ22sBOAVLsDyC1Hx9AqVK0xjKEnBwZAEnjXaHBePhlGArtTYDqq5AI3n843Vc1ess4dN0ehvkiAGcytH6lpoRT3A2eYLTK0WLJOpp5mbSzoaMNxjCJvzeTs9fXWiIocB0eJ8ZLK7UuRshiZWs4gQ6t6tZSg2PblDsYoLBhVb2CfR4XpTOROqtoZGj6i5W89NZGbFuP2gH5BmGsIbWwrDvr9oomzXTFB1vDSZrPg95mCQf9OyHYQf47uQYJ5xoMCpH9HdH5W5LsA0SHBcWAWGo3BiXzdcJxBISVMLxHKySoLJcJshYQyybWSSGdwcdEZ6bu0qfmmj7WMsVRYqJD7WciZqArimTf6csurioASDHNRTYERTDGHFHDFGSDFGRTYFGHDFGDFSFGWERTFSDFGDFADWERTFGHDFGHFGH70} \end{align} where $\tilde{S}= \tilde{S}(t,x)$, $v= v(t,x)$ and $G=G(t,x)$.  Denote \begin{align} \Vert u \Vert_{B(\tau)} = \sum_{(j,k,i)\in \mathbb{N}_0^3}  \frac{\kappa^{j} \bar{\kappa}^k \tau^{(j+k+i-2)_+}}{(j+k+i-2)!} \Vert \eprtudfsigjsdkfgjsdfg_{x}^j \TT^k (\epsilon \eprtudfsigjsdkfgjsdfg_t)^i u \Vert_{L^2}    ,    \llabel{8Ow5xADChPG7slxY1gyegeB9OC39KFnGn7aB1yr4vo4hHwIC93Lq9KQvIgFnBoS04NZd5ay1rlErPpbvxpseYDufgnOiibGhYOCBwwv4pxm37IgUW9XY4zxXMVQDNmYwkHaG258v6MBSerLf0Fya6yCLCzGm3P8COVDg7Rpp3FaWdG0kIQYRIs0ieIu2tgMmiYoJOmYt3fBSHj29Sq8ab4N94WEhm8BC79jHL9KmhK9Zfr6AlpyfIhpG93BnOMt16adBKKu3oNcfIhWjMfkeGdnzssVHKcphPZlBgcewSd0u1PePYEZowysGuGBvc6ea5iWB3XgSzyHnOaDp8VRDVqmVy3099pj5gFkDpqzPQ96m5MyJSQ7bj2LyLRn1PsvRqAaN5QOIcJKVPgGovU3Q69LN4W4v8wbS0jsi6krY7uOj9j1aBSZ6uonf5hsIj6BQ8Nnb2fcKcWDzL7nMsD7NxvuY95USbCwaHEBadGIqHgoriu5esF7AbYAYtJodoqq8VAyc8JApV0lJlYwyn6Sc6qTU7O1ek3r6y8zE0NvqFxOe0etCQfFQYU6wydAK5xsgtWyoHf4wkLJE7vsDtaSeyQYnecGAEE9Yiehg1bbi18Qic5rJyA6E1GNznKFeO8F23Ren7C5mlsRUt5moBQE0P0vMb5sNRJiRENqaZl1S0XQCLAWKOfnhsEcrp3BDnd8xRXBqgedxRV9ds7hxDYsXfuUOYUpPd6nkpnqU1BZylozBiHlOZqUwQHw6uyx1CevAKJpX8j3FDHKdkFGe1PLPwa3zuUlrnPllSDHNRTYERTDGHFHDFGSDFGRTYFGHDFGDFSFGWERTFSDFGDFADWERTFGHDFGHFGH73} \end{align} with the corresponding dissipative analytic norm \begin{align} \Vert u \Vert_{\tilde{B}(\tau)} = \sum_{j+k+i \geq 3}  \frac{\kappa^{j} \bar{\kappa}^k (j+k+i-2) \tau^{j+k+i-3}}{(j+k+i-2)!} \Vert \eprtudfsigjsdkfgjsdfg_{x}^j \TT^k (\epsilon \eprtudfsigjsdkfgjsdfg_t)^i u \Vert_{L^2} .    \llabel{ZsvR42FozOOJ2xv7SOMypK6I9dxqWRHMdk90PZn1CJiHjFiQOpyPnFWcnPNLzFCUTZbYKKM8QyLgFXtFqKzGFo8PnWkYxYg2TsGQZfmqqDqizHd93oVTk6QGDUljuHwS88DFPcV0PfY1PqwpsQSbMrVVNeZDe1F6mt7PgelfBt5ita8qXMizwKxCNoigAqPysitAbI7g4R8XE7tAKNexkg5a2urRig9qMqHSkERbIp9hrW5jzTcJlTgpMeLWFzzSVNgDtecg2YfSfNjVFP155DAWMbntnVm17177lYTObMWmVtwfPqJ0nF6baQxuSP13ZJubfmgRHsWKVnCRDcMkLRVXd3PEvZme7xKUmMURkbif3lAOuTemZqkO5fId0htO7xUpLztA9RO7gysAfqj751DjTsfImVIPQne9Fo9aRex4T4oypenSzEDID7KWchKsIrTolagXLh4f0vi4yl25MHPpbk9zYIfpFVV19CRh1tkqoQ1bQGtzGL0f4xN6GCHP7RXOnRbpX2WlP56rNQlEyARF84JeprN20CZxiyjplnNJiE3CfUxUjGxVfClOs8VwvDylP654xADcvWYtPPlqyDzXN83tJgydFQhHueuqU6ZyVijYQjtdJf4J93VVpNwzagiEmsHTInsFJaiogOY7GnLhHqUyLxrAq8itXPsMtyXO32mVYU2GMV729mGz8vXx4rNFR84VgcS6Bmn15nKdREfkhXGuf2c9CrXBqsEVOhPCASyf9Js77D9kqMznbcTe2Uu2Jnq8DBs8rfU7rd0BqQ4zzXOxTg7CSDHNRTYERTDGHFHDFGSDFGRTYFGHDFGDFSFGWERTFSDFGDFADWERTFGHDFGHFGH74} \end{align} \par \cole \begin{Lemma} \label{L20} For any $\kappa, \bar{\kappa} \in (0,1]$, there exists $\tau_1 \in (0,1]$ such that if $0 < \tau(0) \leq \tau_1$, then \begin{align} \Vert \tilde{S}(t)\Vert_{A(\tau)}  \les \Vert \tilde{S}(0)\Vert_{A(\tau)} + \int_0^t  \left( \Vert G(s) \Vert_{A(\tau)}  + \Vert v(s) \Vert_{A(\tau)} \right) + t ,    \llabel{rZuthjYW4aLtrFcCjwGZs7MWMmSpHFHfLiB7kX34MWAu5LIj3gvVGAkt1HUEE8UBYWVMJzMsv5P5JZvQnVHI1ZMg6AfsQjxl2VB9XKgEZrlgD5GqiFivOJgx0uYgE71IUqgjKnLSelC6w6tFI8fKMiHsqHDSeNXK9vtSmVcR8qXEEOqdh6WiObbuBi1dL1TGyUlerX0GfVwJXZcO4PCie5n0gNVUs6zGEiK4Qu0eVOBbGIv31f7MmZWGUui0XsJeIKrc6e1Hdz1A1uAiEJ1cBUWYPsoTyxIE5BV9Eyujob6iZVLOiNN6OdNQmB1lIjdwiBRH8d2G2I62dVz4IwScfAfPtOMSH1H1XALjFIXixvFrN9eesLSkD92y7poOY5zfXT21phWyN6C3iGT3jHRxnLx4ShHFPSeag0pIBSVsPSRRgqoHFPB9xiM4Lzsaaon5qdO9eSIckw1TWpA0dPRxm0mJS4CO20oacZNRhyXd9vF8myOtVO1yy8IPXRi69P7nEiVT9tqwzd7Si8yQLJx6HizMwkKqPKVouNGwhrsqKiGALuz34Qk0yTDFsJvXIYjMYHS9XWlgEwN6gCwyVdnIxK4B05Xchlwsf7kVsFFU6TUiBVUraV7WrTLbySahdqdcZgGtKtv0Uqp86jB630pHaqGwDnGpg0ZprB4fFDmtx2J23Ed37CNW5LkvQFWEouXhTIqdUmrBx7BGPMED17hd24zjCEavON1DJoxfmZadkTVq0VaQaDqE0vOGEazPBEKbULHqxEUgdPZyHW7xH2QWZqjxAaP0GszJSDHNRTYERTDGHFHDFGSDFGRTYFGHDFGDFSFGWERTFSDFGDFADWERTFGHDFGHFGH75} \end{align}  for some constant $C>0$ and sufficiently small $T_0>0$, provided $K$ in \eqref{SDHNRTYERTDGHFHDFGSDFGRTYFGHDFGDFSFGWERTFSDFGDFADWERTFGHDFGHFGH99} satisfies \begin{align} K \geq C \Vert v \Vert_{A(\tau)} \comma t\in [0,T_0] , \label{SDHNRTYERTDGHFHDFGSDFGRTYFGHDFGDFSFGWERTFSDFGDFADWERTFGHDFGHFGH223} \end{align} where $C\geq1$ is a sufficiently large constant, and $T_0$ is chosen so that \eqref{SDHNRTYERTDGHFHDFGSDFGRTYFGHDFGDFSFGWERTFSDFGDFADWERTFGHDFGHFGH148} holds. Similarly, for any $\kappa, \bar{\kappa} \in (0,1]$, there exists $\tau_1 \in (0,1]$ such that if $0 < \tau(0) \leq \tau_1$, then \begin{align} \Vert \tilde{S}(t)\Vert_{B(\tau)}  \les \Vert \tilde{S}(0)\Vert_{B(\tau)} + \int_0^t  \left( \Vert G(s) \Vert_{B(\tau)}  + \Vert v(s) \Vert_{B(\tau)} \right) + t ,    \llabel{TMSWwxJ2sZwiYA2Hu5wNNpJNQr1x3edvP8QVEszZCf2893tzSk3sLWLvXG2Ro4lH9AvQb5Ty4O1FTOAYScPcPIhJBPFPuHDAtdSw9x8vu9hrWf4XOypGsNQvxf9kolHSNEkMEGaZMPg1AwXYaRSUot7Fmp3SE4ZMRFbqq9KCjfC9LHi1X5nEhUw8q2Y6HLATRc4TNlNWkstGQ2RiLjSdaizOb6P848n3HLWHyAptpfQyXEcd7x0zqVNmDKzl5NtEVG2EvbbivLHBtnas3T4ddSM0QJyx3fU3CuYD2A5yhtcshT2ktVFKXXUDHoEs9pq3V1LmrePCwx4R8R51mNsW3ViX6mzJoctVZW1os7k61bwJoz5nOEsiih2VHu4D0bfy5RNqGQbIH8lXV9pKsHGGjm4xCypKPPTAUwxsx7SLjPjQdzHCVuarB6XUve32Z5b1lfFSRyfI0SaOdaYTBXN9qmijDO3az1zo9yKuqnB2UShf0GnWhQEht1eHPAd4a9l5p3Yi81r25uZwwR8G3OIzhTOGYsITG5YwtjFM4r2gLKn87OPhoVn11UEIn32O8sGpRnbNnRIZe0JQ29vldLUrB00Fd3eabLWFPv66hXwjkEcRTo4b0deLWv7rsDc7jiveH75j8F32Vz9lkcNphuH21hOE9LV28bc6f8a6szNZXRpJxDwCkqBrr6ewILJZFhHe8QzMVDJtJewQZJUI551nJ9PeovNSZZalieWJsew3kjdkmKn3Q6GKkfsnKeFlMnycL7fD9IaZ2zgk6dlqd01Z1rwEcbdMRH5ISDHNRTYERTDGHFHDFGSDFGRTYFGHDFGDFSFGWERTFSDFGDFADWERTFGHDFGHFGH76} \end{align}  for some constant $C>0$ and sufficiently small $T_0>0$, provided $K$ in \eqref{SDHNRTYERTDGHFHDFGSDFGRTYFGHDFGDFSFGWERTFSDFGDFADWERTFGHDFGHFGH99} satisfies \begin{align} K \geq C \Vert v \Vert_{B(\tau)} \comma t\in [0,T_0] , \label{SDHNRTYERTDGHFHDFGSDFGRTYFGHDFGDFSFGWERTFSDFGDFADWERTFGHDFGHFGH224} \end{align} for $C\geq 1$ sufficiently large and  where $T_0$ is chosen so that \eqref{SDHNRTYERTDGHFHDFGSDFGRTYFGHDFGDFSFGWERTFSDFGDFADWERTFGHDFGHFGH148} holds.  \end{Lemma} \colb Note that $\Vert v \Vert_{A(\tau)} \geq \Vert v \Vert_{B(\tau)}$ for all $v$, and thus \eqref{SDHNRTYERTDGHFHDFGSDFGRTYFGHDFGDFSFGWERTFSDFGDFADWERTFGHDFGHFGH223} implies \eqref{SDHNRTYERTDGHFHDFGSDFGRTYFGHDFGDFSFGWERTFSDFGDFADWERTFGHDFGHFGH224}.  \par \subsection{Product rule} In this section, we introduce the product rule which is needed in the estimates of divergence, curl, and pure time derivative components of the velocity. We denote  \begin{align} \Vert u \Vert_{Y} = \sup_{(j,k,i)\in \mathbb{N}_0^3, 0\leq j+k+i \leq 4} \Vert \eprtudfsigjsdkfgjsdfg_x^j \TT^k (\epsilon \eprtudfsigjsdkfgjsdfg_t)^i u \Vert_{L^2} .    \llabel{49KqEFRtg0tvkExZCRtMCkHE5snrnPg6Ijvsob8KThiOslyOFYa4mbysgIga7pzCdiXnubntkexBbWpwpMo4k8yxesQnXcG8mEFNWaegqbClcZ5bzZyNjEVtxND8cqbDf7Y7CZmIa140Y9sQMljx4EzBM6U0PHNOBsii1AKPov8wBSPX2VPf8xV3D3TDIWXly4kOf59po0Skn07HCFxP9bjyZBmyHMTfaFdJGqzVFPGNy6poIBElRsWC485AFty6t60NFgbMhZajDei2xvrczUQK9iMPhtiC1JBiBOtS95oFWT7jV0Q25YSeG3jID3SXk9xJIJ1QgYMR2hjzlnmTNtfqC5QI5IwpeqxKG1t8wq4sthwH6nBDBUclkzZWz0B2CqsO4fDgXc3OuMGJonsZM2EKIUOKr73Yv7JL9CdyguQqtIUrrhkHnHJ5JigfRrnKyNp0T37DBJiFzVqyPnQSY5QBv3LYZ5LrCdTqazeczlmGVAcx7uExXYu8ik7BIX3DEpaIyldWvAtDxvl70fkOcQ7RreyXDJE1gj3cLSNQTmK6lqOE4SVGxgDJO9muKDvvMaHGeevgMM7lWB15Wq9lFoUxFcPSyjHSWX5kgPup5yvw28JWVZ9At0WjfJuLfhFQd6aE7VnPq7ow2C6hK2loMkgJn03Vz5EteIMoiHWS2Htw4f99HNj7318gaIvJxjakNjunHNNM8kpBy8bp9KEQclHhTNFgu2uEkb6lPKahMNekvnAo0mMoeLXY0spWIDgTgbamnkCTSUvKLj6dtN3xYQr7kycTga6PSDHNRTYERTDGHFHDFGSDFGRTYFGHDFGDFSFGWERTFSDFGDFADWERTFGHDFGHFGH118} \end{align} The following lemma provides a product rule for the analytic $B$-norm. \par \cole \begin{Lemma} \label{L06} Let $k\in \{2,3, \ldots \}$ and $\tau >0$. For $f_1, \ldots, f_k \in B(\tau)$, and any $\kappa$, $\bar{\kappa} \in (0,1]$, there exists $\tau_1 \in (0,1]$ and $T_0>0$ such that if $0< \tau(0) \leq \tau_1$, then \begin{align} \Vert \prod_{i=1}^k f_i \Vert_{B(\tau)} \leq C^k \sum_{i=1}^k \Bigl( \Vert f_i \Vert_{B(\tau)} \prod_{1\leq j \leq k; j\neq i}  \left( \Vert f_j \Vert_{B(\tau)} + \Vert f_j \Vert_{Y} \right) \Bigr) . \llabel{3EVDPR1kftVZdHMlaCiTwUNJAw1pfZrEtytozvY85OB0FtLYpfKqFkodCG2d5ezkg1rqSaY8EJdLoUYFDMLsaWIITSAphuqtx20vtbdYDk8ReLv274NHf3aKJU0coov2EcMcwomvuDxZkoHa4L8uhvoKCiMy7UzuiN29wfPfG8TFtjOdqlG4M0luRt0lxZM3YA5uYoGkNkBzhbHo4aOEE5rxCUp1ArJjBN1h5cWs1F6qQVapGNgBgx7M5AswfEMyMgPsoVexGPimA4yXPma5dTjJ2BN3F3SMYKxqSegaB769y5PVO8QyJTRBe2vVe0pkNSyOnhkXc9d0clSXlbsqp8tCvKLTo5alzigOORUlP3qMGvTJARYozYDJLvTUuqMOONYNP82Ps9OiDs17x75SajMoJGBdp885oixszl6oKVEaWAZWwF96XfzlOd9zCCgma4P67OY59PgbGtzVZ4KWPNxAjy0Fp94LWBmVqjqW3vXH4rCz4IuAL6FdNeNtlHCh4NRn6UxnBYSpywkHC2I0Gictzw30lLmQmNTZ8HjuybH1W0zwXr9JEACx0pSDpGuO60FM2QmewNRFhP9fiIm1FeGcDRw2ovhfk2fSx8NUAIwPv4CmFCE8lJz5IeHrsPwiuDaT69r6Mj4keBfBhvgBXOKJdTbjUFNA59p2fDj1WpyNbiYAL9OywosNCiw4tSyNk8avAm4NgApm0XgUy3yAhP3TWptUe3EjdLgM1O9lt32XbpFzK4Avi5WwFdPyHAYdZrfFhrdkYrdfnJAepd29E66Cjsi1JRZ6SDHNRTYERTDGHFHDFGSDFGRTYFGHDFGDFSFGWERTFSDFGDFADWERTFGHDFGHFGH77} \end{align} \end{Lemma} \colb \par The next statement provides an analytic estimate for composition of functions. \par \cole \begin{Lemma} \label{L21} Assume that $f$ is an entire real-analytic function. Then  for $\tau>0$ and $w\in A(\tau)$, we have \begin{align} \Vert f(w) \Vert_{A(\tau)} \les Q(\Vert w\Vert_{A(\tau)}  + \Vert w \Vert_{Y} )    \llabel{lJUOAgxAp9MosewoBEV4n6aH75qqkvnXybkmeyYj4TaXPWupqUzKbkpqyo78o4zwXErs8jZMpW6ij9AP4qU6Bofbe06qWG1LY6kbllcrPmrXjpULG6mkqqGOqtbvMOO3tXZlHfb2xBrXkKXCfbTjxRvDbKlDrZ1aDF352Sg5Aav3lzOXMjar4a7qGoYXu3m60XPlQBE1PUxOiviBcglptUEyz5EDs7eDQpxJvDhG8DJzvScH9gTnXUuUftbxSmTtdb33oeQ6aiHUncNe4HecijPLdrnbGzeArrypC5pWBqjIsjCp4x25HWZwo8j2MNbdX2CfvqMUtl0eT39LzGoZGUXcC1cxRKQmrXFI2A5XLOpxHgNP1vXnZ35YTT3VGgG0Y6loTkgX0LjtDCaBSFgBKQvpoDC9HobyECQkmxBrJyBuZ3W39tu4Dea8xsYqE4OxT15FIkmYw0KTTYcNjy8VvmxWdyr6uNTV13OrOhZwAeZavnDhcEmBSweWlJps7N31mM2jXy92SdSas7K9c3MBcWyUSHdIXC9V5oGtzaEY8nMBJ5DNZHwBDOyWamRPVYnoV2qlhDcqpB8tHmM0e4ja5ptjDRAUaAgTEbYwBr4rapjLtCHQORU6i4GrYkFzx91EBrqRGhmlSu3RKzh489F6L6b4bxcdxYhlbk26D4DPDjmcj3Gm6MqaXGeNphBOSVYr0mOQs30RQc5wEge7aeVYf4gvJBA2eWoh292cHSKXVNgRN8T0oSND87zYU63epzrMIO2auIWoZe6SOIILTpn4dMOk7XtfvenMSDHNRTYERTDGHFHDFGSDFGRTYFGHDFGDFSFGWERTFSDFGDFADWERTFGHDFGHFGH78} \end{align} and \begin{align} \Vert f(w) \Vert_{B(\tau)} \les Q(\Vert w\Vert_{B(\tau)}  + \Vert w \Vert_{Y} ) ,    \llabel{lUF0r5GvYIauSAak20YFPB9jOxRu8zz3Dq32eq6iUeKOXgJ0zCW18zX2n0RhKCtqtZJJx2v2Polu2Sxjglz9aeKoC1db0OnKaouVpCm4lvK41H5Gev2sruVbkFYDjm7MwfiT1blqc93lFjHNwUtfPPNvyGdZOkimHRHTNNNzkw71lIGsYTSWR3ulXvokSf1KE3N9QLhfH7OEAELTre0BXLhex4qypTkhHsZBVnrgnHatAiEJv6Mz3X5fP5iiBI7m9q5VzUODo9YgA9uFrhGKOYZH4NNngGYwPYwLPrtmj63dkYx01d8zwbroon3LtCqpSeeQHTdVyRdz2GIxyjxKPpLa43Qqkz9bJVgNuivcKNtiU4ujWunorFDIcKzhicQZNU6Pz2uuRdQuKOmeh3mgffnVsf6MJMjBMsQhPH3o32OSx2MDEZMH5HnpWJibL5OFkDOUwiRxT3v06s2BYw4IJjX1zNhb7vTWSRff0ksZRBBODc5bCv1kW9ImjwmTRJwtKdlnqPn05tygZWviIEgd8s1MtAoe0mEjEIz3EafdFHNRWD785FlERfwQnJKGnjmzfVU24NgtQrxrOBtTKwBpqF7DJJJAJCF81h2eRhZvuStrU9VrPvjDzscPPw3rRiPiJgkT7MfVPTQHq12JwvTwx4cNeEegDTNyQWJX10wYClRDLGm1xe7jcIASC2Jgx3cfRzO9Ig6MxcdNRwZL0nyw5ufePnsNP09sz5ZIWrMQ3uh0NiF92ae0vXnX05RreQzCA5jvDdHXnSHmkTwKQaSFKKKQC5JeHMTGSDHNRTYERTDGHFHDFGSDFGRTYFGHDFGDFSFGWERTFSDFGDFADWERTFGHDFGHFGH79} \end{align} for some function $Q$. \end{Lemma} \colb \par Recall that $E$ is assumed to be the product of two entire real-analytic functions. Suppose that $\tilde{e}$ is one of the components of the matrix $E$. Thus, \begin{align} \tilde{e}(S, \epsilon u) = f(S) g(\epsilon u) . \label{SDHNRTYERTDGHFHDFGSDFGRTYFGHDFGDFSFGWERTFSDFGDFADWERTFGHDFGHFGH225} \end{align} \par In the next statement, we provide an analytic estimates of~$\tilde{e}$. \par \cole \begin{Lemma} \label{L22} Let $M_0>0$ and assume that $\tilde{e}$ satisfies \eqref{SDHNRTYERTDGHFHDFGSDFGRTYFGHDFGDFSFGWERTFSDFGDFADWERTFGHDFGHFGH225}. Then  \begin{align} \Vert \eprtudfsigjsdkfgjsdfg_t \tilde{e} \Vert_{B(\tau)} \les Q(\Vert u \Vert_{A(\tau)} + \Vert u \Vert_{Y},  \Vert S \Vert_{A(\tau)} + \Vert S \Vert_{Y} ) \comma \epsilon\in(0,1] \llabel{vF4peEMRxWeVYXcqNHhqFBvBuEXPxCf80MNsv3RRJVhcMyUPRU89oE0z7TCgQFf4WPet8YMqfaL8BmyuoDrCSc5yZ5DJj9icvgYu0SwkO9exAq2B2SIi96GD0UY9M6YcTlOvAjeT1AVE9R01xsVSOKOLxXtRzLb4LO1iXmyBmOTsiN6yYd9By2JhhZs7aek4dsjUZgANw1V9Q3EQHil8oz1jQVr4AbJQe7CPnRsumj5y1JvKb3ZeVPOx39tp3IaPyxezp10fAv5r50op2rVlfsC8SzD4oXaOPZdzKv8bQrCMmkOr6bXwVexQZPKXAyF29DYhbtcMcnWNvlSqrTKy2EZ76xdl2I7t2UznuGacBurSENCW7zNZH7j5Yy6DqUUG300pPFdrveXdhomN4iuraJ8k2FbP40Zxt9u69P5wCLbgZ6kTGSR8Thvau6fXjuSS6ERpRdkXha4tn5qaDKBVosbHi4mzuWFPFWeNzpRqESErUDkhmdF26MZJU6AODd2aoIazsQyxNwMW0x74T0cSdSnCTxaSDymsItYD6PfKil2DpE7Oc8jd2hf9U345GP8r88CjSMa0rqr0aDPg0FclzFatVBeiWFnnul8x6MekggEHNtGwoLbsvlctwjdtVt5DR8YU1hMVGstvfAzJ5QTGI1sBM6CT8PJZ2i1oylPtOOxN1AOAEvni73holwUaOOVCnKsU2VP4FJwi7ZjBSlbfzcpPubtsJSfu18LQjcoy2nmBvKLSwT0uqTVUlJY1k9kIXtCl46vGilSEgIkdy9ku8Tyn5C1le0FgSDHNRTYERTDGHFHDFGSDFGRTYFGHDFGDFSFGWERTFSDFGDFADWERTFGHDFGHFGH80} \end{align} and \begin{align} \Vert \tilde{e} \Vert_{A(\tau)}  \les Q(\Vert u \Vert_{A(\tau)} + \Vert u \Vert_{Y},  \Vert S \Vert_{A(\tau)} + \Vert S \Vert_{Y}) \comma \epsilon \in(0,1] ,    \llabel{3khMNu7ASqRthgPQIv8qlADHUJj7eGzm6T58HckDkZB95DFKrUx95ZWuFaoxKkyHUckpKg0hj576zokEXuGZOKIqhFuZ36Rzlx8JVFMOpexwpQB1JHH18HU1zvhqzykpbXrhunclgW8x6lvAe7jZ6k5OzcXh0G4XDPAoFIC3c51hWt1T3u2Kg9lk0aMMKiCIKoXN1X7S6fpYbm8vr4GWb2341o5oXTXboc0HigOGeP7GrDjucYcu5dKxRuKqeLRj4Jvv2ATxoHOHYx9HluGjpNXLJgPCNnaObg0zv8TFhT4LF1VPP1gF4EDNXBTIAZTrkHu7jREL59J8x4Q1LCCgsRqjOy9531BZrk8LjenbjnZxzXbyhTjqSWBncnmW6RC3xjhxgclBiyrZBSNpeysqVY8dUkg1ECGXRwYhMkWlx0YMCMuLz4kPnCcBzWAGBgkS1AkqJystPTHYq0jnAuj2izfOt4c7T2XyQw8LDBH1S4OqMrTPQevzqImDi1lu7AqznMT06GdWG18YkmZMAprJ80gqCMSoSaSK10401RJ3aAaivMlE5hnLj3umYPL9ZDUl2K48oEP5OfcQpaGuZxiOuCd4HN2gzwz26bd821FDge0Wi9AUrIa9i1GJT9bmDdxfE68pNo8cGYvu7UYzYoN5CGJT8Oslp4rLBezUcDD4dBTFjOmzd4QrW6O9volY9siRkzcMNyRL58ptXTEMsMvkWFYGUxutZLG4SL2hetQ2O6qG0gfak95iq7fphqolryvnOzFQDPhEpL4dYI0sakvwh4ylHPKA7m3MSDHNRTYERTDGHFHDFGSDFGRTYFGHDFGDFSFGWERTFSDFGDFADWERTFGHDFGHFGH85} \end{align} where $Q$ is a function. \end{Lemma} The proofs of Lemmas~\ref{L20}, \ref{L06}, \ref{L21}, and \ref{L22} are analogous to Lemmas~4.2, 5.1, 5.2, and~5.3 in \cite{JKL}. Thus we simply state them and refer to \cite{JKL} for details. \colb \par \subsection{The divergence components} \label{subsec01} In this section, we estimate the sum of terms involving the divergence of the velocity. First, we rewrite the equation \eqref{SDHNRTYERTDGHFHDFGSDFGRTYFGHDFGDFSFGWERTFSDFGDFADWERTFGHDFGHFGH01} as \begin{align} L(\eprtudfsigjsdkfgjsdfg_x ) u = -E(S, \epsilon u) (\epsilon \eprtudfsigjsdkfgjsdfg_t u + \epsilon v \cdot \nabla u) . \label{SDHNRTYERTDGHFHDFGSDFGRTYFGHDFGDFSFGWERTFSDFGDFADWERTFGHDFGHFGH178} \end{align} For $j\geq 2$ and $k,i \in \mathbb{N}_0$, we commute $\eprtudfsigjsdkfgjsdfg_{x}^{j-1} \TT^k (\epsilon \eprtudfsigjsdkfgjsdfg_t)^i$ with \eqref{SDHNRTYERTDGHFHDFGSDFGRTYFGHDFGDFSFGWERTFSDFGDFADWERTFGHDFGHFGH178}, obtaining \begin{align} \begin{split} & \Vert \eprtudfsigjsdkfgjsdfg_{x}^{j-1} \TT^k (\epsilon \eprtudfsigjsdkfgjsdfg_t)^i L(\eprtudfsigjsdkfgjsdfg_x) u \Vert_{L^2} \\ &\indeq \les \sum_{j'=0}^{j-1} \sum_{k'=0}^k \sum_{i'=0}^i \binom{j-1}{j'} \binom{k}{k'} \binom{i}{i'}  \Vert \eprtudfsigjsdkfgjsdfg_{x}^{j'} \TT^{k'} (\epsilon \eprtudfsigjsdkfgjsdfg_t)^{i'} E \eprtudfsigjsdkfgjsdfg_{x}^{j-1-j'} \TT^{k-k'} (\epsilon \eprtudfsigjsdkfgjsdfg_t)^{i-i'+1} u \Vert_{L^2}  \\ &\indeq\indeq + \epsilon \Vert Ev \Vert_{L^\infty}  \Vert \eprtudfsigjsdkfgjsdfg_{x}^{j-1} \TT^k (\epsilon \eprtudfsigjsdkfgjsdfg_t)^i \nabla u \Vert_{L^2} \\ &\indeq\indeq + \epsilon  \sum_{j'=0}^{j-1} \sum_{k'=0}^k \sum_{i'=0}^i \binom{j-1}{j'} \binom{k}{k'} \binom{i}{i'}  \Vert \eprtudfsigjsdkfgjsdfg_{x}^{j'} \TT^{k'} (\epsilon \eprtudfsigjsdkfgjsdfg_t)^{i'} (Ev) \eprtudfsigjsdkfgjsdfg_{x}^{j-1-j'} \TT^{k-k'} (\epsilon \eprtudfsigjsdkfgjsdfg_t)^{i-i'} \nabla u \Vert_{L^2}  . \end{split}    \label{SDHNRTYERTDGHFHDFGSDFGRTYFGHDFGDFSFGWERTFSDFGDFADWERTFGHDFGHFGH86} \end{align} Multiplying the above estimate with appropriate weights and following the arguments from \cite[Lemma~6.3]{JKL}, which is justified since we have Lemmas~\ref{L06},~\ref{L21}, and~\ref{L22}, we obtain \begin{align} \sum_{j=2}^\infty \sum_{k=0}^\infty \sum_{i=0}^\infty D_{j,k,i} \les (\kappa+\epsilon) Q(M_{\epsilon, \kappa, \bar{\kappa}}(t)) . \label{SDHNRTYERTDGHFHDFGSDFGRTYFGHDFGDFSFGWERTFSDFGDFADWERTFGHDFGHFGH179} \end{align} For $j=1$, $k=0$, and $i\in\mathbb{N}_0$, we proceed as in \cite[Lemma~6.3]{JKL} to get \begin{align} \sum_{i=0}^\infty D_{1,0,i} \les 1 + (t+\epsilon) Q(M_{\epsilon, \kappa, \bar{\kappa}}(t)) . \label{SDHNRTYERTDGHFHDFGSDFGRTYFGHDFGDFSFGWERTFSDFGDFADWERTFGHDFGHFGH180} \end{align} For $j = 1$, $k\in \mathbb{N}$, and $i\in\mathbb{N}_0$, from \eqref{SDHNRTYERTDGHFHDFGSDFGRTYFGHDFGDFSFGWERTFSDFGDFADWERTFGHDFGHFGH08}, we have \begin{align} \Vert \TT^k (\epsilon \eprtudfsigjsdkfgjsdfg_t)^i L(\eprtudfsigjsdkfgjsdfg_{x}) u \Vert_{L^2} \les \Vert \eprtudfsigjsdkfgjsdfg_x \TT^{k-1} (\epsilon \eprtudfsigjsdkfgjsdfg_t)^i L(\eprtudfsigjsdkfgjsdfg_{x}) u \Vert_{L^2} ,    \llabel{YtUSmTx15Qt0X54AVmgdCsHmUrA7DAbpNXCtJ0P1eb0Q6Vo2Jaff0EOAQUf0510tfWogcc6p0dVuaWtW0uRsj2XoOJ2OWmCpXbaFvysWRSnkWo6mEyK3GHLNkxe3B5TF1qBnOpngLqytbymb68q36Z5f6EOm9IDEZkYycjnLpeLi1xpyIwueT2WP1xWO39QiM5YRvaH8I1gYrfa3ciagDowq9Tneks4wb92M8lXtz9TelMyPkGiwPOvlk4wSZZNVwvJCD1vhy3wL4tZ7jWEJfyJLXLlmEZjCLuhUhE8DznevGxQse4yPSmj8mZWQ39zhnVKLXerveM9vXS8NoyyIsflj0xzdc4tGEME5eMYGEERC5uOp5VbWSeXJsGlmo4OKC3kEdmZXpeXz9CdSMteccbWFEbpr9EZ9ggzFD5btTrFtoNmgfG2S9TEZvwdWLK7SS5VD91VYWQEkAEy3fDCZFxAr3hOXO7SumpBsZ1fDpn5wGzW9QgoFHsLsL0QGw5m0JS7L0T9IdJ6oFRHfr8V3nT3mBWBbw43JnCoAmgjgvYo39JNasqEblQ7LJfimsl7hE6su9iIWXlxCGQdPEv5laV3TGcQODdaI20rJTSrgUHC4oCjW3ybypA1TqznoO4tiE7aGkqt0W0MeafqkPM7cnkvhJY0BXzMbQTO71ZsW2zPjKOf80k1Sq95G2MTbvnLYEiGu0QNGTbEFJZm6dmEzN3b9wKiuwVk5YFNRTxRuiAiqTdoXn15E4VtxaotyYO9yfH7tT3VompVndiZP46icLlIxBJCnpYsFSDHNRTYERTDGHFHDFGSDFGRTYFGHDFGDFSFGWERTFSDFGDFADWERTFGHDFGHFGH119} \end{align} from where we proceed as in \eqref{SDHNRTYERTDGHFHDFGSDFGRTYFGHDFGDFSFGWERTFSDFGDFADWERTFGHDFGHFGH86}--\eqref{SDHNRTYERTDGHFHDFGSDFGRTYFGHDFGDFSFGWERTFSDFGDFADWERTFGHDFGHFGH179} to obtain \begin{align} 	\sum_{k=1}^\infty \sum_{i=0}^\infty 	D_{1,k,i} 	\les 	(\bar{\kappa}+\epsilon) Q(M_{\epsilon, \kappa, \bar{\kappa}}(t)) 	. 	\label{SDHNRTYERTDGHFHDFGSDFGRTYFGHDFGDFSFGWERTFSDFGDFADWERTFGHDFGHFGH479} \end{align} Combining \eqref{SDHNRTYERTDGHFHDFGSDFGRTYFGHDFGDFSFGWERTFSDFGDFADWERTFGHDFGHFGH179}--\eqref{SDHNRTYERTDGHFHDFGSDFGRTYFGHDFGDFSFGWERTFSDFGDFADWERTFGHDFGHFGH180} and \eqref{SDHNRTYERTDGHFHDFGSDFGRTYFGHDFGDFSFGWERTFSDFGDFADWERTFGHDFGHFGH479}, we arrive at \begin{align} \sum_{E} D_{j,k,i,} \les 1 + (\kappa + \bar{\kappa} + t+ \epsilon) Q(M_{\epsilon, \kappa, \bar{\kappa}}(t)) . \label{SDHNRTYERTDGHFHDFGSDFGRTYFGHDFGDFSFGWERTFSDFGDFADWERTFGHDFGHFGH181} \end{align} \par \subsection{The curl components} As in \cite[Section~6.1]{JKL}, we use Lemmas~\ref{L06}, \ref{L21}, and~\ref{L22} to obtain \begin{align} \sum_{E} C_{j,k,i} \les 1 + (t+\tau) Q(M_{\epsilon, \kappa, \bar{\kappa}}(t)) , \label{SDHNRTYERTDGHFHDFGSDFGRTYFGHDFGDFSFGWERTFSDFGDFADWERTFGHDFGHFGH182} \end{align} for all $t\in[0,T_0]$,  \par \subsection{The pressure estimates} The analytic norm of the pressure can be recovered by the mixed space-time derivatives and pure time derivatives.  Namely, for $j\in \mathbb{N}$ and $k\in \mathbb{N}_0$, we have \begin{align} \Vert \eprtudfsigjsdkfgjsdfg_{x}^j \TT^k (\epsilon \eprtudfsigjsdkfgjsdfg_t)^i p \Vert_{L^2} \les \Vert \eprtudfsigjsdkfgjsdfg_{x}^{j-1} \TT^k (\epsilon \eprtudfsigjsdkfgjsdfg_t)^i L(\eprtudfsigjsdkfgjsdfg_x) u \Vert_{L^2} + \Vert \eprtudfsigjsdkfgjsdfg_{x}^{j-1} [\TT^k, \nabla] (\epsilon \eprtudfsigjsdkfgjsdfg_t)^i p \Vert_{L^2} . \label{SDHNRTYERTDGHFHDFGSDFGRTYFGHDFGDFSFGWERTFSDFGDFADWERTFGHDFGHFGH510} \end{align} The first term on the right side is estimated in Section~\ref{subsec01}, while the second term is estimated analogously to \eqref{SDHNRTYERTDGHFHDFGSDFGRTYFGHDFGDFSFGWERTFSDFGDFADWERTFGHDFGHFGH505}. Thus, we arrive at \begin{align} \sum_{j=1}^\infty \sum_{k=0}^\infty \sum_{i=0}^\infty \frac{\kappa^{(j-1)_+} \bar{\kappa}^k \tau^{(j+k+i-3)_+}}{(j+k+i-3)!} \Vert \eprtudfsigjsdkfgjsdfg_{x}^j \TT^k (\epsilon \eprtudfsigjsdkfgjsdfg_t)^i p \Vert_{L^2} \les 1 + (\kappa + \bar{\kappa} + t+ \epsilon) Q(M_{\epsilon, \kappa, \bar{\kappa}}(t)) . \label{SDHNRTYERTDGHFHDFGSDFGRTYFGHDFGDFSFGWERTFSDFGDFADWERTFGHDFGHFGH511} \end{align} For $j=0$ and $k\in \mathbb{N}$, we may use \eqref{SDHNRTYERTDGHFHDFGSDFGRTYFGHDFGDFSFGWERTFSDFGDFADWERTFGHDFGHFGH08} and \eqref{SDHNRTYERTDGHFHDFGSDFGRTYFGHDFGDFSFGWERTFSDFGDFADWERTFGHDFGHFGH510}--\eqref{SDHNRTYERTDGHFHDFGSDFGRTYFGHDFGDFSFGWERTFSDFGDFADWERTFGHDFGHFGH511} to get \begin{align} \begin{split} \sum_{k=1}^\infty \sum_{i=0}^\infty \frac{ \bar{\kappa}^k \tau^{(k+i-3)_+}}{(k+i-3)!} \Vert \TT^k (\epsilon \eprtudfsigjsdkfgjsdfg_t)^i p \Vert_{L^2} & \les \sum_{k=1}^\infty \sum_{i=0}^\infty \frac{ \bar{\kappa}^k \tau^{(k+i-3)_+}}{(k+i-3)!} \Vert \eprtudfsigjsdkfgjsdfg_{x} \TT^{k-1} (\epsilon \eprtudfsigjsdkfgjsdfg_t)^i p \Vert_{L^2} \\ & \les 1 + (\kappa + \bar{\kappa} + t+ \epsilon) Q(M_{\epsilon, \kappa, \bar{\kappa}}(t)) . \label{SDHNRTYERTDGHFHDFGSDFGRTYFGHDFGDFSFGWERTFSDFGDFADWERTFGHDFGHFGH512} \end{split} \end{align} Combining \eqref{SDHNRTYERTDGHFHDFGSDFGRTYFGHDFGDFSFGWERTFSDFGDFADWERTFGHDFGHFGH511}--\eqref{SDHNRTYERTDGHFHDFGSDFGRTYFGHDFGDFSFGWERTFSDFGDFADWERTFGHDFGHFGH512} with the pure time derivative estimates of the pressure obtained in \eqref{SDHNRTYERTDGHFHDFGSDFGRTYFGHDFGDFSFGWERTFSDFGDFADWERTFGHDFGHFGH555}, we arrive at \begin{align} \Vert p \Vert_{A(\tau)} \les 1 + (\kappa + \bar{\kappa} + t+ \epsilon) Q(M_{\epsilon, \kappa, \bar{\kappa}}(t)) .	 \label{SDHNRTYERTDGHFHDFGSDFGRTYFGHDFGDFSFGWERTFSDFGDFADWERTFGHDFGHFGH513} \end{align} \par \subsection{Proof of the main lemma} Here we combine the results of the previous section to prove Lemma~\ref{L12}. \par \begin{proof}[Proof of Lemma~\ref{L12}] Combining the estimates \eqref{SDHNRTYERTDGHFHDFGSDFGRTYFGHDFGDFSFGWERTFSDFGDFADWERTFGHDFGHFGH183}, \eqref{SDHNRTYERTDGHFHDFGSDFGRTYFGHDFGDFSFGWERTFSDFGDFADWERTFGHDFGHFGH174}, \eqref{SDHNRTYERTDGHFHDFGSDFGRTYFGHDFGDFSFGWERTFSDFGDFADWERTFGHDFGHFGH166}, \eqref{SDHNRTYERTDGHFHDFGSDFGRTYFGHDFGDFSFGWERTFSDFGDFADWERTFGHDFGHFGH171}--\eqref{SDHNRTYERTDGHFHDFGSDFGRTYFGHDFGDFSFGWERTFSDFGDFADWERTFGHDFGHFGH176}, and \eqref{SDHNRTYERTDGHFHDFGSDFGRTYFGHDFGDFSFGWERTFSDFGDFADWERTFGHDFGHFGH181}--\eqref{SDHNRTYERTDGHFHDFGSDFGRTYFGHDFGDFSFGWERTFSDFGDFADWERTFGHDFGHFGH182}, we arrive at \begin{align} \begin{split} \Vert v \Vert_{A(\tau)} & \les 1 + \left( t  + \kappa + \epsilon + \tau(0) + \bar{\kappa} \right) Q(M_{\epsilon, \kappa, \bar{\kappa}}(t)) , \label{SDHNRTYERTDGHFHDFGSDFGRTYFGHDFGDFSFGWERTFSDFGDFADWERTFGHDFGHFGH226} \end{split} \end{align} by taking $\bar{\kappa} \leq 1/C$ and $\kappa \leq \bar{\kappa}/C$. Thus Lemma~\ref{L12} follows by combining \eqref{SDHNRTYERTDGHFHDFGSDFGRTYFGHDFGDFSFGWERTFSDFGDFADWERTFGHDFGHFGH513}--\eqref{SDHNRTYERTDGHFHDFGSDFGRTYFGHDFGDFSFGWERTFSDFGDFADWERTFGHDFGHFGH226} with \eqref{SDHNRTYERTDGHFHDFGSDFGRTYFGHDFGDFSFGWERTFSDFGDFADWERTFGHDFGHFGH227}. \end{proof} \par \startnewsection{Analyticity assumptions on the initial data}{sec07} In this section, we assume the initial data satisfies \eqref{SDHNRTYERTDGHFHDFGSDFGRTYFGHDFGDFSFGWERTFSDFGDFADWERTFGHDFGHFGH23}, and show that for low values of $i$, \begin{align} \sum_{i=0}^3 \sum_{j,k=0}^\infty  \Vert \eprtudfsigjsdkfgjsdfg_x^j \TT^k (\epsilon \eprtudfsigjsdkfgjsdfg_t)^i u (0) \Vert_{L^2} \frac{\tau_0^{(j+k+i-3)_+}}{(j+k+i-3)!} \leq \Gamma , \label{SDHNRTYERTDGHFHDFGSDFGRTYFGHDFGDFSFGWERTFSDFGDFADWERTFGHDFGHFGH92} \end{align} and \begin{align} \sum_{i=0}^3 \sum_{j,k=0}^\infty  \Vert \eprtudfsigjsdkfgjsdfg_x^j \TT^k (\epsilon \eprtudfsigjsdkfgjsdfg_t)^i S (0) \Vert_{L^2} \frac{\tau_0^{(j+k+i-3)_+}}{(j+k+i-3)!} \leq \Gamma , \label{SDHNRTYERTDGHFHDFGSDFGRTYFGHDFGDFSFGWERTFSDFGDFADWERTFGHDFGHFGH93} \end{align} where $\Gamma>0$ is a sufficiently large constant depending on $M_0$; for high values of $i$, there exists a sufficiently large constant $C>1$ such that for all $n\geq 4$ we have \begin{align} \sum_{i=4}^n \sum_{j,k=0}^\infty  \Vert \eprtudfsigjsdkfgjsdfg_x^j \TT^k (\epsilon \eprtudfsigjsdkfgjsdfg_t)^i u(0) \Vert_{L^2} \frac{ \tau_0^{j+k+i-3}}{C^{i-3}(j+k+i-3)!} \leq 1 , \label{SDHNRTYERTDGHFHDFGSDFGRTYFGHDFGDFSFGWERTFSDFGDFADWERTFGHDFGHFGH90} \end{align} and \begin{align} \sum_{i=4}^n \sum_{j,k=0}^\infty \Vert \eprtudfsigjsdkfgjsdfg_x^j \TT^k (\epsilon \eprtudfsigjsdkfgjsdfg_t)^i S(0) \Vert_{L^2} \frac{\tau_0^{j+k+i-3}}{C^{i-3} (j+k+i-3)!} \leq 1 . \label{SDHNRTYERTDGHFHDFGSDFGRTYFGHDFGDFSFGWERTFSDFGDFADWERTFGHDFGHFGH91} \end{align} In \eqref{SDHNRTYERTDGHFHDFGSDFGRTYFGHDFGDFSFGWERTFSDFGDFADWERTFGHDFGHFGH90} and \eqref{SDHNRTYERTDGHFHDFGSDFGRTYFGHDFGDFSFGWERTFSDFGDFADWERTFGHDFGHFGH91} we then choose $\tilde{\tau}_0 = \tau_0/C$ and using \eqref{SDHNRTYERTDGHFHDFGSDFGRTYFGHDFGDFSFGWERTFSDFGDFADWERTFGHDFGHFGH92}--\eqref{SDHNRTYERTDGHFHDFGSDFGRTYFGHDFGDFSFGWERTFSDFGDFADWERTFGHDFGHFGH93}, we obtain \begin{align} \begin{split} \Vert (p_0, v_0, S_0) \Vert_{A(\tilde{\tau}_0)} & \leq \sum_{(j,k,i) \in \mathbb{N}_0^3}  \Vert \eprtudfsigjsdkfgjsdfg_x^j \TT^k (\epsilon \eprtudfsigjsdkfgjsdfg_t)^i (u, S)(0) \Vert_{L^2} \frac{ \tilde{\tau}_0^{(j+k+i-3)_+}}{(j+k+i-3)!} \\ &\leq \sum_{i=0}^3 \sum_{j,k=0}^\infty \Vert \eprtudfsigjsdkfgjsdfg_x^j \TT^k (\epsilon \eprtudfsigjsdkfgjsdfg_t)^i S(0) \Vert_{L^2} \frac{\tau_0^{(j+k+i-3)_+}}{(j+k+i-3)!} \\ &\indeq + \sum_{i=4}^\infty \sum_{j,k=0}^\infty \Vert \eprtudfsigjsdkfgjsdfg_x^j \TT^k (\epsilon \eprtudfsigjsdkfgjsdfg_t)^i S(0) \Vert_{L^2} \frac{ \tau_0^{j+k+i-3}}{C^{i-3}(j+k+i-3)!} \\ & \leq \Gamma+1 , \end{split}    \llabel{UuMNQbvnG9hYiesb9N47f43mmBa9OTZtDHBTlfZWAgVRjoEt8R9hJt51LtxxkSrewmHE8O4a7EM7NLdqZbHtrVr0wQEzeNqhPsyb3Ve0OCynaamXp8chaECFTH3T5ikwA3bSVgy7uE9coQZT3qOIUDx42NqgDMFNTdMvwgD1fZFPRkvHKWJfvGMZiDn7SKsz4i8Y7q6rgpj1lNQ6bj3WRZSDtSy0XpePGkzv1DrV2ojU7FyHd97XY5LvKEu8TxFpQzrUzmNf9Lh1XhK5LIJYtAROxR4h0gsFuUocRIEHVuHnUmln7Xzc2SnjcQmAAmnlUoIXhkWG9BfRtd6GlFdZpHHvXGeyBrnGvjderN5OdLReh4E5blwZCdd01DSjopD4IGoyeLL5WvrNhbyGZg5N9NBMIi0t9GBWqaEXzIIXee67NKdmy5gFliPWOzcvxiNA7y4czlk4zeQKWY5YuhXpCYJX2T1lRzgRbTjmsmqUeYU1LqVcFfeZ5UxOoAKjOuNn2YvvPByHIqI23R7LrsbkILe0K05Ad2qBJQdgdUYw6NuUH3Qff5SKjAGxxlelehaFQwnzF4oUbe2vqhPcri2YpwZrWQ4bLPVBctPeMQhFQgZBY27KyJN3gbDAZfo2X5KKZZPksuknOOSk9M4w8d26KRzXWOnOxKgMphzUxC4PO35hQoD0RmsBYDSw151D2Ib1Zja6HDtyKlHqACT8H88oti35EFFdWlyiDKxFwymGdrdvnz24CRPVdcv8mowWaQumLO2lbYLIPZjmUe9UfklvaRck3KhqKhd5SDHNRTYERTDGHFHDFGSDFGRTYFGHDFGDFSFGWERTFSDFGDFADWERTFGHDFGHFGH88} \end{align} and we conclude as in \cite[Section~8]{JKL}. The proofs of \eqref{SDHNRTYERTDGHFHDFGSDFGRTYFGHDFGDFSFGWERTFSDFGDFADWERTFGHDFGHFGH92}--\eqref{SDHNRTYERTDGHFHDFGSDFGRTYFGHDFGDFSFGWERTFSDFGDFADWERTFGHDFGHFGH91} are analogous to those in \cite[Section~8]{JKL} by using the commutator estimate \eqref{SDHNRTYERTDGHFHDFGSDFGRTYFGHDFGDFSFGWERTFSDFGDFADWERTFGHDFGHFGH19}, and thus we omit the details. \par \startnewsection{Proof of the convergence theorem}{sec08} The following lemma is needed in the proof of the Theorem~\ref{T02}. \par \cole \begin{Lemma} \label{L23} There exists a constant $C_1>1$ such that for any $\alpha \in \mathbb{N}_0^3$ with $|\alpha |= j$ where $j\in \mathbb{N}_0$, we have \begin{align} \Vert \eprtudfsigjsdkfgjsdfg_x^j u \Vert_{L^2} \leq C_1^j \Vert u \Vert_{L^2}^{1/(j+1)} \Vert \eprtudfsigjsdkfgjsdfg_x^{j+1} u \Vert_{L^2}^{j/(j+1)} , \label{SDHNRTYERTDGHFHDFGSDFGRTYFGHDFGDFSFGWERTFSDFGDFADWERTFGHDFGHFGH228} \end{align} for all $u \in H^{j+1}(\Omega)$. \end{Lemma} \colb \par We emphasize that $C_1$ is $j$-independent. \par \begin{proof}[Proof of Lemma~\ref{L23}] We proceed by induction on $j\in \mathbb{N}_0$. The case $j=0$ is clear. Now we assume that \eqref{SDHNRTYERTDGHFHDFGSDFGRTYFGHDFGDFSFGWERTFSDFGDFADWERTFGHDFGHFGH228} holds for some $j\in \mathbb{N}_0$ and aim to prove it for the case $j+1$.  Let $\alpha\in \mathbb{N}_0^3$ be any multiindex with $|\alpha| = j+1$.  We choose a multiindex $\alpha_{-} \in \mathbb{N}_0^3$ such that there  exists $\xi\in \mathbb{N}_0^3$ for which $\alpha_{-} + \xi = \alpha$ and $|\xi| =1$. Let $\alpha_{+} = 2\alpha - \alpha_{-}$. Using integration by parts, we obtain \begin{align} \Vert \eprtudfsigjsdkfgjsdfg^\alpha u\Vert_{L^2}^2 = \int_{\eprtudfsigjsdkfgjsdfg \Omega} \eprtudfsigjsdkfgjsdfg^{\alpha_{-}}u \cdot \eprtudfsigjsdkfgjsdfg^{\alpha} u \nu^i \,d\sigma - \int_{\Omega} \eprtudfsigjsdkfgjsdfg^{\alpha_{-}}u \cdot \eprtudfsigjsdkfgjsdfg^{\alpha_{+}} u \,dx = I_1 + I_2 , \label{SDHNRTYERTDGHFHDFGSDFGRTYFGHDFGDFSFGWERTFSDFGDFADWERTFGHDFGHFGH231} \end{align} for some $i \in \{1,2,3\}$. For the term $I_1$, we use the Cauchy-Schwarz inequality to get \begin{align} \begin{split} I_1 & \leq \left( \int_{\eprtudfsigjsdkfgjsdfg \Omega} |\eprtudfsigjsdkfgjsdfg^{\alpha_{-}}u|^2 \,d\sigma \right)^{1/2} \left( \int_{\eprtudfsigjsdkfgjsdfg \Omega} |\eprtudfsigjsdkfgjsdfg^{\alpha}u|^2 \,d\sigma \right)^{1/2} \\ & \leq C \Vert \eprtudfsigjsdkfgjsdfg^{\alpha_{-}} u \Vert_{L^2}^{1/2} \Vert \nabla \eprtudfsigjsdkfgjsdfg^{\alpha_{-}} u \Vert_{L^2}^{1/2} \Vert \eprtudfsigjsdkfgjsdfg^{\alpha} u \Vert_{L^2}^{1/2} \Vert \nabla \eprtudfsigjsdkfgjsdfg^{\alpha} u \Vert_{L^2}^{1/2} , \label{SDHNRTYERTDGHFHDFGSDFGRTYFGHDFGDFSFGWERTFSDFGDFADWERTFGHDFGHFGH229} \end{split} \end{align} where the last inequality follows from the trace theorem. For the term $I_2$, using the Cauchy-Schwarz inequality,  we arrive at \begin{align} I_2 \leq \Vert \eprtudfsigjsdkfgjsdfg^{\alpha_{-}} u \Vert_{L^2} \Vert \eprtudfsigjsdkfgjsdfg^{\alpha_{+}} u \Vert_{L^2} . \label{SDHNRTYERTDGHFHDFGSDFGRTYFGHDFGDFSFGWERTFSDFGDFADWERTFGHDFGHFGH230} \end{align} Summing \eqref{SDHNRTYERTDGHFHDFGSDFGRTYFGHDFGDFSFGWERTFSDFGDFADWERTFGHDFGHFGH231} over $|\alpha| = j+1$ and using \eqref{SDHNRTYERTDGHFHDFGSDFGRTYFGHDFGDFSFGWERTFSDFGDFADWERTFGHDFGHFGH229}--\eqref{SDHNRTYERTDGHFHDFGSDFGRTYFGHDFGDFSFGWERTFSDFGDFADWERTFGHDFGHFGH230}, we obtain \begin{align} \begin{split} \Vert \eprtudfsigjsdkfgjsdfg_{x}^{j+1} u \Vert_{L^2}^2 & \leq C \sum_{|\alpha| = j+1}  \Vert \eprtudfsigjsdkfgjsdfg^{\alpha_{-}} u \Vert_{L^2}^{1/2}  \Vert \nabla \eprtudfsigjsdkfgjsdfg^{\alpha_{-}} u \Vert_{L^2}^{1/2}  \Vert \eprtudfsigjsdkfgjsdfg^{\alpha} u \Vert_{L^2}^{1/2}  \Vert \nabla \eprtudfsigjsdkfgjsdfg^{\alpha} u \Vert_{L^2}^{1/2} + \sum_{|\alpha| = j+1} \Vert \eprtudfsigjsdkfgjsdfg^{\alpha_{-}} u \Vert_{L^2}  \Vert \eprtudfsigjsdkfgjsdfg^{\alpha_{+}} u \Vert_{L^2} \\ & \leq C\Vert \eprtudfsigjsdkfgjsdfg_{x}^j u \Vert_{L^2}^{1/2} \Vert \eprtudfsigjsdkfgjsdfg_{x}^{j+1} u \Vert_{L^2} \Vert \eprtudfsigjsdkfgjsdfg_{x}^{j+2} u \Vert_{L^2}^{1/2} + C\Vert \eprtudfsigjsdkfgjsdfg_{x}^j u \Vert_{L^2} \Vert \eprtudfsigjsdkfgjsdfg_{x}^{j+2} u \Vert_{L^2} . \end{split}    \llabel{oo6xQ3SJZlHPlpJmWaI9Gnouxhms0FFYOekUe45v50iD6E0b6SKde3z6MPE2e5HTHswZgbxuaEtdSpNDkM9H7rqtV7pPuMCjAz4ep33DiUKte6TbQ0zKosGYfQZfHL32d6ttwctf6erXgnauYjRkdJD0DS9kyXCpDg3mIXvHjkryboW5i8yYm1ubQ65gWSeLPrv3iv0o9ioWBpGgxr9wAj4XYYFcCcQyjnWGXr8bh2z28gXq1IC5PPzMuxLRlG2gVdIFXc77P8Sj6bEyteAZ1V3pXgXLhM4Y9cnuk8DYhn9LsPFOJrD9Zvr9evJnqS8Un9Bqy6exFGa12BvaSvhfYowVAw5NTM1W9KaIC6ZsoHCQvpGM7eGB5Vtp4RIGBfYiCH8s4Tg29MX0Bh5XomESjnYLZLfEHuzueADNDSPbTuF57zpdaFlyN1u73A20KAQvMgfjgNkezAxDz9SAVc9wiwMQiD66xHJQDJroYipvIxiRQjT006B3H6RkI7tbHaT52HFO7VT1TpsD2U2yAiyN0Z6V9hQWcIlJqlTRbtlHbSsoYlJJMicYWC7K8YkJcvwaduLx1MbvG0pEWu3AzbhRoIkEscznk0e2QssuO11PRn84FQAjIuew3yJ9Etg9O0aDiE8mqqvwyMePqPjuTa1Dl4rAJOc039CqDwcIVaPfzK0idQFferzOIAED9ybgiq71ZEGsW46j2IkJCIZcnpPwLeHtFHrtXQUZHVnp21potXVUViymcdbdZXVvvKMi2gJ9ZhRsXR1p9oICEMU0rCcdsYrACZR2thZeSDHNRTYERTDGHFHDFGSDFGRTYFGHDFGDFSFGWERTFSDFGDFADWERTFGHDFGHFGH89} \end{align} Using the Cauchy-Schwarz inequality to absorb $\Vert\eprtudfsigjsdkfgjsdfg_{x}^{j+1}u\Vert_{L^2}$ from the first factor and then using the induction hypothesis for the case $j$, we obtain \begin{align} \begin{split} \Vert \eprtudfsigjsdkfgjsdfg_{x}^{j+1} u \Vert_{L^2}^2 & \leq C \Vert \eprtudfsigjsdkfgjsdfg_{x}^j u \Vert_{L^2} \Vert \eprtudfsigjsdkfgjsdfg_{x}^{j+2} u \Vert_{L^2} \leq C C_1^j \Vert u \Vert_{L^2}^{1/(j+1)}  \Vert \eprtudfsigjsdkfgjsdfg_{x}^{j+1} u \Vert_{L^2}^{j/(j+1)} \Vert \eprtudfsigjsdkfgjsdfg_{x}^{j+2} u \Vert_{L^2} , \end{split}    \llabel{BF72obwxcCXmFLGNodosE5btAuYUIapWSFafYFE8YM70jkK8DbRnCZdMNFwtfFFRzej2AawAKtsPJ9MtskxbZSWbZXG2wjyMQw6dHMwtOkR6eLbYuIZeahyfVDDX6zZBgTSnY3npeknXH1O2A7PxNUrWT62ypaJFM5jp2mOnPnSE14cueyJwjdvmUUg8YdgO50JWttI2wVCSWq2sXNux8GLNh3g573bvXnCCyUXxGMsAtmsFDAFuveOQBn1CZNYnQ29uHGLd3BY3qekCCzgg5P50j9b1hhQy627jU6niwi7yWB53cJqRiST4tEyr16YjC7tW3Oos9MSxuwn6Nv8q9it0g9se9yUwjL3ceFhVhiLmZV1ijFbewnyziUjge7jTjWjA1txNA22jRaDjenZ9zmfiFQlYw6yXu9381mppYfUTpzTlkpm7IedFg3KIBmlonV50rz2nQXGVKD92hJEh7iV1YF92Ag4Z40SanXBh4JIMBQoazlDgnlFlh3VfAhpdRk9lKnkT9bMY9nSfF11dw2gXsyK93lCFBTA97P0tX0sZJwWrkLwbWvA66OVG1YifxH1uYXrv6pNCtKX0hjLSFfnBp7WZlVSGOfwlQJN1N97EaQqw9khRirFGnOvC496pcfsJ7JKzUSaiN0WBRHSOtVMDIaNQnTdNlRTAMdFV7xODZ9HqW8rB7kOaM4CRPCW3FJUEUToZhazPAowJbjhJobkojLqBPDAzMyTmUdmsdjyQSQ0sBh7WhA6ifhjdnJ7RFc3aL3oWf7ZhSURFaplLj1dizgbmqaOCSDHNRTYERTDGHFHDFGSDFGRTYFGHDFGDFSFGWERTFSDFGDFADWERTFGHDFGHFGH94} \end{align} from where we arrive at \begin{align} \begin{split} \Vert \eprtudfsigjsdkfgjsdfg_{x}^{j+1} u \Vert_{L^2} & \leq C^{(j+1)/(j+2)} C_1^{j(j+1)/(j+2)} \Vert u \Vert_{L^2}^{1/(j+2)} \Vert \eprtudfsigjsdkfgjsdfg_{x}^{j+2} u \Vert_{L^2}^{(j+1)/(j+2)} \\& \leq C_1^{j+1} \Vert u \Vert_{L^2}^{1/(j+2)} \Vert \eprtudfsigjsdkfgjsdfg_{x}^{j+2} u \Vert_{L^2}^{(j+1)/(j+2)} , \llabel{5qO1gwB6fYWREmh6EusXKMDdRqgbsTDqWE3usJb3BtjyeL9eoTVFJmEyZ64CUP3wXSjP4ah5lL8k62vSbGG6uCUpvKTYhsadJfm3sReIsdsHcOCUP3gHLxRmk3cmgeCORhCAw9QaaNWqvNyM3b2fgKN2O0VgUGUBds8Mq8dvqOlnHxaZIz0fqnGIXFFwynYp4xUz2MruIFaJTdtFBKMWYXZmAKwmkIJvrBj387HcXfqs9m8jQDCuvtoqIT11mYgiLNtBjcaNMvKi5qeKNblMrm4VpiyafU176NwyjxqRe3OEV7989uPV6pykdrJWcJTyNrac0hpr8NFlZR4cEO4JGqODL0KSA3J563kw4cfKJjc02ZlcWQCe96grnNj23z0CnpET7tKoPpAy2faNw13QSR3DUpsGyd2QKfHIh6DLorpU4tldL8pDin8RWR86wfeJ7jbUYQNMCvyTtyqmCQfSvEBURtB1P1tGvxRObO152ujxQ2x62wK9jlGNnR9COkeXi9DO0TPeiasnfO5WnEEyIXsmW4BkMKKIYeZpkYraRI9b6JIskSbxcxqa35JKpRQgcGRYy224xdbvD1m0XdbDZbQLUpb2EkMhQon5SWdrrOrK9D9dbSXEhsdlZrZ2ZhHGdP9WTLS52Sm9EUts59qFxX2twhNZhqzhq46bgrxloeb4sSpU1xCoyCCc1ER7SCXuGNTCRH1ER5mwlQ76EbV21tUGFm0VATYXbi9ezVyqQtfy6iy92Z5kq0Z8WiZhxvK1h611cQ8Ehji0CA30Hq5uVTQVeifNgiWfSDHNRTYERTDGHFHDFGSDFGRTYFGHDFGDFSFGWERTFSDFGDFADWERTFGHDFGHFGH232} \end{split} \end{align} by taking $C_1 \geq C$. Therefore, \eqref{SDHNRTYERTDGHFHDFGSDFGRTYFGHDFGDFSFGWERTFSDFGDFADWERTFGHDFGHFGH228} is proven for $j+1$ for some sufficiently large  constant $C_1 >1$. \end{proof} \par Next we prove the second main result on convergence of the Mach limit.  \par \begin{proof}[Proof of Theorem~\ref{T02}] Denote the spatial analyticity radius \begin{align} \delta = \frac{\kappa \tau(0)}{C_0},  \llabel{vRu8jt85LrIxcUMxg3i7u5amTLxScQ0GfKjWV89OWeuwbnpXC0WlZmg7rgnK65JXVckuK1PTLfbxK8cOTgP4kOTVfGOZWHpXrhRnV3Vh81I08FJx2pNiAE14IS1VN1OTPJYNcRP6g5svQbHcgg34tjx86NWCV477R2ynzILTD8kQLKuEPZFxRCi0h3XHCXsw9eSbWagGHVbvfaBH1MiFtg232mxVhbxh9A3FIdJfCViRAccom54uJpPko0pxMuod9XoPedVGSzluukRpg5blSNEeCXTi71OZsHnnnCY2r1wRwQwPBHcHG6d3PqIC1eLr0xQhXbWLEz54LeTBw83OX3Bm6qWpdqPqfaB5QXMLTflrmUVXLwizx4YsYDGr5qFsHetvIdw61f7kFABBspoHrtTcknv5oUauttyTZvsWa5ZoVUDvPrAZgo0pm70qqDmkbNJvQA4imEuqlbCifrfxv8mKxG13OyYEYSOTVQDKmyMoNgbNeDV9nKjuBB1iGOnpHJ70lMcnBZh8PxxLvUbVzmrOuWQliTMAxSnZcWXDyNc0x6IOOxhUCkgCiTH1rUbPiOLmF84lFuFLVXDddgqRcO33Jr4n9a2KNMl8cluomBbP44AXHzw5mKWPfo2LoslbcXJVPQ7v5QRmhzVZMefgoAjEyxn3NtECysUjwnj0ZKfVNKCm8YWb8UXji6tap8f018lqygV1idPL43gYdmeLdhPOJ5eCtsPdvNZkuYMtspEam1wqycdiRx6OsIePbtLeoR4XfeFWPWPn3mx554xOWlD1gEfQj1zTSDHNRTYERTDGHFHDFGSDFGRTYFGHDFGDFSFGWERTFSDFGDFADWERTFGHDFGHFGH205} \end{align} where $\tau(0)$, $\kappa \in(0,1]$ are fixed constants as in Theorem~\ref{T01} and $C_0\geq 2$ is a sufficiently large constant to be determined below. It is proved in Theorem~\ref{T01} that if the initial data $(p_0^\epsilon, v_0^\epsilon, S_0^\epsilon)$ satisfies \eqref{SDHNRTYERTDGHFHDFGSDFGRTYFGHDFGDFSFGWERTFSDFGDFADWERTFGHDFGHFGH109}--\eqref{SDHNRTYERTDGHFHDFGSDFGRTYFGHDFGDFSFGWERTFSDFGDFADWERTFGHDFGHFGH23} for some constants $M_0$, $\tau_0>0$, as well as the compatibility condition of all orders, then \begin{align} \Vert (p^\epsilon, v^\epsilon, S^\epsilon)(t) \Vert_{X(\delta)}  \les  M\kappa^{-2} \comma \epsilon \in (0,\epsilon_0] \commaone t\in [0,T_0],    \llabel{OAthIhD9pviyhCcDsnyPcvwhAdksAVSgefN4kQVPbwZEUgjE46jfYKgMprkBhjXh2nPkL2Eu2N0DB0OfQpwyk8ZEAW0RC40FhP70sshNFqXZGzasQhiiVmbPwhXgHI1VQYzGBXmQdlpSuhe0j8PFGMXKqeu6y5ZObG0SAEeVxavrD7Agawe0praTtfuZIqrwUgqzBSODq3cbrK0S6QxA4fGEwwCllz6Fozqfsq0bDe0lBq29Nb4jyA0PYxIW4No0f9fgD5tYAAkrjISrHseV6Ybntr7l9dwnxTF4XD3wpPgHjBxuAW2jQlFd5Rl9H0mXCvBBOlvuT6jxQPbv2svdkNXmwdKJbZ9fPRoAX4CyZNAOanf2miuAhYAjT8D430zXnRWIRHs7tClGeHCu48zin3r6Am4Jd4b8J8pAW3vAtYeKbl0Le0zuiwxc384oC1Keyjf0RHSXC2UH7AmrbirQvOaxlTvlssI1a9WRSagmneWfLEes9pSS9E7J0NSqFelUIY0FEmHizVpLJHBcN9GsyXfTxslP8Zwcor2KTBYrTRPp0mbHaa6cdDqEPkJWCOoYjB1cDFJj6BLwYhzXr0EiN4WcqzysJnwo57PCp6YaqCBa7xRIPCUrBiCi4bwgJDmI5Nj6CiwPMlxJfSCIMbIsDt3FR9QcnkAuJtwJxTFcUd2QEZb3nWPTpLl0KboxlcaO3VJW3TRUMaDY5BFT7QcotGbmKuvlGHJIRHOz3yVaCEjRjf7Mytqx7t6dDxxYwR5RqcmGTG1UgLQnkNCrcgc0hTEKO5DlNihkSDHNRTYERTDGHFHDFGSDFGRTYFGHDFGDFSFGWERTFSDFGDFADWERTFGHDFGHFGH110} \end{align} for some parameters $\kappa$, $\tau(0)$, $\epsilon_0$, and $T_0>0$. \par We proceed differently than in \cite[Section~7]{JKL}.  For the sake of contradiction,  we assume that $(v^\epsilon, p^\epsilon, S^\epsilon)$  does not converge to  $(v^{(\rm inc)}, 0, S^{(\rm inc)})$ in $L^{2} ([0,T], X(\delta))$.  Then there exists a sequence $\{\epsilon_n\} \to 0$ such that $\{(v^{\epsilon_n}, p^{\epsilon_n}, S^{\epsilon_n})\}$ does not converge to $(v^{(\rm inc)}, 0, S^{(\rm inc)})$ in $L^2 ([0,T], X(\delta))$. Recall from \cite[Theorem~1.2]{A05} that $(v^{\epsilon_n}, p^{\epsilon_n}, S^{\epsilon_n})$  converges to $(v^{(\rm inc)}, 0, S^{(\rm inc)})$ in $L^2([0,T], L^2(\Omega))$ as $\epsilon_n \to 0$.  We define $v_{kn} (t) = v^{\epsilon_k}(t) - v^{\epsilon_n} (t)$, for $k,n\in \mathbb{N}$.  Using Lemma~\ref{L23}, we obtain   \begin{align}   \begin{split}   \phi_{kn}(t)   =   \sum_{m=4}^\infty \sum_{|\alpha| = m}   \frac{\delta^{m-3}}{(m-3)!}   \Vert\eprtudfsigjsdkfgjsdfg^\alpha v_{kn} \Vert_{L_x^2}   \leq   \sum_{m=4}^\infty   \frac{C_1^m \delta^{m-3}}{(m-3)!}   \Vert v_{kn} \Vert_{L_x^2}^{1/(m+1)}   \Vert \eprtudfsigjsdkfgjsdfg_{x}^{m+1} v_{kn} \Vert_{L_x^2}^{m/(m+1)}   ,   \label{SDHNRTYERTDGHFHDFGSDFGRTYFGHDFGDFSFGWERTFSDFGDFADWERTFGHDFGHFGH235}   \end{split} \end{align} and thus, using the Minkowski and H\"older inequalities   \begin{align}    \begin{split}    \Vert \phi_{kn}(t)\Vert_{L_t^2}    \les    \sum_{m=4}^\infty    \frac{C_1^m \delta^{m-3}}{(m-3)!}    \Vert v_{kn} \Vert_{L_{x,t}^2}^{1/(m+1)}    \Vert \eprtudfsigjsdkfgjsdfg_{x}^{m+1} v_{kn} \Vert_{L_{x,t}^2}^{m/(m+1)}    ,    \end{split}    \label{SDHNRTYERTDGHFHDFGSDFGRTYFGHDFGDFSFGWERTFSDFGDFADWERTFGHDFGHFGH120}   \end{align} where the space and time domains are understood to be  $\Omega$ and $[0,T]$, respectively. To show convergence to zero, we shall apply the discrete dominated convergence theorem. To get a uniform bound, we use the discrete Young inequality, we get \begin{align} \begin{split} & \sum_{m=4}^\infty  \frac{(C_1\delta)^{m-3}}{(m-3)!} \Vert \eprtudfsigjsdkfgjsdfg_{x}^{m+1} v_{kn} \Vert_{L_{x,t}^2}^{m/(m+1)} \\ &\indeq = \sum_{m=4}^\infty  \left( \frac{(C_1\delta)^{m-3}}{(m-3)!} \Vert \eprtudfsigjsdkfgjsdfg_{x}^{m+1} v_{kn} \Vert_{L_{x,t}^2} \right)^{m/(m+1)} \left( \frac{(C_1\delta)^{m-3}}{(m-3)!} \right)^{1/(m+1)} \\ &\indeq \leq \sum_{m=4}^\infty  \frac{m(C_1\delta)^{m-3}}{(m+1)(m-3)!} \Vert \eprtudfsigjsdkfgjsdfg_{x}^{m+1} v_{kn} \Vert_{L_{x,t}^2} + \sum_{m=4}^\infty  \frac{(C_1\delta)^{m-3}}{(m+1)(m-3)!} . \end{split}    \llabel{G4gHaNERiFMktmRwNMlkjjUVRULqT1imwQMZrVoe0q6JMwACyQr0sJKJkXmKDJdD5rk3g6QSRonXTxOJOwCB7xi2qPtJp2g4hyupZ18KBRCvk0Pl4A3DkJC9O1uwO7FdxZP23MIHHdEY0dy4cbsQ0CYkNYruoqO48XwVY4fa8YG22fZydwrk0SumBo1ZAWTu6sjtTgyigZRG8e1PD1qQMroRAKnX2A4FyAhmPhnmjyrx82UUh9L9rXqN9mYwcwByy3U9HksxcpPfRtJTP6MkFNV11ASQv0sQesPGq0QUqs3aEdSMIIiOkF0H2BW0NE0zs4IG3DuyeVCMJYksDWbIdKwFZl1xqKeyFxrnaR6kbyQDsPCUXMFosezxIVH8fTNprVlT7kajWKg5NBiP2mmbDIAyVG79Ekp2w2XVjQ44V6RD9w9CIC7QB1ZJEQEXseJUNweZYKOUOJ7Rn3WVEZnLbgtCeqJVD6Cl3jOXAKqEwnZPZl7Uxx2JuKxqkROsLyqg8O9csoG2F620l2OqHXv18rKuij89Cqf94GDjFJp2f6OO53TduXXyGQTVfBWnRrXbrXYcu9tio0pO7W9nOGtWSgjJKYfTxpYRtB9FuFXOOmO9rMmwEq9xVBGfpBLhEfGgib8prW1E4ETHJVIDHAPwUtX29Yvy4qIQwJxRaQnpgUu42IX0RX3oBQdY2TnL50jOVAFMIXFXcXThlqYv99z3QeITPuQnPl50gPB1uTUZwEgOyZ34iWUgjAtQ2S5tRpoUlQzRZDbSzBhJ7R9rE2DahWTLgVCBYKVgSDHNRTYERTDGHFHDFGSDFGRTYFGHDFGDFSFGWERTFSDFGDFADWERTFGHDFGHFGH116} \end{align} Now, choose $\delta = \kappa \tau(0)/ C_1 C_2$, where $C_1$ is the constant from Lemma~\ref{L23} and $C_2 \geq 1$. We obtain   \begin{align}   \sum_{m=4}^\infty    \frac{(C_1\delta)^{m-3}}{(m-3)!}   \Vert \eprtudfsigjsdkfgjsdfg_{x}^{m+1} v_{kn} \Vert_{L_{x,t}^2}^{m/(m+1)}   \leq   C M \kappa^{-2} \delta^{-1} C_1^{-1}   +   C   .   \label{SDHNRTYERTDGHFHDFGSDFGRTYFGHDFGDFSFGWERTFSDFGDFADWERTFGHDFGHFGH234}   \end{align} Thus the discrete dominated convergence theorem applies, and it  thus follows from \eqref{SDHNRTYERTDGHFHDFGSDFGRTYFGHDFGDFSFGWERTFSDFGDFADWERTFGHDFGHFGH235}--\eqref{SDHNRTYERTDGHFHDFGSDFGRTYFGHDFGDFSFGWERTFSDFGDFADWERTFGHDFGHFGH234} that   \begin{align}   \sum_{m=4}^\infty \sum_{|\alpha| = m}   \frac{\delta^{(m-3)_+}}{(m-3)!}   \Vert\eprtudfsigjsdkfgjsdfg^\alpha v_{kn} \Vert_{L_{x,t}^2}   \to   0   \inon{as $k,n\to \infty$}   \label{SDHNRTYERTDGHFHDFGSDFGRTYFGHDFGDFSFGWERTFSDFGDFADWERTFGHDFGHFGH237}   \end{align} since $\Vert v_{kn}(t) \Vert_{L_{x,t}^2} \to 0$ as $k,n \to \infty$. Note that from \eqref{SDHNRTYERTDGHFHDFGSDFGRTYFGHDFGDFSFGWERTFSDFGDFADWERTFGHDFGHFGH24} we have \begin{align} \Vert v_{kn} \Vert_{L^2_t H^3}^2  \leq  C \Vert v_{kn} \Vert_{H^4}^{3/2} \Vert v_{kn} \Vert_{L_{x,t}^2}^{1/2} + C \Vert v_{kn} \Vert_{L_{x,t}^2}^2  \leq  C\Vert v_{kn} \Vert_{L_{x,t}^2}^{1/2} + C \Vert v_{kn} \Vert_{L_{x,t}^2}^2   \to 0 \inon{as $k,n\to \infty$} . \label{SDHNRTYERTDGHFHDFGSDFGRTYFGHDFGDFSFGWERTFSDFGDFADWERTFGHDFGHFGH236} \end{align} From \eqref{SDHNRTYERTDGHFHDFGSDFGRTYFGHDFGDFSFGWERTFSDFGDFADWERTFGHDFGHFGH237}--\eqref{SDHNRTYERTDGHFHDFGSDFGRTYFGHDFGDFSFGWERTFSDFGDFADWERTFGHDFGHFGH236} and analogous inequalities for $p_{kn}$ and $S_{kn}$, we infer that the sequence $\{(v^{\epsilon_n}, p^{\epsilon_n}, S^{\epsilon_n})\}$ is Cauchy in $L^2([0,T_0], X(\delta))$,  which along with  $(v^\epsilon, p^\epsilon, S^\epsilon)\to (v^{(\rm inc)}, 0, S^{(\rm inc)})$ in $L^2([0,T],L^2(\Omega))$ leads to a contradiction.  Therefore, $\{(v^{\epsilon}, p^{\epsilon}, S^{\epsilon})\}$ is convergent and thus goes to $(v^{(\rm{inc})}, 0, S^{(\rm{inc})})$ in $L^2([0,T_0], X(\delta))$ as $\epsilon \to 0$. \end{proof} \par \section*{Acknowledgments} JJ was supported in part by the NSF Grant DMS-2009458 and WiSE Program at the University of Southern California, IK was supported in part the NSF grant DMS-1907992, while LL was supported in part by the NSF grants DMS-2009458 and DMS-1907992. The work was completed while the authors participated in the MSRI program during the Spring~2021 semester (NSF DMS-1928930). \par  \par \end{document}